%% file: stacky-covers.tex
\documentclass[a4paper]{amsart}

\calclayout

\usepackage[utf8]{inputenc}

\usepackage[english, swedish]{babel}

\usepackage{amsmath}
\usepackage{graphicx}
\usepackage{amsfonts}
\usepackage{siunitx}
\usepackage{amsthm}
\usepackage{mathtools}
\usepackage{tensor}
\usepackage{bbm}
\usepackage{float}
\usepackage{rotating}
\usepackage{amssymb}
\usepackage[section]{placeins}
\usepackage{ stmaryrd }
\usepackage{ marvosym }
\usepackage{mathrsfs}
\usepackage{tikz}
\usepackage{tikz-cd}
\usepackage[all,cmtip]{xy}
\usepackage{hyperref}

\tikzset{shorten <>/.style={shorten >=#1,shorten <=#1}}

\newtheorem{thm}{Theorem}
\newtheorem{innercustomthm}{Theorem}
\newenvironment{customthm}[1]
  {\renewcommand\theinnercustomthm{#1}\innercustomthm}
  {\endinnercustomthm}
\numberwithin{thm}{section}

\newtheorem{prop}[thm]{Proposition}
\newtheorem{lem}[thm]{Lemma}
\newtheorem{cor}[thm]{Corollary}
\theoremstyle{definition}
\newtheorem{df}[thm]{Definition}
\newtheorem{ex}[thm]{Example}
\theoremstyle{remark}
\newtheorem{rk}[thm]{Remark}

\DeclareMathOperator{\Pic}{Pic}

\DeclareMathOperator{\Ann}{Ann}

\DeclareMathOperator{\ExtAP}{Ext_f(A,P)}
\DeclareMathOperator{\ExtAM}{Ext_f(\A,\M)}

\DeclareMathOperator{\CAP}{Z^2(A,P)}
\DeclareMathOperator{\CAM}{Z^2(\A,\M)}
\DeclareMathOperator{\HAM}{H^2(\A,\M)}
\DeclareMathOperator{\CAMS}{\mathcal{Z}^2(\A,\M)}

\DeclareMathOperator{\Fit}{Fitt}
\DeclareMathOperator{\Map}{Map}

\DeclareMathOperator{\Tr}{Tr}

\DeclareMathOperator{\HOM}{\mathscr{H}\text{\normalfont \kern -2.5pt {\emph{om}}}}
\DeclareMathOperator{\AUT}{\mathscr{A}\text{\normalfont \kern -2.5pt {\emph{ut}}}}

\DeclareMathOperator{\END}{\mathcal{E}\text{\kern -1.5pt {\emph{nd}}}}

\DeclareMathOperator{\DIV}{\mathscr{D}\text{\normalfont \kern -1pt {\emph{iv}}}}
\DeclareMathOperator{\WDIV}{\mathscr{WD}\text{\normalfont \kern -1pt {\emph{iv}}}}
\DeclareMathOperator{\PIC}{\mathscr{P}\text{\normalfont \kern -1pt {\emph{ic}}}}

\DeclareMathOperator{\PAR}{\mathscr{P}\text{\kern -1pt {\emph{ar}}}}
\DeclareMathOperator{\EFVECT}{\mathscr{EFV}\text{\kern -1pt {\emph{ect}}}}

\DeclareMathOperator{\STCOV}{\mathscr{S}\text{\normalfont \kern -1pt {\emph{t}}}\mathscr{C}\text{\normalfont \kern -1pt {\emph{ov}}}}
\DeclareMathOperator{\STDATA}{\mathscr{S}\text{\normalfont \kern -1pt {\emph{t}}}\mathscr{D}\text{\normalfont \kern -1pt {\emph{ata}}}}

\DeclareMathOperator{\QCOH}{\mathscr{QC}\text{\normalfont \kern -1pt {\emph{oh}}}}

\DeclareMathOperator{\Div}{Div}
\DeclareMathOperator{\EFVect}{EFVect}
\DeclareMathOperator{\Vect}{Vect}
\DeclareMathOperator{\WDiv}{WDiv}

\DeclareMathOperator{\NGG}{\mathbb{N}^{\GG}}

\DeclareMathOperator{\NGxGo}{\mathbb{N}^A/\langle e_0\rangle\times\mathbb{N}^A/\langle e_0\rangle}

\DeclareMathOperator{\ord}{\normalfont{\text{ord}}}

\newcommand{\mmu}{\pmb{\mu}}
\newcommand{\llambda}{\pmb{\lambda}}
\newcommand{\aalpha}{\pmb{\alpha}}
\newcommand{\rrho}{\pmb{\rho}}
\newcommand{\ggamma}{\pmb{\gamma}}

\newcommand{\PA}{\mathscr{P}_{\mathcal{A}}}
\newcommand{\QA}{\mathscr{Q}_{\mathcal{A}}}

\newcommand{\id}{\operatorname{id}}
\newcommand{\Hom}{\normalfont\text{Hom}}

\newcommand{\uuu}{\boldsymbol{\normalfont{u}}}
\newcommand{\sss}{\boldsymbol{\normalfont{s}}}
\newcommand{\aaa}{\boldsymbol{\normalfont{a}}}
\newcommand{\QQ}{\mathbb{Q}}
\newcommand{\GG}{A\times A}

\newcommand{\F}{\mathcal{F}}
\newcommand{\Gm}{\mathbb{G}}
\newcommand{\GmS}{\mathbb{G}_{m,S}}

\newcommand{\cyc}{\mbox{cyc}}

\newcommand{\Aut}{\normalfont{\mathsf{Aut}}}
\newcommand{\QCoh}{\normalfont{\mathsf{QCoh}}}

\newcommand{\X}{\mathscr{X}}

\newcommand{\Y}{\mathscr{Y}}

\newcommand{\NN}{\mathbb{N}}
\newcommand{\ZZ}{\mathbb{Z}}
\newcommand{\FF}{\mathbb{F}}
\newcommand{\J}{\mathcal{J}}

\newcommand{\I}{\mathcal{I}}
\newcommand{\Hj}{\mathcal{H}}

\newcommand{\GI}{{\mathcal{A}_{\mathscr{X}}}}

\newcommand{\A}{\mathcal{A}}
\newcommand{\PP}{\mathcal{P}}
\newcommand{\Q}{\mathcal{Q}}

\newcommand{\Li}{\mathcal{L}}
\newcommand{\Ei}{\mathcal{E}}

\newcommand{\SEt}{S_{{\normalfont\textsf{ét}}}}

\newcommand{\lala}{{\lambda, \lambda'}}

\newcommand{\et}{{\normalfont \textsf{ét}}}
\newcommand{\op}{{\normalfont \textsf{op}}}

\newcommand{\oj}{\mathcal{O}}

\newcommand{\C}{\mathcal{C}}

\newcommand{\Cov}{\mbox{D-}\mathscr{C}\kern -1.5pt ov}
\newcommand{\Ab}{\normalfont( \textsf{Ab})}

\newcommand{\Sch}{{\normalfont\textsf{Sch}}}

\newcommand{\m}{\mathfrak{m}}

\newcommand{\spec}{{\normalfont\text{Spec}}\, }

\newcommand{\stab}{{\normalfont\text{Stab}}}

\newcommand{\proj}{{\rm Proj}\,}
\newcommand{\M}{\mathcal{M}}
\newcommand{\N}{\mathbb{N}}
\newcommand{\CC}{\mathbb{C}}
\newcommand{\PPP}{\mathbb{P}}
\newcommand{\Af}{\mathbb{A}}

\let\amsamp=&

\usepackage{amsmath}

\begingroup
\catcode`\&=13
\gdef\pampmatrix{%
  \begingroup
  \let&=\amsamp
  \begin{pmatrix}%
}
\gdef\endpampmatrix{\end{pmatrix}\endgroup}
\endgroup

\makeatletter
\newcommand{\pushright}[1]{\ifmeasuring@#1\else\omit\hfill$\displaystyle#1$\fi\ignorespaces}
\newcommand{\pushleft}[1]{\ifmeasuring@#1\else\omit$\displaystylonele#1$\hfill\fi\ignorespaces}
\makeatother

\title{{Building Data for Stacky Covers}}
\author{Eric Ahlqvist}
\address{School of Mathematics, University of Edinburgh, James Clerk Maxwell Building, Peter Guthrie Tait Road, Edinburgh, EH9 3FD, United Kingdom}
\email{p.e.l.ahlqvist@gmail.com, eric.ahlqvist@ed.ac.uk}

\begin{document}
\selectlanguage{english}
\input{abstract.tex}

\thanks{The author was supported by the Swedish Research Council 2015-05554 and the Knut and Alice Wallenberg Foundation 2021.0279.}

\maketitle

\setcounter{tocdepth}{1}
\tableofcontents

\input{intro}

\input{covers}

\input{DF-structures}

\input{special-DF-data}

\input{stacky-2-cocycles}

\input{DF-data-from-covers}

\input{stacky-covers-as-root-stacks}

\input{building-data}

\input{application}

\appendix

\input{groups}

\bibliographystyle{dary.bst}
\bibliography{references}

\end{document}

%% file: abstract.tex
\begin{abstract}
We define \emph{stacky building data} for \emph{stacky covers} in the spirit of Pardini and give an equivalence of (2,1)-categories between the category of stacky covers and the category of stacky building data. We show that every stacky cover is a flat root stack in the sense of Olsson and Borne--Vistoli and give an intrinsic description of it as a root stack using stacky building data.
When the base scheme $S$ is defined over a field, we give a criterion for when a \emph{birational} building datum comes from a tamely ramified cover for a finite abelian group scheme, generalizing a result of Biswas--Borne.
\end{abstract}

%% file: intro.tex
\section{Introduction}
The class of \emph{stacky covers} contains flat (classical) root stacks and flat stacky modifications in the sense of \cite{Rydh-comp}.
Root stacks first appeared in \cite{Matsuki--Olsson}, \cite{Abramovich--Graber--Vistoli}, and \cite{Cadman}. It was used by Abramovich--Graber--Vistoli in \cite{Abramovich--Graber--Vistoli} to define Gromov--Witten theory of Deligne--Mumford stacks and by Cadman--Chen \cite{Cadman--Chen} when counting rational plane curves tangent to a smooth cubic. Root stacks may also be used in birational geometry. For instance, Matsuki--Olsson used root stacks in the logarithmic setting to interpret the Kawamata--Viehweg vanishing theorem as an application of Kodaira vanishing for stacks \cite{Matsuki--Olsson}. Root stacks and stacky modifications where also used by Rydh in \cite{Rydh-comp} to prove compactification results for tame Deligne--Mumford stacks and by Bergh in \cite{Bergh} when constructing a functorial destackification algorithm for tame stacks with diagonalizable stabilizers. Bergh--Rydh also extended the latter result to remove the assumption that stabilizers are diagonalizable \cite{Bergh--Rydh}. The aim of this paper is to shed more light on these constructions in the flat case. We will do so by classifying stacky covers in terms of \emph{stacky building data} à la Pardini \cite{Pardini}.

A \emph{stacky cover} $\pi\colon \X\to S$ of a scheme $S$ consists of a tame stack $\X$ which has finite diagonalizable stabilizers at geometric points, together with a morphism $\pi\colon \X\to S$ which is
\begin{enumerate}
  \item flat, proper, of finite presentation,
  \item a coarse moduli space, and
  \item for any morphism of schemes $T\to S$, the base change $\pi|_T\colon \X_T\to T$ has the property that $(\pi|_T)_*$ takes invertible sheaves to invertible sheaves.
\end{enumerate}
For instance, if $\X\to S$ is a \emph{flat stacky modification}, that is, flat, proper, locally of finite presentation, and birational with finite diagonalizable stabilizers, then $\X\to S$ is a stacky cover. We prove a classification of stacky covers in terms of \emph{stacky building data}. To a stacky cover $\pi\colon \X\to S$ we associate an \'etale sheaf of abelian groups $\A$ over $S$ and construct a \emph{2-cocycle}
  \[
    f_\X\colon \A\times\A\to \DIV_S:=[\Af^1_S/\Gm_{m,S}]\,.
  \]
We show that $\X$ may be thought of as the stack parametrizing \emph{1-cochains with boundary $f_\X$}.

From $\A$ we construct two quasi-fine \'etale sheaves of monoids $P_\A$, $Q_\A$, and a flat Kummer homomorphism $\gamma_\A\colon P_\A\to Q_\A$. To the 2-cocycle $f_\X$ we associate a symmetric monoidal functor
$\Li\colon P_\A\to \DIV_{S_\et}$ where $\DIV_{S_\et}$ denotes the restriction of $[\Af^1_S/\Gm_{m,S}]$ to the small \'etale site of $S$. Hence we get a diagram
  \begin{equation}\label{eq:DF-datum}
    \begin{tikzcd} 
      P_\A \ar{r}{\Li} \ar{d}[swap]{\gamma_\A} & {\DIV_{S_\et}} \\ Q_\A & 
    \end{tikzcd}
  \end{equation}
and we refer to this as a \emph{Deligne--Faltings datum}.
We denote by $S_{(\A,\Li)}$ the associated root stack.
The main results are the following:

\begin{customthm}{\ref{thm:main}}
Let $\pi\colon \X\to S$ be a stacky cover. Then there exists a canonical (up to canonical isomorphism) \emph{building datum} $(\A,\Li)$ where $\A=\Pic_{\X/S}$ is the relative picard functor, and a canonical isomorphism of stacks \[\X\to S_{(\A,\Li)}\] where $S_{(\A,\Li)}$ is the root stack associated to the building datum $(\A,\Li)$.
\end{customthm}

\begin{customthm}{\ref{thm:build}}
We have an equivalence of (2,1)-categories \[\STCOV\simeq \STDATA\] between the category of stacky covers and the category of stacky building data.
\end{customthm}

A stacky cover will \'{e}tale locally on $S$ look like a quotient stack of a ramified Galois cover $X$ for a diagonalizable group $D(A)$, where $A$ is a finite abelian group. Such covers have been studied for example in \cite{Pardini} (Galois covers that are generically torsors) and \cite{Tonini} (general setting) and can be described combinatorially by giving a line bundle $\Li_\lambda$ for each $\lambda\in A$ together with global sections $s_{\lambda,\lambda'}\in \Gamma(S,\Li_\lambda^{-1}\otimes\Li^{-1}_{\lambda'}\otimes\Li_{\lambda+\lambda'})$ corresponding to the multiplication in $\oj_X$. These data are required to satisfy the appropriate axioms to constitute an associative and commutative algebra (see Remark \ref{rk:lin-bun}). This suggests that the quotient stack
$[X/D(A)]$
can also be described in a combinatorial way using line bundles and sections, or more precisely, as a root stack. Using constructions in \cite{Tonini} we show that the group $A$ gives rise to two (constant) quasi-fine and sharp monoids $P_A$ and $Q_A$ with a flat Kummer homomorphism between them.
These sit in an exact sequence of monoids $0\to P_A\to Q_A\to A\to 0$ which is the \emph{universal free extension} of $A$.
From the data of the cover $X$ we can then construct a symmetric monoidal functor $\Li_X\colon P_A\to \DIV_{S_\et}$ such that the root stack of the diagram
\[\begin{tikzcd} P_A \ar{r}{\Li_X} \ar{d}[swap]{\gamma_A} & {\DIV_{S_\et}} \\ Q_A & \end{tikzcd}\]
(compare with Diagram (\ref{eq:DF-datum})) is isomorphic to $[X/D(A)]$.

We will use root stacks in the language of Deligne--Faltings structures as in \cite{Borne-Vistoli}. The monoids and the symmetric monoidal functor in the main theorem are constructed intrinsically on $\X$ using the relative Picard functor $\Pic_{\X/S}$. When $\X$ is Deligne--Mumford then $\pi^*\Pic_{\X/S}\cong D(\I_\X)$, where $D(\I_\X)$ is the Cartier dual of the inertia stack.

The category of quasi-coherent sheaves on a root stack $\X/S$ associated to a Deligne--Faltings datum $(P,Q,\Li)$ is equivalent to the category of \emph{parabolic sheaves} on $S$ with respect to $(P,Q,\Li)$. When $S$ is defined over a field and is geometrically reduced and geometrically connected, we give a criterion for when a \emph{birational} building datum comes from a ramified $G$-cover (Definition \ref{df:ram-cov}), for some finite abelian group scheme $G$ over $k$, generalizing the main result of \cite{Biswas-Borne}. Our theorem looks as follows:

\begin{customthm}{\ref{thm:appl-main}}
Let $S$ be a scheme proper over a field $k$ and assume that $S$ is geometrically connected and geometrically reduced. Let $(\A,\Li)$ be a birational building datum and $(P_\A,Q_\A,\Li)$ the associated Deligne--Faltings datum.
Then the following are equivalent:
\begin{enumerate}
  \item There exists a finite abelian group scheme $G$ over $k$ and a ramified $G$-cover $X\to S$ with birational building datum $(\A,\Li)$;
  \item For every geometric point $\bar{s}$ in the branch locus, we have that
  \begin{enumerate}
    \item[(i)] the map $\Gamma(S,\A)\to \A_{\bar{s}}$ is surjective, and
    \item[(ii)] for every $\lambda\in \A_{\bar{s}}$, there exists an essentially finite, basic, parabolic vector bundle $(E,\rho)$ on $(S,P_\A,Q_\A,\Li)$ such that the morphism
    \[\bigoplus_{\lambda'}E_{e_\lambda-e_{\lambda'}}|_{\bar{s}}\xrightarrow{(E(e_{\lambda'})|_{\bar{s}})_{\lambda'}} E_{e_\lambda}|_{\bar{s}}\] is not surjective, where the direct sum is over all $\lambda'\in \Gamma(S,\A)$ such that $\lambda'_{\bar{s}}\neq 0$.
  \end{enumerate}
\end{enumerate}
\end{customthm}

Suppose that there exists a finite abelian group scheme $G$ over $k$ and a ramified $G$-cover $X\to S$ with ramification datum $(\A,\Li)$ as in Theorem \ref{thm:appl-main}. If $\X=S_{(\A,\Li)}=S_{(P_\A,Q_\A,\Li)}$ is the associated root stack then it follows that $\X\simeq [X/G]$.

\subsection*{Organization}
In Section \ref{sec:cov} we study the theory of ramified Galois covers for diagonalizable group schemes. The reader who is familiar with ramified covers may skip ahead. 

In Section \ref{sec:DF} we recall the theory developed in \cite{Borne-Vistoli} involving Deligne--Faltings structures and root stacks. We also review some properties of monoids and symmetric monoidal categories. In the end we investigate what it means for a root stack to be flat in terms of the monoids defining it.

In Section \ref{sec:sp-DF-data} we define and prove statements about the universal monoids $P_A$ and $Q_A$ that we use use to model the local charts of our stacky covers.
For instance, we show that the monoids $P_A$ and $Q_A$ associated to $A$ are quasi-fine and sharp and that the action of $P_A$ on $Q_A$ is free. When $A$ is an abelian group and $P$ a monoid, we identify \emph{free extensions of $A$ by $P$} and \emph{2-cocycles of $A$ with values in $P$}. We show that there is a universal 2-cocycle $A\times A\to P_A$ corresponding to a universal free extension $0\to P_A\to Q_A\to A\to 0$.

In Section \ref{sec:2-coc-in-sym-mon} we generalize the theory developed in Section \ref{sec:sp-DF-data} to the setting where $P$ and $Q$ are replaced by symmetric monoidal categories. We define the stack parametrizing 1-cocycles with a fixed coboundary. 

In Section \ref{sec:DF-from-cov} we look at the local structure of the stacks considered in this paper and show that to every ramified $D(A)$-cover $X\to S$ we may associate a Deligne--Faltings datum $(P_A,Q_A,\Li_X)$ and that the root stack $S_{(P_A,Q_A,\Li_X)}$ is isomorphic to $[X/D(A)]$.

In Section \ref{sec:birat} we show how to realize a stacky cover as a root stack using the relative Picard functor.

In Section \ref{sec:building-data} we define \emph{stacky building data} and show that there is an equivalence of (2,1)-categories between the category of stacky covers and the category of stacky building data.

In Section \ref{sec:application} we generalize the result of Biswas--Borne and give a criterion for when building data on a scheme over a field comes from a ramified abelian cover.

In Appendix \ref{sec:groups} we study the Cartier dual $D(H)$ of a non-flat closed subgroup $H$ of a multiplicative group. 

We use this in Section \ref{sec:birat} to show that $\pi_*D(\I_\X)\cong \Pic_{\X/S}$ when $\X$ is Deligne--Mumford.

\subsection*{Notation and conventions}
The letter $S$ will always denote the base scheme which we assume to be locally Noetherian.
The letter $A$ will always denote an abelian group. If $A$ is an abelian group, $\lambda\in A$, and $B$ is an $A$-graded $R$-module, then $B[\lambda]$ is the $A$-graded $R$-module with \[B[\lambda]_{\lambda'}=B_{\lambda+\lambda'}\,.\]

When $\X$ is a stack and $f\colon T\to \X$ is an object, we write $\stab(T)$ or $\stab(f)$ for the sheaf of groups $\underline{\Aut}_f$.
The unit object of a symmetric monoidal category will be denoted by $\mathbbm{1}$.

\subsection*{Acknowledgements}
I want to thank my PhD supervisor David Rydh for many invaluable discussions and his great enthusiasm for the subject. I would like to thank Niels Borne, Magnus Carlson, Martin Olsson, Fabio Tonini, and Angelo Vistoli for useful discussions. 
I would also like to thank the anonymous referee for many good suggestions for improvement.

%% file: covers.tex
\section{Ramified covers}\label{sec:cov}

Throughout this section we will always assume that $A$ is an abelian group and $G=D(A)$ the corresponding finite diagonalizable group over the base scheme $S$. This means that \[\begin{split}G&=D_S(A)\\ &=\HOM_{grp}(A_S,\Gm_m)\\ &\cong \spec\oj_S[A]\,.\end{split}\] We also write \[\oj_S[G]=\oj_S[A]\] for the group algebra of $A$ over $\oj_S$.

\begin{df}[{{\cite[Definition 2.1]{Tonini}}}]\label{def:cover}
Let $G\to S$ be a finite flat diagonalizable group scheme of finite presentation. A \emph{G-cover over} $S$ is a finite locally free morphism $f\colon X\to S$ together with an action of $G$ such that there exists an fppf cover $\{U_i\to S\}$ and an isomorphism of $\oj_{U_i}[G]$-comodules \[(f_*\oj_X)|_{U_i}\cong \oj_{U_i}[G]\,,\] where the comodule structure on the right hand side is the regular representation.
\end{df}

\begin{rk}\label{rk:lin-bun}
This means that we have a splitting \[f_*\oj_X\cong \bigoplus_{\lambda\in A}\Li_\lambda\] where each $\Li_\lambda$ is a line bundle and $\Li_0=\oj_S$. We also have multiplication morphisms
\[\Li_\lambda\otimes \Li_{\lambda'}\to \Li_{\lambda+\lambda'}\]
which we think of as global sections \[s_{\lambda,\lambda'}\in \Li^{-1}_\lambda\otimes\Li^{-1}_{\lambda'}\otimes\Li_{\lambda+\lambda'}\,.\] These global sections will have the following properties:
\begin{enumerate}
\item $s_{0,\lambda}=1\quad$ $\forall \lambda\in A$;
\item $s_{\lambda,\lambda'}=s_{\lambda',\lambda}\quad$ $\forall \lambda,\lambda'\in A$;
\item $s_{\lambda,\lambda'}s_{\lambda+\lambda',\lambda''}=s_{\lambda',\lambda''}s_{\lambda'+\lambda'',\lambda}\quad$ $\forall \lambda,\lambda',\lambda''\in A$.
\end{enumerate}
Note that equality $a=b$ here means that the element $a$ is sent to $b$ under the canonical isomorphism of line bundles. For instance, $s_{\lambda,\lambda'}=s_{\lambda',\lambda}$ means that $s_{\lambda,\lambda'}\mapsto s_{\lambda',\lambda}$ under the canonical isomorphism $\Li^{-1}_\lambda\otimes\Li^{-1}_{\lambda'}\otimes\Li_{\lambda+\lambda'}\cong \Li^{-1}_{\lambda'}\otimes\Li^{-1}_{\lambda}\otimes\Li_{\lambda+\lambda'}$.
\end{rk}

\begin{df}
Let $S$ be a scheme. A \emph{generalized effective Cartier divisor} is a pair $(\Li,s)$ consisting of
\begin{enumerate}
  \item a line bundle $\Li$ on $S$ and
  \item a global section $s\in \Gamma(S,\Li)$.
\end{enumerate}
\end{df}

\begin{rk}
Note that each pair $(\Li^{-1}_\lambda\otimes\Li^{-1}_{\lambda'}\otimes\Li_{\lambda+\lambda'}, s_{\lala})$ forms a generalized effective Cartier divisor.
Note that the data of a generalized effective Cartier divisor $(\Li,s)$ is equivalent to the data of a morphism of stacks $S\to [\Af^1/\Gm_{m}]$.
\end{rk}

\begin{rk}
By Remark \ref{rk:lin-bun} we may replace the fppf cover in Definition \ref{def:cover} by a \emph{Zariski} cover. This is however not always possible if one allows non-diagonalizable group schemes as in \cite[Definition 2.1.2]{Tonini-thesis}.
\end{rk}

\begin{rk}
The ramification locus of a ramified cover which is generically a torsor has pure codimension~1 \cite[Theorem 6.8]{Altman-Kleiman}.
\end{rk}

\begin{rk}\label{rk:cov2}
Any finite locally free morphism $f\colon X\to S$ of rank 2 is a $\mu_2$-cover if $2$ is invertible in $\Gamma(S,\oj_S)$. Indeed, there is a trace map $T\colon f_*\oj_X\to \oj_S$ sending a section $x$ to the trace of the matrix corresponding to multiplication by $x$. The composition $\oj_S\to f_*\oj_X\to\oj_S$ is multiplication by 2 and if 2 is invertible, we get that \[\oj_S\to f_*\oj_X\xrightarrow{\frac{1}{2}T}\oj_S\] is the identity and hence
\[f_*\oj_X\cong \oj_S\oplus \Li\,,\] where $\Li=\ker T$ is a line bundle.
It remains to show that the multiplication $\Li\otimes\Li\to f_*\oj_X$ lands in $\oj_S$. This can be checked on stalks so we may assume that $\Li$ is trivial. Take $x\in\Gamma(s,\Li)$. Then multiplication by $x$ is given by a $2\times2$-matrix
\[\begin{pmatrix}
0 & b \\ x & d
\end{pmatrix}\] since the multiplication $\oj_S\otimes \Li\to f_*\oj_X$ is just the module action, and hence lands in $\Li$. But $\Li=\ker T$ and hence $d=0$. Hence we conclude that $X\to S$ is a $\mu_2$-cover.
\end{rk}


\begin{ex} Here is a list of examples of ramified covers.
\begin{enumerate}
\item The map on spectra induced by the inclusion $\ZZ\to \ZZ[x]/(x^2-2)$ is a $\mu_2$-cover when $x$ has weight 1.
\item The map on spectra induced by the inclusion $\mathbb{C}[s]\to \mathbb{C}[s,x,y]/(x^2-sy,y^2-sx,xy-s^2)$ is a $\mu_3$-cover when $x$ has weight 1 and $y$ has weight 2.
\item The map $\proj \mathbb{C}[x,y,z]/(z^2-x^2)\to \mathbb{P}^1_{\mathbb{C}}=S$ induced by the inclusion $\mathbb{C}[x,y]\to \mathbb{C}[x,y,z]/(x^2-z^2)$ is a $\mu_2$-cover when $z$ has weight 1.
We have $\proj \mathbb{C}[x,y,z]/(z^2-x^2)\cong \spec (\oj_S\oplus\oj_S(-1))$ with multiplication given by $x^2\in \Gamma(S,\oj_S(2))$.
\item In view of Remark \ref{rk:cov2}, any degree 2 finite surjective morphism of varieties $X\to S$ over an algebraically closed field, where $X$ is Cohen--Macaulay and $S$ is regular, is a $\mu_2$-cover (flatness follows from \cite[Corollary 18.17]{Eisenbud-CA}).
\item In particular, a K3 surface obtained as a double cover of $\PPP^2$ branched along a sextic is a $\mu_2$-cover.
\end{enumerate}
\end{ex}

See also Example \ref{ex:non-example} for a non-example, which could be mistaken for a $G$-cover in the sense of Definition \ref{def:cover}.

\subsection*{Covers and 2-cocycles}
\begin{df}[\cite{Patchkoria_1977_1, Patchkoria_1977_3, Patchkoria_2018}]\label{df:coc}
Let $A$ be an abelian group. A \emph{(commutative) 2-cocycle} of $A$ with values in a monoid $P$ is a function
\[
f\colon A\times A\to P
\]
that satisfies the following properties:
\begin{enumerate}
\item $f(0,\lambda)=0\quad$ $\forall \lambda\in A$;
\item $f(\lambda,\lambda')=f(\lambda',\lambda)\quad$ $\forall \lambda,\lambda'\in A$;
\item $f(\lambda,\lambda')+f(\lambda+\lambda',\lambda'')=f(\lambda',\lambda'')+f(\lambda'+\lambda'',\lambda)\quad$ $\forall \lambda,\lambda',\lambda''\in A$.
\end{enumerate}
\end{df}

\begin{rk}
Note that the set of 2-cocycles of $A\times A\to P$ form a monoid under pointwise addition.
\end{rk}

\begin{df}\label{def:cap}
We denote the monoid of 2-cocycles $A\times A\to P$ by $\CAP$.
\end{df}

\begin{rk}
There is a bijection between the set of 2-cocycles $A\times A\to \N$ and the set $\Hom(P_A,\mathbb{N})$ of \emph{rays} as in \cite[Notation 3.11]{Tonini}, where $P_A$ is the universal monoid we define in Definition \ref{df:P}. This will be explained in detail in Section \ref{sec:sp-DF-data}.
\end{rk}

Recall that if $\Li$ is a line bundle on a scheme $S$ then we have a bijection 
  \[
    \Bigg\{s\in\Gamma(S,\Li): s\mbox{ is regular}\Bigg\}\Big/\vspace{-2pt}\sim\ \to 
    \left\{
      \begin{split}
        & \mbox{effective Cartier divisors } \\ 
        & D \mbox{ such that }\oj_S(D)\cong \Li
      \end{split} 
    \right\}\,,
  \]
where $s\sim s'$ if there is a unit $u\in \Gamma(S,\oj_S^\times)$ such that $us=s'$.

Let $S$ be a \emph{normal} scheme and let $X\to S$ be a $D(A)$-cover with branch locus $B=\bigcup_{i\in I}D_i$ (union of irreducible components with $I$ finite), which is generically a torsor. Then the global sections \[s_{\lambda,\lambda'}\in\Gamma(S,\Li_\lambda^{-1}\otimes \Li_{\lambda'}^{-1}\otimes \Li_{\lambda+\lambda'})\] are all regular and correspond to divisors $D_{\lambda,\lambda'}$ such that \[\oj_S(D_{\lambda,\lambda'})\cong \Li_\lambda^{-1}\otimes \Li_{\lambda'}^{-1}\otimes \Li_{\lambda+\lambda'}\,.\]
Hence we get a 2-cocycle \[\begin{split}f_X\colon A\times A & \to \mathbb{N}^I \\ (\lambda,\lambda') & \mapsto \mbox{ord}_{D_i}(D_{\lambda,\lambda'})\,,\end{split}\] which we refer to as the \emph{2-cocycle} of the cover $X$. The cover $X$ is determined by $f_X$ (see Proposition \ref{prop:structure}).

If $X=X_1\, \tensor*[^{}_{\varphi_1}]{\vee}{^{}_{\varphi_2}} X_2$ then \[f_X(\lambda,\lambda')=f_{X_1}(\varphi_1(\lambda),\varphi_1(\lambda'))+f_{X_2}(\varphi_2(\lambda),\varphi_2(\lambda'))\,.\]

%
%

\subsection*{Structure of $D(A)$-covers over a normal scheme}
Let $p\colon X\to S$ be a $D(A)$-cover with multiplication in $p_*\oj_X$ given by \[s_{\lambda,\lambda'}\in \Gamma(S,\Li_\lambda^{-1}\otimes\Li_{\lambda'}^{-1}\otimes\Li_{\lambda+\lambda'})\,.\] Assume that the cover is \emph{generically a torsor}. Then the global sections $s_{\lambda,\lambda'}$ are all regular since they are generically isomorphisms.
If $S$ is normal then \[\cyc\colon \Div(S)\to \WDiv(S)\] is injective and hence $\cyc(\Li_\lambda^{-1}\otimes\Li_{\lambda'}^{-1}\otimes\Li_{\lambda+\lambda'},s_{\lambda,\lambda'})$
determines $(\Li_\lambda^{-1}\otimes\Li_{\lambda'}^{-1}\otimes\Li_{\lambda+\lambda'},s_{\lambda,\lambda'})$.

Let $C$ be an irreducible component of the branch locus $B$ and let $\spec \oj_{X,C}=X\times_S\spec \oj_{S,C}$. Then \[\spec \oj_{X,C}\to \spec \oj_{S,C}\] is an affine ramified $G$-cover. For each $\lambda\in A$ we let $v_\lambda$  be a generator for the graded piece of $\oj_{X,C}$ corresponding to $\lambda$ (the graded pieces are free since we are over a local ring).

\begin{prop}\label{prop:structure}
With the setup just described, we have an isomorphism \[\Li_\lambda^{-1}\otimes\Li_{\lambda'}^{-1}\otimes\Li_{\lambda+\lambda'}\cong \oj_S(\sum_{\substack{C\subseteq B \\ irred}}\ord_C(s_{\lambda,\lambda'})[C])\] sending $s_{\lambda,\lambda'}$ to the canonical global section. Let $s$ be a uniformizer of $\oj_{S,C}$. Then $\ord_C(s_{\lambda,\lambda'})$ can be determined by the formula
\[\ord_C(s_{\lambda,\lambda'})=\min\{n\in \mathbb{N}:s^n\in (v_\lambda v_{\lambda'}:v_{\lambda+\lambda'})\}\] where $(v_\lambda v_{\lambda'}:v_{\lambda+\lambda'})$ is the ideal quotient.
\end{prop}

\begin{proof}
The pair $(\Li_\lambda^{-1}\otimes\Li_{\lambda'}^{-1}\otimes\Li_{\lambda+\lambda'},s_{\lala})$ determines an effective Cartier divisor which in turn gives the Weil divisor
\[\sum_{\substack{C\subseteq B \\ irred}}\ord_C(s_{\lambda,\lambda'})[C]\]
since $s_{\lala}$ has support in $B$. This proves the first part.

To prove the second part we consider the cover $\spec\oj_{X,C}\to \spec\oj_{S,C}$. We have \[\oj_{X,C}=\oj_{S,C}[\{v_\lambda\}_{\lambda\in A}]/(\{v_\lambda v_{\lambda'}-s^{\ord(s_\lala)}v_{\lambda+\lambda'}\}_{(\lambda,\lambda')\in A^2})\,.\]
This proves the second part since $\oj_{X,C}$ is free over $\oj_{S,C}$.
\end{proof}

\begin{rk}\label{rk:hens}
Note that we could replace $\oj_{S,C}$ by its strict henselization \cite[\href{https://stacks.math.columbia.edu/tag/0AP3}{Tag 0AP3}]{stacks-project}.
\end{rk}

\begin{rk}
The function \[\begin{split}A\times A & \to \mathbb{N} \\ (\lambda,\lambda') & \mapsto \ord_C(s_{\lambda,\lambda'})\end{split}\] is a 2-cocycle.
\end{rk}

In \cite{Pardini}, Pardini gives an explicit description of Proposition \ref{prop:structure} in the case when $X$ is normal and $S$ is smooth over over an algebraically closed field $k$ whose characteristic does not divide $|A|$. In this setting, if $C$ is an irreducible component of the branch locus $B$, then the stabilizer group of a component in $p^{-1}(C)$ is always cyclic \cite[Lemma 1.1]{Pardini} (i.e., its group of characters is cyclic). For every such $C$, the corresponding stabilizer group $D(N)$ acts via some character $\psi\in N\cong D(D(N))$ (which generates $N$) on the cotangent space $\m_T/\m^2_T$, where $T$ is any component of $p^{-1}(C)$. The character $\psi$ is independent of the choice of $T$.
This means that to every component $C$ we may associate a cyclic group together with a generator.
Hence we may write \[B=\sum_{N}\sum_{\psi} D_{N,\psi}\,, \] where we sum over cyclic quotients $A\twoheadrightarrow N$ and generators $\psi\in N$.

Let $i\colon A\to \N$ be the dual of the inclusion $D(N)\to D(A)$ composed with the map $N\to \N$ defined by $x\mapsto \min\{a:\psi^a=x\}$. For $\lambda,\lambda'\in A$ (and $N$, $\psi$ as above), Pardini defines \[\varepsilon_{\lambda,\lambda'}^{N,\psi}=\begin{cases}0\,, & \mbox{if }i(\lambda)+i(\lambda')<|N|\,, \\ 1\,, & \mbox{otherwise}\,.\end{cases}\]
We have that $\varepsilon_{(-,-)}^{N,\psi}$ is a 2-cocycle and the following theorem is part of \cite[Theorem 2.1]{Pardini}.

\begin{thm}
Let $C$ be a component with cyclic group $N$ and generator $\psi$. Let \[s_{\lambda,\lambda'}\in \Gamma(S,\Li_\lambda^{-1}\otimes\Li_{\lambda'}^{-1}\otimes\Li_{\lambda+\lambda'})\] be the global section corresponding to the multiplication $\Li_\lambda\otimes\Li_{\lambda'}\to \Li_{\lambda+\lambda'}$. Then \[\ord_C(s_{\lambda,\lambda'})=\varepsilon_{\lambda,\lambda'}^{N,\psi}\,.\]
\end{thm}

For completeness, we give a proof.

\begin{proof}
Let $\phi\colon A\to N$ be the dual of the inclusion $D(N)\to D(A)$. The cover \[X_C^{sh}=X\times_S\spec\oj_{S,C}^{sh}\to S_C^{sh}=\spec \oj_{S,C}^{sh}\] factors as \[X_C^{sh}\to X'_C\to S_C^{sh}\] where the first arrow is a totally ramified $D(N)$-cover and the second is a trivial $D(K)$-torsor for $K=\ker \phi$ (since char$(k)\nmid |A|$).
Hence $\Gamma(X'_C,\oj_{X'_C})$ is just a product of copies of $\oj_{S,C}^{sh}$ and we may replace $S_C^{sh}$ by one of the connected components of $X'_C$. Hence we may assume that $X_C^{sh}\to S_C^{sh}$ is connected which implies that $R=\Gamma(X_C^{sh},\oj_{X_C}^{sh})$ is a local ring.
By \cite[Lemma 1.2]{Pardini}, we may choose a generator $x_\psi$ for the line bundle of $X_C^{sh}$ of weight $\psi$ such that $x_\psi$ is a generator of the maximal ideal of $R$. Let $s=x_{\psi}^{|N|}$. Then $s$ is a generator for the maximal ideal in $\oj_{S,C}^{sh}$ since $R$ is normal by assumption.
It follows that the line bundle $\Li_{\lambda,C}$ is generated by $x_\lambda=x_{\psi}^{i(\lambda)}$ and since  \[x_{\lambda}x_{\lambda'}=s_{\lambda,\lambda'}x_{\lambda+\lambda'}\]
we get that
\[s_{\lambda,\lambda'}=\begin{cases}1\,, & \mbox{if }i(\lambda)+i(\lambda')<|N|\,, \\ s\,, & \mbox{otherwise}\,.\end{cases}\qedhere\]
\end{proof}

\subsection*{Branch locus}
In this subsection we consider only $D(A)$-covers $X\to S$ such that there is an open dense subscheme $U\subseteq S$ such that $U\times_SX\to U$ is a $D(A)$-torsor.

We define the ramification locus $R\subset X$ of a $G$-cover $\pi\colon X\to S$ as the set of points where $X$ is ramified, i.e., the set of points in $X$ where $\Omega_{X/S}$ do not vanish. Hence there is a canonical scheme structure on the ramification locus, namely \[\spec \oj_X/\Ann(\Omega_{X/S})\,.\] However, there are several ways to put a scheme structure on the branch locus (the set-theoretic image of the ramification locus) $B\subset S$ of a $G$-cover. We will compare two possible choices of ideals defining the branch locus:

\begin{enumerate}
\item the discriminant ideal $d(\pi)\subseteq \oj_S$, and
\item the ideal $\oj_S\cap\Ann(\Omega_{X/S})=\Ann(\pi_*\Omega_{X/S})$.
\end{enumerate}

\begin{rk}
One could also consider the zeroth Fitting ideal $\Fit_0(\Omega_{X/S})\subseteq \Ann(\Omega_{X/S})$ which has the same radical as $\Ann(\Omega_{X/S})$. But we will not do this here.
\end{rk}

\begin{lem}\label{lem:disc}
If $X\to S$ is a $D(A)$-cover where $S=\spec R$ and all line bundles defining $X$ are trivial, then $X$ is the spectrum of the ring \[R_X=R[\{v_\lambda\}_{\lambda\in A}]/(\{v_\lambda v_{\lambda'}-s_{\lambda,\lambda'}v_{\lambda+\lambda'}\}_{\lambda,\lambda'\in A})\] and the discriminant $d(\pi)$ of the cover is given by the formula
\[d(\pi)=\left(|A|^{|A|}\prod_{\lambda\in A}s_{\lambda,-\lambda}\right)\,.\]
\end{lem}

\begin{proof}
The first assertion is trivial and we prove the second. The discriminant $d(\pi)$ is the determinant of the map $R_X\to (R_X)^\vee$ defined on the generators by \[v_{\lambda}\mapsto \Tr(m_{-\cdot v_{\lambda}})\] where $\Tr(m_{-\cdot v_{\lambda}})$ is the map $R_X\to R$ sending a generator $v_{\lambda'}$ to the trace of the map $m_{v_{\lambda'}v_{\lambda}}$ given by multiplication by $v_{\lambda'}v_{\lambda}=s_{\lambda,\lambda'}v_{\lambda+\lambda'}$. The matrix of $m_{v_{\lambda}v_{\lambda'}}$ in the basis $\{v_{\lambda''}\}_{\lambda''\in A}$ will have no element on the diagonal if $\lambda+\lambda'\neq 0$, and $s_{\lambda, -\lambda}$ at every entry of the diagonal otherwise. Hence the trace of the matrix will be either 0 or $|A|s_{\lambda, -\lambda}$. This means that 
  \[
    \Tr(m_{-\cdot v_{\lambda}})=|A|s_{\lambda,-\lambda}v_{-\lambda}^*\,,
  \]
where $v_{-\lambda}^*$ is the dual of $v_{-\lambda}$. Hence $d(\pi)$ is the determinant of the map $v_{\lambda}\mapsto |A|s_{\lambda,-\lambda}v_{-\lambda}^*$ and we conclude that 
  \[
    d(\pi)=\pm|A|^{|A|}\prod_{\lambda\in A}s_{\lambda,-\lambda}\,.\qedhere
  \]
\end{proof}

\begin{ex}[Non-example]\label{ex:non-example}
  Let $p\geq 3$ be an odd prime number and let $K$ be the cyclotomic field obtained by adding a primitive $p$th root of unity to $\mathbb{Q}$. Let $\oj_K\subset K$ be the ring of integers. Then $\spec\oj_K\to \spec \ZZ$ is not a ramified $D(\ZZ/(p-1))$-cover since the discriminant of $K$ is a power of $p$ whereas $p-1$ divides the discriminant of any $D(\ZZ/(p-1))$-cover by Lemma \ref{lem:disc}.
\end{ex}
  
Despite the fact that $\spec\oj_K\to \spec \ZZ$ of the previous example is not a ramified $D(\ZZ/(p-1))$-cover, we have that the stack quotient $[\spec\oj_K/\ZZ/(p-1)]$ is a stacky cover (Definition \ref{df:st-cov}). Also see Example \ref{ex:cy}.

\begin{lem}\label{lem:branch}
If $\pi\colon X\to S$ is a $D(A)$-cover, then \[d(\pi)\subseteq\Ann(\pi_*\Omega_{X/S})\,.\]
\end{lem}

\begin{proof}
Throughout this proof we write $A$ multiplicatively. We may reduce to the case where $X$ and $S$ are affine and all line bundles of the cover are trivial. Let $v_\lambda$ be a generator for the graded piece of $\oj_X$ corresponding to $\lambda\in A$. The module of Kähler differentials is generated by the elements $dv_\lambda$. Let $n=n_\lambda$ be the order of $\lambda\in A$. We know that $v_\lambda^{n}$ lies in the zeroth piece so $nv_\lambda^{n-1}dv_\lambda=0$ and hence $nv_\lambda^{n-1}$ annihilates $dv_\lambda$. Hence it is enough to show that the generator of $d(\pi)$ contains a factor $nv_\lambda^{n-1}$ for each $\lambda\in A$. We expand the part of the product (in the formula for $d(\pi)$ given in Lemma \ref{lem:disc}) which is indexed by elements in the subgroup generated by $\lambda$. We will also group the elements two-by-two: $(s_{\lambda,\lambda^{-1}}s_{\lambda^{-1},\lambda})$, $(s_{\lambda^2,\lambda^{-2}}s_{\lambda^{-2},\lambda^2}), \dots$ except when $n$ is even, where one element will be grouped alone. Most importantly, we use the \emph{associativity} (and commutativity) of $\oj_X$: $s_{\lambda,\lambda^{n-1}}=s_{\lambda^{n-1},\lambda}=s_{\lambda,\lambda^{n-1}}s_{1, \lambda^2}=s_{\lambda^{n-1}, \lambda^2}s_{\lambda, \lambda}$. For $n$ even, we get
  \[
    \begin{split}
      d(\pi) & =  |A|^{|A|}\prod_{\lambda'\in A}s_{\lambda',(\lambda')^{-1}} \\ & =
                  |A|^{|A|}\tilde{s}(s_{\lambda,\lambda^{n-1}}s_{\lambda^{n-1},\lambda})(s_{\lambda^2,\lambda^{n-2}}s_{\lambda^{n-2},\lambda^2})\dots \\ & =
                  |A|^{|A|}\tilde{s}(s_{\lambda,\lambda}s_{\lambda^2,\lambda^{n-1}})^2(s_{\lambda^2,\lambda}s_{\lambda^3,\lambda^{n-2}})^2\dots(s_{\lambda^{n/2-1},\lambda}s_{\lambda^{n/2},\lambda^{n/2}})^2s_{\lambda^{n/2},\lambda^{n/2}} \\ & =
                  |A|^{|A|}\tilde{s}'s_{\lambda,\lambda}s_{\lambda^2,\lambda}\dots s_{\lambda^{n/2-1},\lambda}s_{\lambda,\lambda^{n/2-1}}\dots s_{\lambda^2,\lambda}s_{\lambda,\lambda} \\ & =
                  |A|^{|A|}\tilde{s}'((\dots(v_\lambda v_\lambda)v_\lambda)\dots(v_\lambda(v_\lambda v_\lambda))\dots) \\ & =
                  |A|^{|A|}\tilde{s}'v_\lambda^{n}\,,\end{split}
  \] 
for some elements $\tilde{s}, \tilde{s}'$, and hence $d(\pi)$ annihilates $dv_\lambda$ since $n$ divides $|A|$. Note here that, in the forth equality, we simply pick the first element from each parenthesis and write these as a product, followed by that same product written out backwards (each parenthesis was raised to a power of two). All other elements in the parentheses becomes a factor of $\tilde{s}'$. 

Similarly, when $n$ is odd we get
\[\begin{split}d(\pi)& =
|A|^{|A|}\prod_{\lambda\in A}s_{\lambda,\lambda^{-1}} \\ & =
|A|^{|A|}\tilde{s}(s_{\lambda,\lambda^{n-1}}s_{\lambda^{n-1},\lambda})(s_{\lambda^2,\lambda^{n-2}}s_{\lambda^{n-2},\lambda^2})\dots \\ & =
|A|^{|A|}\tilde{s}(s_{\lambda,\lambda}s_{\lambda^2,\lambda^{n-1}})^2\dots(s_{\lambda^{\frac{n-1}{2}-1},\lambda}s_{\lambda^{\frac{n-1}{2}},\lambda^{\frac{n+3}{2}}})^2(s_{\lambda^{\frac{n-1}{2}},\lambda}s_{\lambda^{\frac{n+1}{2}},\lambda^{\frac{n+1}{2}}})^2 \\ & =
|A|^{|A|}\tilde{s}''s_{\lambda,\lambda}s_{\lambda^2,\lambda}\dots s_{\lambda^{\frac{n-1}{2}},\lambda}s_{\lambda^{\frac{n-1}{2}},\lambda}\dots s_{\lambda,\lambda} \\ & =
|A|^{|A|}\tilde{s}''v_\lambda^{n}\,,
\end{split}\] which again annihilates $dv_\lambda$ since $nv_\lambda^{n-1}$ annihilates $dv_\lambda$. Since $\lambda$ was arbitrary we conclude that $d(\pi)\subseteq \Ann(\pi_*\Omega_{X/S})$.
\end{proof}

The inclusion in Lemma \ref{lem:branch} is most often strict.

\begin{ex}
Consider the $\mmu_{3,S}$-cover $\pi\colon X:=\spec \ZZ[s,x,y]/(x^2-sy,y^2-sx,xy-s^2)\to S:=\spec \ZZ[s]$ where $x$ has weight 1 and $y$ weight 2. We have \[\begin{split}\Ann(\pi_*\Omega_{X/S}) & = (3s^2) \\ d(\pi) & = (27s^4)\,.\end{split}\]
\end{ex}

%% file: DF-structures.tex
\section{Deligne--Faltings structures and root stacks}\label{sec:DF}

In this section we discuss the notions of Deligne--Faltings structures and root stacks associated to a Deligne--Faltings structure together with a homomorphism of monoids. The main reference for this section is \cite{Borne-Vistoli}. All monoids are assumed to be \emph{commutative}.

We first recall some basic definitions about monoids:

\begin{df}\label{mon-defs}
A monoid $P$ is called
\begin{enumerate}
\item \emph{finitely generated} if there is a number $n\in \N$ and a surjection $\N^n\to P$;
\item \emph{sharp} if $P^\times=\{0\}$;
\item \emph{integral} if $p,q,q'\in P$ and $p+q=p+q'$ implies that $q=q'$;
\item \emph{u-integral} if $P^\times$ acts freely on $P$;
\item \emph{quasi-integral} if $p,q\in P$ and $p+q=p$ implies that $q=0$;
\item \emph{fine} if it is integral and finitely generated;
\item \emph{quasi-fine} if it is quasi-integral and finitely generated.
\end{enumerate}
\end{df}

\begin{rk}
A monoid $M$ is integral if and only if the canonical map $M\to M^{gp}$ is injective.
\end{rk}

\subsection*{Deligne--Faltings structures}

\begin{df}\label{df:div}
We denote by $\DIV_{S_{\et}}$ the restriction of $\DIV_S=[\Af^1_S/\Gm_{m,S}]$ to the small \'{e}tale site of $S$.
\end{df}

The following definition is very closely related to the notion of a log structure (see Remark \ref{rk:log}):

\begin{df}
Let $S$ be a scheme. A \emph{pre-Deligne--Faltings structure} (\emph{pre-DF-structure} short) \[\Li\colon \PP\to \DIV_{\SEt}\] consists of
\begin{enumerate}
\item a presheaf $\PP$ of monoids on $S_\et$, and
\item a symmetric monoidal functor $\Li\colon \PP\to \DIV_{\SEt}$.
\end{enumerate}
A pre-DF-structure is called a \emph{Deligne--Faltings structure} (\emph{DF-structure} short) if $\PP$ is a sheaf and $\Li$ has trivial kernel (Recall that the ``zero'' in $\DIV_{\SEt}$ is $(\oj_S, 1)$).
\end{df}

\begin{rk}
Given a pre-Deligne--Faltings structure there is a notion of the \emph{associated Deligne--Faltings structure} \cite[Proposition 3.3]{Borne-Vistoli}.
\end{rk}

\begin{rk}
Here we view $\PP$ as a symmetric monoidal category where all arrows are identities and the tensor product is given by the binary operation in $\PP$. A symmetric monoidal category is a braided monoidal category such that for each pair of objects $a$ and $b$, the diagram \[\xymatrix{a\otimes b\ar[rd]^{\id}\ar[r]^{\gamma_{a,b}} & b\otimes a\ar[d]^{\gamma_{b,a}} \\ & a\otimes b}\] commutes, where $\gamma_{a,b}\colon a\otimes b\to b\otimes a$ is the braiding isomorphism.
By a symmetric monoidal functor we mean a braided monoidal functor as in \cite[IV, \S 2, p. 257]{Maclane}, i.e., a monoidal functor which commutes with the braiding.
\end{rk}

\begin{rk}\label{rk:log}
The notion of a Deligne--Faltings structure is equivalent to the notion of a \emph{u-integral} log structure \cite[Theorem 3.6]{Borne-Vistoli}, that is, a log structure $\rho\colon M\to \oj_S$ such that the action of $\rho^{-1}\oj_S^\times\simeq \oj_S^\times$ on $M$ is free. If $\PP\to \DIV_{S_\et}$ is a Deligne--Faltings structure, then the corresponding log structure is given by the projection \[\PP\times_{\DIV_{S_\et}}\oj_S\to \oj_S\,,\]
where the map $\oj_S\to \DIV_{S_\et}$ sends $f\in \Gamma(U,\oj_U)$ to $(\oj_U,f)$.
\end{rk}

\begin{rk}
Note that $\Li$ may have trivial kernel but still map different elements to isomorphic objects. For example, let $P$ be the constant monoid $\N^2$ and $\Li$ a line bundle on $S$ with a global section $s$, and assume that $\Li$ is non-trivial or that $s$ vanishes at some point of $S$. Then the symmetric monoidal functor which sends both $(0,1)$ and $(1,0)$ to $(\Li,s)$ has trivial kernel.
\end{rk}

\begin{df}[{\cite[Definition 4.1]{Borne-Vistoli}}]\label{df:Kummer}
A morphism of monoids $\varphi\colon P\to Q$ is called \emph{Kummer} if
\begin{enumerate}
\item it is injective and
\item for every $q\in Q$ there is an $n\in \N$ and $p\in P$ such that $\varphi(p)=nq$.
\end{enumerate}
A morphism of \'etale sheaves of monoids $\PP\to \Q$ is called \emph{Kummer} if for every geometric point $\overline{x}\in S$, $\PP_{\overline{x}}\to \Q_{\overline{x}}$ is Kummer.
\end{df}

\begin{df}
A \emph{chart} for a sheaf of monoids $\PP$ is a finitely generated monoid $P$ together with a homomorphism of monoids $P\to \PP(S)$ such that the induced morphism $P_S\to \PP$ is a cokernel in the category of sheaves of monoids.
An \emph{atlas} for $\PP$ consists of an \'{e}tale covering ${U_i\to S}$ together with charts $P_i\to \PP(U_i)$ for each $i$.
If $\varphi\colon \PP\to \Q$ is Kummer then a \emph{chart} for $\varphi$ consists of charts $P\to \PP(S)$, $Q\to \Q(S)$, and a Kummer homomorphism $P\to Q$ such that the induced diagram \[\begin{tikzcd}P\ar{r}\ar{d} & Q\ar{d} \\ \PP(S)\ar{r} & \Q(S)\end{tikzcd}\] commutes.
\end{df}

\begin{df}
A sheaf of monoids $\PP$ is
\begin{enumerate}
\item \emph{sharp} if $\PP(U)$ is sharp for every object $U$ in the site, i.e., $\PP(U)$ has a unique invertible element, namely 0, and
\item \emph{coherent} if it is sharp and has an atlas.
\end{enumerate}
\end{df}

\begin{lem}[{\cite[Lemma 4.7]{Borne-Vistoli}}]
Let $\PP$ be a coherent sheaf of monoids. Let \[\begin{tikzcd}P\ar{r}{\varphi}\ar{d} & Q\ar{d} \\ \PP(S)\ar{r} & \Q(S)\end{tikzcd}\] be a chart for a Kummer homomorphism $\PP\to \Q$ and denote by $K_P$ and $K_Q$ the kernels of $P_S\to \PP$ and $Q_S\to \Q$ respectively. If $U\to S$ is \'{e}tale and $q\in K_Q(U)$, then there is an \'{e}tale cover $\{U_i\to U\}$, integers $n_i\in \N$, and elements $p_i\in K_P(U_i)$ for all $i$ such that $\varphi(p_i)=n_iq|_{U_i}$.
\end{lem}

\begin{proof}
Since $P\to Q$ is Kummer, there is an \'{e}tale cover $\{U_i\to U\}$, integers $n_i\in \N$, and elements $p_i\in P_S(U_i)$ for all $i$ such that $\varphi(p_i)=n_iq|_{U_i}$. But the image of $p_i$ in $\Q(U_i)$ is zero and since $\PP\to \Q$ has trivial kernel we see that $p_i$ must map to zero in $\PP(U_i)$, i.e., $p_i\in K_P(U_i)$.
\end{proof}

A morphism of sharp monoids $\phi\colon P\to Q$ induces a map of schemes $\spec \ZZ[Q]\to \spec \ZZ[P]$ and we may ask what property A the map $\phi$ need to have for the map on spectra to have a property B.

\begin{df}\label{df:int-mor}
A morphism $\phi\colon P\to Q$ of monoids is called
\begin{enumerate}
  \item \emph{integral} if it satisfies the following condition:
  Whenever $q_1,q_2\in Q$ and $p_1,p_2\in P$ satisfy $\phi(p_1)+q_1=\phi(p_2)+q_2$ there exist $q'\in Q$ and $p_1', p_2'\in P$ such that \[\begin{split}q_1 & =\phi(p_1')+q'\,, \\ q_2 & =\phi(p_2')+q'\,, \\ p_1 & +p_1' = p_2+p_2'\,;\end{split}\]
  \item \emph{flat} if it is integral and satisfy the following supplementary condition:
  Whenever $q\in Q$ and $p_1,p_2\in P$ satisfy $\phi(p_1)+q=\phi(p_2)+q$ there exist $q'\in Q$ and $p'\in P$ such that \[\begin{split}q & =\phi(p')+q'\,, \\ p_1 & +p' = p_2+p'\,.\end{split}\]
\end{enumerate}

\end{df}


\begin{rk}
One may think of the property of being integral as allowing us to complete every pair of solid arrows to a commutative square: \[\begin{tikzcd}q'\ar[dashed]{r}{p_1'}\ar[dashed]{d}{p_2'} & q_1\ar{d}{p_1} \\ q_2\ar{r}{p_2} & q_0\,.\end{tikzcd}\]
Flatness, in addition, allows us to complete every pair of solid arrows to: \[\begin{tikzcd}q'\ar[dashed]{r}{p'} & q\arrow[r, shift left=0.5ex]{}{p_1}\arrow[r, shift left=-0.5ex,swap]{}{p_2} & q_0\,,\end{tikzcd}\]
such that the two compositions agree.
\end{rk}

\begin{rk}\label{rk:flatGroup}
Let $P\hookrightarrow Q$ be an injective morphism of integral monoids. Note that we can complete every pair of arrows to a commutative square: \[\begin{tikzcd}q''\ar[dashed]{r}{p_2}\ar[dashed]{d}{p_1} & q_1\ar{d}{p_1} \\ q_2\ar{r}{p_2} & q\end{tikzcd}\] in the \emph{associated group} $Q^{gp}$, if we put $q''=q_1-p_2=q_2-p_1$. Hence if there is a $p\in P$ such that
$q''+p\in Q$, $p_1-p\in P$, and $p_2-p\in P$, then $P\hookrightarrow Q$ is integral.
\end{rk}

\begin{lem}\label{lem:flat}
Let $f\colon \spec \ZZ[Q]\to \spec \ZZ[P]$ be the morphism induced by a morphism $\phi\colon P\to Q$ of integral monoids. Then $f$ is flat if and only if $\phi$ is flat if and only if $\phi$ is integral and injective.
\end{lem}

\begin{proof}
See \cite[Remark 4.6.6]{Ogus}.
\end{proof}

\subsection*{Root stacks}

\begin{df}
Let $\Li\colon \PP\to \DIV_{\SEt}$ be a symmetric monoidal functor and $j\colon \PP\to \Q$ a Kummer homomorphism of sheaves of monoids. This will be referred to as a \emph{Deligne--Faltings datum}.
\end{df}

\begin{df}[{\cite[Definition 4.16]{Borne-Vistoli}}]
Let $\Li\colon \PP\to \DIV_{\SEt}$ be a symmetric monoidal functor and $j\colon \PP\to \Q$ a homomorphism of sheaves of monoids. The \emph{root stack} associated to this Deligne--Faltings datum, denoted $S_{\PP,\Q,\Li}$ or $S_{\Q/\PP}$ is the fibered category over $S$ associated with the following pseudo-functor:
Let $f\colon T\to S$ be a morphism of schemes. This gives a symmetric monoidal functor $f^*\Li\colon f^*\PP\to \DIV_{T_{\et}}$ by pulling back $\Li$.
We also get a morphism of sheaves of monoids $f^*\PP\to f^*\Q$ and we define the category $(f^*\Li)(f^*\Q/f^*\PP)$ with
\begin{enumerate}
\item objects: pairs $(\Ei,\alpha)$, where $\Ei\colon f^*\Q\to \DIV_{T_\et}$ is a symmetric monoidal functor, and $\alpha\colon f^*\Li\to \Ei\circ f^*j$ is an isomorphism of symmetric monoidal functors. This is pictured in the following diagram: 
  \[
    \begin{tikzcd}f^*\PP\arrow{r}[name=U]{{f^*\Li}}\ar{d}[left]{f^*j}  & \DIV_{T_\et}  \\ f^*\Q\arrow[Rightarrow, from=U, shorten <>=10pt]{}[swap]{\alpha}\ar{ur}[below]{{\Ei}} & \,,
    \end{tikzcd}
  \]
\item morphisms $(\Ei',\alpha')\to (\Ei,\alpha)$ given by an isomorphism $\Ei'\to \Ei$ such that the diagram
  \[
    \begin{tikzcd}[row sep=1.5em, column sep=1.5em] & f^*\Li\ar{ddr}{\alpha}\ar{ddl}[swap]{\alpha'} & \\
    & & \\
    \Ei'\circ f^*j\ar{rr} & & \Ei\circ f^*j
    \end{tikzcd}
  \]
commutes.
The category $(f^*\Li)(f^*\Q/f^*\PP)$ is obviously a groupoid.
\end{enumerate}
If we have a commutative diagram of schemes
  \[\begin{tikzcd}[row sep=1.5em, column sep=1.5em]
    T'\ar{rr}{h}\ar{ddr}[swap]{f'} & & T\ar{ddl}{f} \\
    & & & \\
    & S &
  \end{tikzcd}\]
and an object $(\Ei,\alpha)$ in $(f^*\Li)(f^*\Q/f^*\PP)$, then we get a symmetric monoidal functor \[h^*\Ei\colon h^*f^*\PP=f'^*\PP\to \DIV_{T'_\et}\,.\]
We may also pull back $\alpha\colon f^*\Li\to \Ei\circ f^*j$ to an isomorphism $h^*\alpha\colon h^*f^*\Li\to h^*\Ei\circ h^*f^*j$
and if we pre-compose with the natural isomorphism $f'^*\Li\to h^*f^*\Li$ we get a natural isomorphism
  \[
    h^*\alpha\colon f'^*\Li\to h^*\Ei\circ f'^*j\,.
  \]
The pair $(h^*\Ei,h^*\alpha)$ defines an object in $(f'^*\Li)(f'^*\Q/f'^*\PP)$. This defines a functor \[h^*\colon (f^*\Li)(f^*\Q/f^*\PP) \to (f'^*\Li)(f'^*\Q/f'^*\PP)\] on the level of objects.
If $(\Ei',\alpha')\to (\Ei,\alpha)$ is a morphism in $(f^*\Li)(f^*\Q/f^*\PP)$, then via pullback we get an isomorphism $h^*\Ei'\to h^*\Ei$ such that the corresponding diagram over $(f')^*\Li$ commutes.

This is the pseudo-functor corresponding to the root stack $S_{\PP,\Q,\Li}$.
\end{df}

\begin{df}
When $(\PP,\Q,\Li)$ is a Deligne--Faltings datum and $\Li^a\colon \PP^a\to \DIV_S$ the associated Deligne--Faltings structure (see \cite[Proposition 3.3]{Borne-Vistoli}) then we write $\PP^a\to \Q^a$ for the pushout of \[\begin{tikzcd}\PP\ar{r}\ar{d} & \PP^a \\ \Q & \,.\end{tikzcd}\]
\end{df}

\begin{lem}
Let $\PP\to \Q$ be an integral homomorphism of monoids. Then $\PP^a\to \Q^a$ is integral.
\end{lem}

\begin{proof}
See \cite[Proposition 4.6.3(1)]{Ogus}.
\end{proof}

\begin{prop}\label{prop:flat-crit}
Let $\PP\to \Q$ be a homomorphism of fine sheaves of monoids. The stack $S_{\PP,\Q,\Li}$ is flat over $S$ if for every geometric point $\bar{x}\in S$, the morphism $\PP_{\bar{x}}^a\to \Q_{\bar{x}}^a$ is integral.
\end{prop}

\begin{proof}
This can be checked \'{e}tale locally. Hence we may assume that there is a chart \[\xymatrix{P\ar[r]\ar[d] & Q\ar[d] \\ \PP\ar[r] & \Q\,.}\] By \cite[Proposition 4.18]{Borne-Vistoli} we have that $S_{P,Q,\Li}$ is isomorphic to the stack $S\times_{[\spec \ZZ[P]/D(P)]}[\spec \ZZ[Q]/D(Q)]$ where $D(P),D(Q)$ denotes the Cartier duals of $P$ and $Q$
respectively (or equivalently, of $P^{gr}$ and $Q^{gr}$), and the action is given by the obvious grading of $\ZZ[P]$ and $\ZZ[Q]$ respectively. We have a commutative diagram \[\xymatrix{[\spec \ZZ[Q]/G]\ar[r]\ar[d] & [\spec\ZZ[Q]/D(Q)]\ar[d] \\ \spec\ZZ[P]\ar[r] & [\spec \ZZ[P]/D(P)]}\] where $G=D(Q^{gp}/P^{gp})$. Since $D(P)$ is fppf over $S$, we get that $\spec\ZZ[Q]\to \spec\ZZ[P]$ is flat $\Leftrightarrow$
$[\spec\ZZ[Q]/D(Q)]\to [\spec\ZZ[P]/D(P)]$ is flat $\Rightarrow$ $S_{P,Q,\Li}\to S$ is flat.
By \cite[Proposition 3.17]{Borne-Vistoli}, we may assume that $P\cong \PP_{\bar{x}}^a$ and $Q\cong \Q_{\bar{x}}^a$ and hence we are done by Lemma \ref{lem:flat}.
\end{proof}

\begin{rk}
  The converse of Proposition \ref{prop:flat-crit} does not hold as the flatness of $S_{P,Q,\Li}\to S$ does also depend on $\Li$. Consider for example the cuspidal cubic $S=\spec R$, where $R=\CC[t^2, t^3]\subseteq \CC[t]$. Let $P=\langle 2,3\rangle \hookrightarrow Q=\NN$ be the submonoid generated by the elements 2 and 3. Then the inclusion $P\hookrightarrow Q$ is clearly not integral and hence not flat. If we take $\Li$ to be the symmetric monoidal functor generated by sending $2\mapsto (\oj_S, t^2)$ and $3\mapsto (\oj_S, t^3)$, then the root stack $S_{P,Q,\Li}$ is the normalization of $S$ which is not flat. On the other hand, if we let $\Li$ be the symmetric monoidal functor generated by $2\mapsto (\oj_S, 0)$ and $3\mapsto (\oj_S, 0)$, then it is not hard to see that $S_{P,Q,\Li}\cong \spec (\oj_S[\varepsilon]/(\varepsilon^2))$ which is clearly flat over $S$.  
\end{rk}

%% file: special-DF-data.tex
\section{Special Deligne--Faltings data}\label{sec:sp-DF-data}
Let $A$ be a finite abelian group. The idea of the following section is to introduce monoids $P_A$ and $Q_A$ together with a homomorphism $\gamma_A\colon P_A\to Q_A$ such that every $D(A)$-cover $X\to S$ will give rise to a symmetric monoidal functor $\Li_X\colon P_A\to \DIV S$ and such that the root stack associated to the Deligne--Faltings datum
  \[\begin{tikzcd}
    P_A\ar{r}{\Li_X}\ar{d}{\gamma_A} & \DIV S \\ Q_A &
  \end{tikzcd}\]
is isomorphic to $[X/D(A)]$.
The monoids $P_A,Q_A$ and the morphism $\gamma_A$ will depend only on the group $A$ but $\Li_X$ will depend on $X$ and the action of $D(A)$.

\subsection*{Free extensions and 2-cocycles}
Whenever we write \emph{monoid} we mean \emph{commutative} monoid.

\begin{rk}
Recall that an action of a monoid $P$ on a set $S$, written $(p,s)\mapsto ps$, is \emph{free} if there exists a basis $T\subseteq S$. That is, a subset $T\subseteq S$ such that the induced function $P\times T\to S$ sending $(p,t)$ to $pt$ is a bijection.
If $Q$ is a monoid and $P$ a submonoid, we get an action of $P$ on $Q$ by addition.
\end{rk}

\begin{df}\label{df:ext}
Let $A$ be an abelian group and $P$ a monoid. A \emph{free extension of $A$ by $P$} (with a chosen basis) is an exact sequence $E$ of monoids
\[
  0\to P\xrightarrow{\gamma} Q \xrightarrow{m} A\to 0
\]
together with a set-theoretic section  $\iota\colon A\to Q$ such that $\iota(0)=0$ and $Q$ is free over $P$ with basis $\iota(A)$. This means that the function $\varphi_E\colon P\times A\to Q$ sending $(p,\lambda)$ to $\gamma(p)+\iota(\lambda)$ is a bijection.
\end{df}

\begin{rk}\label{rk:mon-ram-cov}
  Note that, given a free extension as in Definition \ref{df:ext}, we get a splitting of $\ZZ[Q]$ as a $\ZZ[P]$-module
  \[
  \ZZ[Q]\cong \bigoplus_{\lambda\in A}\ZZ[P]\,,
  \]
  where $(p,\lambda)$ has degree $\lambda$. Hence $\spec\ZZ[Q]\to \spec\ZZ[P]$ is a ramified $D(A)$-cover (Definition \ref{def:cover}).
\end{rk}

\begin{rk}
Definition \ref{df:ext} also makes sense when $A$ is just a monoid. A free extension of $A$ with values in $P$ without the choice of a section $\iota$ is called a \emph{Schreier extension} (see \cite{Redei, Strecker, Inasaridze_1965} or \cite[Definition 4.1]{Patchkoria_2018}).
\end{rk}

\begin{rk}
When $P$ is sharp and $Q$ is quasi-integral, the section $\iota$ of Definition \ref{df:ext} is uniquely determined. Indeed, if we have sections $\iota_1$ and $\iota_2$, both making $Q$ free over $P$ with basis $\iota_1(A)$ and $\iota_2(A)$ respectively, then for any $\lambda\in A$, there are elements $p_1$ and $p_2$ in $P$ such that $p_2+\iota_1(\lambda)=\iota_2(\lambda)$ and $p_1+\iota_2(\lambda)=\iota_1(\lambda)$. Hence $p_1+p_2+\iota_1(\lambda)=p_1+\iota_2(\lambda)=\iota_1(\lambda)$. But then $p_1+p_2=0$ since $Q$ is quasi-integral and $p_1=p_2=0$ since $P$ is sharp.   
\end{rk}

\begin{df}
A 1-morphism of extensions from $0\to P\xrightarrow{\gamma} Q \xrightarrow{m} A\to 0$ to $0\to P\xrightarrow{\gamma'} Q' \xrightarrow{m'} A\to 0$, with sections $\iota$ and $\iota'$ respectively, is an isomorphism $\varphi\colon Q\to Q'$ such that the diagrams 
  \[
    \begin{tikzcd}
      P\ar{r}{\gamma}\ar[equal]{d} & Q \ar{r}{m}\ar{d}{\varphi} & A \ar[equal]{d} \\
      P\ar{r}{\gamma'} & Q' \ar{r}{m'} & A
    \end{tikzcd}  
    \,,\quad 
    \begin{tikzcd}
      Q \ar{d}{\varphi} & \ar{l}[swap]{\iota} A \ar[equal]{d}\\
      Q' & \ar{l}[swap]{\iota'} A
    \end{tikzcd}
  \]
commute.
We denote by $\ExtAP$ the 1-category of free extensions of $A$ by $P$ with 1-morphisms between them.
\end{df}

Recall that $\CAP$ denotes the set of commutative 2-cocycles (Definition \ref{def:cap}). The following proposition shows that $\ExtAP$ is naturally equivalent to a set. 

\begin{prop}\label{prop:ext-coc}
There is an equivalence
\[
\ExtAP\simeq \CAP\,.
\]
\end{prop}

\begin{proof}
Let $E$ be a free $A$-extension $0\to P\xrightarrow{\gamma} Q \xrightarrow{m} A\to 0$ with basis $\iota\colon A\to Q$ and let $\varphi_E\colon P\times A\to Q$ be the induced bijection.
Let $\nabla\colon Q\times Q\to Q$ be the addition. Then we get a canonical function
\[
f_E\colon A\times A\to Q\times Q\xrightarrow{\nabla}Q\to P
\]
where $A\times A\to Q\times Q$ and $r\colon Q\to P$ are the canonical inclusion and projection respectively, obtained via $\varphi_E$.
Since $Q$ is commutative and associative and since \[A\xrightarrow{\iota}Q\xrightarrow{r}P\] is zero we conclude that $f_E\colon A\times A\to P$ is a 2-cocycle.
We define
\[\begin{split}\Psi\colon \ExtAP & \to\CAP \\
E & \mapsto f_E
\end{split}\]
on objects. Two isomorphic extensions will give the same 2-cocycle and hence $\Psi$ is in fact a functor. 

Conversely, given a 2-cocycle $f\colon A\times A\to P$, define a monoid $Q=P\times_fA$ with underlying set $P\times A$ and addition given by
\[
(p,\lambda)+(p',\lambda')=(p+p'+f(\lambda,\lambda'),\lambda+\lambda')\,.
\]
This gives a free $A$-extension $0\to P\to Q\to A\to 0$ with $\gamma\colon P\to Q$ and $\iota\colon A\to Q$ the canonical inclusions. Hence we get a functor $\Theta\colon \CAP \to \ExtAP$.
We leave to the reader to check that $\Theta$ and $\Psi$ are quasi-inverse to each other.
\end{proof}

\begin{rk}
Note that the bijection in Proposition \ref{prop:ext-coc} provides $\ExtAP$ with the structure of a monoid.
\end{rk}



\subsection*{The universal extension and the universal 2-cocycle}
Let $A$ be an abelian group. We will define a \emph{universal} free extension $E_A$ of $A$
\[
0\to P_A\to Q_A\to A\to 0
\]
such that for any extension $E : 0\to P\to Q\to A\to 0$, there exists unique morphisms $P_A\to P$ and $Q_A\to Q$ such that the diagram
  \[
    \begin{tikzcd}
      P_A \ar{r}\ar{d} & P\ar{d} \\
      Q_A\ar{r}\ar{d} & Q\ar{d} \\
      A\ar[equal]{r} & A
    \end{tikzcd}
  \]
commutes.
The corresponding 2-cocycle $A\times A\to P_A$ is called the \emph{universal} 2-cocycle of $A$.

\begin{df}[{{\cite[Definition 4.1]{Tonini}}}]\label{df:congr-rel}
Let $R_A\subset \NGG\times\NGG$ be the congruence relation generated by the relations \[\begin{split}e_{\lambda,\lambda'} & \sim e_{\lambda',\lambda} \\ e_{0,\lambda} & \sim 0 \\ e_{\lambda,\lambda'}+e_{\lambda+\lambda',\lambda''} & \sim e_{\lambda',\lambda''}+e_{\lambda'+\lambda'',\lambda} \,.\end{split}\]
\end{df}

\begin{df}\label{df:P}
We define $P_A = \N^{\GG}/R_A$.
\end{df}

\begin{rk}\label{rk:uni-coc}
There is a function $e_{(-,-)}\colon A\times A\to P_A$ sending $(\lambda,\lambda')$ to $e_{\lambda,\lambda'}$ which by definition of $R_A$ is a 2-cocycle. This will be referred to as the \emph{universal} 2-cocycle. One immediately checks that any 2-cocycle $A\times A\to P$ factors uniquely through
$e_{(-,-)}\colon A\times A\to P_A$.
\end{rk}


\begin{df}\label{df:Q}
We define
$Q_A=P_A\times_{e_{(-,-)}} A$\,,
i.e., $Q_A$ is the monoid in the universal free extension of $A$ corresponding to the universal 2-cocycle $e_{(-,-)}\colon A\times A\to P_A$.
That is, $Q_A$ has underlying set $P_A\times A$ with addition $(p,\lambda)+(p',\lambda')=(p+p'+e_{\lambda,\lambda'},\lambda+\lambda')$.
Let $\gamma_A\colon P_A\to Q_A$ be the canonical inclusion $p\mapsto (p,0)$.
\end{df}

\begin{lem}\label{lem:free-kummer}
  Let $A$ be a finite abelian group. The morphism $\gamma_A\colon P_A\to Q_A$ is a Kummer homomorphism and the induced action of $P_A$ on $Q_A$ is free.
  \end{lem}
  
  \begin{proof}
  For any $(p,\lambda)\in Q_A$ we have $\ord(\lambda)(p,\lambda)=(\ord(\lambda)p+e_{\lambda,\lambda}+e_{2\lambda,\lambda}+\dots+e_{(\ord(\lambda)-1)\lambda,\lambda},0)$. The rest follows readily from the definitions.
  \end{proof}

\subsection*{The monoid $Q_A^+$}

There are two monoid homomorphisms
  \[
    \Sigma, \Pi\colon \mathbb{N}^{\GG}\to \mathbb{N}^A
  \]
which are defined on the basis by
  \[\begin{split}
    \Sigma(e_\lala) & =e_\lambda+e_{\lambda'} \\ \Pi(e_\lala) & = e_{\lambda+\lambda'}\,.
  \end{split}\]
Consider the induced monoid homomorphism
  \[
    \NGG\xrightarrow{(\Sigma,\Pi)}\NGxGo\,.
  \]
We have that
  \[
    (\NGxGo)/(\Sigma,\Pi)(R_A)\cong \mathbb{Z}^A/\langle e_0\rangle\,,
  \]
where $(\Sigma,\Pi)(R_A)$ is the induced congruence relation on $\NGxGo$.
Indeed, $0\sim e_{0,\lambda}$ maps to $0\sim(e_0+e_\lambda,e_\lambda)= (e_\lambda,e_\lambda)$ in $\NGxGo$ and $\mathbb{Z}^A/\langle e_0\rangle$ is obtained as the quotient of $\NGxGo$ by the diagonal.
The induced map is \[\begin{split}\varphi_A\colon P_A & \to \ZZ^A/\langle e_0\rangle \\ e_{\lambda,\lambda'} & \mapsto e_\lambda+e_{\lambda'}-e_{\lambda+\lambda'}\,.\end{split}\]

\begin{rk}
The monoid $P_A$ need not be integral. For instance, it is not integral when $A=\ZZ/2\ZZ\times \ZZ/2\ZZ\times \ZZ/2\ZZ$ (here we used Macaulay 2). This means that the morphism $\varphi_A$ need not be injective.
\end{rk}

\begin{df}
Define $m\colon \ZZ^A/\langle e_0\rangle\to A$ to be the group homomorphism defined on generators by $e_\lambda\mapsto \lambda$.
\end{df}

\begin{rk}\label{rk:1}
As in \cite[Definition 4.4]{Tonini} we may consider the short exact sequence of abelian groups
\[\begin{split}0\to K\to \ZZ^{A}/\langle e_0\rangle & \xrightarrow{m} A\to 0 \\ e_{\lambda} & \mapsto \lambda \,.\end{split}\]
and $\varphi_A\colon P_A\to \ZZ^{A}/\langle e_0\rangle$ factors through $P_A\to K$, which is the groupification of $P_A$ \cite[Lemma 4.5]{Tonini}.
\end{rk}

\begin{df}
We define $Q_A^+ = \N^A/\langle e_0\rangle\,.$
\end{df}

\begin{rk}\label{rk:2}
  By \cite[Lemma 4.5]{Tonini} $P_A\to \varphi_A(P_A)$ is the associated integral monoid and hence we identify $P_A^{int}$ with $\varphi_A(P_A)$.
  Similarly, the map $Q_A\to \ZZ^A/\langle e_0\rangle\,;\ (p,\lambda)\mapsto \varphi_A(p)+e_{\lambda}$ is a homomorphism and the image
  is the associated integral monoid
  $Q_A^{int}\cong\langle P_A^{int},Q_A^+\rangle\subset \ZZ^A/\langle e_0\rangle$. Also $Q_A^{int}\cong P_A^{int}\times_{e_{(-,-)}} A$.
  \end{rk}

  \begin{ex}
  If $A=\ZZ/3\ZZ$, then we have $P_A^{int}\subset Q_A^{int}\subset \ZZ^2$ with $P_A^{int}=\langle (2,-1),(-1,2)\rangle\cong \N^2$ and $Q_A^{int}=\langle (1,0),(0,1),(2,-1),(-1,2)\rangle$.
  This is illustrated in Figure \ref{fig:monoids}.
  
  \begin{figure}[h]
  \begin{center}
  \begin{tikzpicture}[scale=1]

  \foreach\l[count=\y] in {-1,0,...,5}
  {
  \draw[dashed] (-2,\y) -- (5.5,\y);
  }
  \draw[->,thick] (-2,2) -- (6,2);
  \draw[->,thick] (0,0) -- (0,8);
  \foreach \x in {-1,0,...,5}
  {
  \draw[dashed] (\x,0) -- (\x,7.5);
  }
  
  \foreach \z in {0,1,...,5}
  {
  \foreach \w in {0,1,...,5}{
  \fill[blue] (\z,\w+2) circle[radius=2pt];
  }
  }
  \fill[blue] (-1,5) circle[radius=2pt];
  \fill[blue] (-1,6) circle[radius=2pt];
  \fill[blue] (3,1) circle[radius=2pt];
  \fill[blue] (4,1) circle[radius=2pt];
  
  \draw[blue] (0,2) -- (4.5,-0.25);
  \draw[blue] (0,2) -- (-2.25,6.5);

  \fill[red] (-1,4) circle[radius=2pt];
  \fill[red] (-1,7) circle[radius=2pt];
  \fill[red] (2,1) circle[radius=2pt];
  \fill[red] (0,5) circle[radius=2pt];
  \fill[red] (5,1) circle[radius=2pt];
  \fill[red] (3,2) circle[radius=2pt];
  \fill[red] (1,3) circle[radius=2pt];
  \fill[red] (2,4) circle[radius=2pt];
  \fill[red] (3,5) circle[radius=2pt];
  \fill[red] (4,6) circle[radius=2pt];
  \fill[red] (5,7) circle[radius=2pt];
  \fill[red] (-2,6) circle[radius=2pt];
  \fill[red] (4,0) circle[radius=2pt];
  \fill[red] (1,6) circle[radius=2pt];
  \fill[red] (4,3) circle[radius=2pt];
  \fill[red] (0,2) circle[radius=2pt];
  \fill[red] (2,7) circle[radius=2pt];
  \fill[red] (5,4) circle[radius=2pt];

  \end{tikzpicture}
  \end{center}
  \caption{The red dots represents the elements of $P_A^{int}\subset Q_A^{int}$ and the blue dots represents the elements of $Q_A^{int}\setminus P_A^{int}$ with $A=\ZZ/3\ZZ$.}\label{fig:monoids}
  \end{figure}
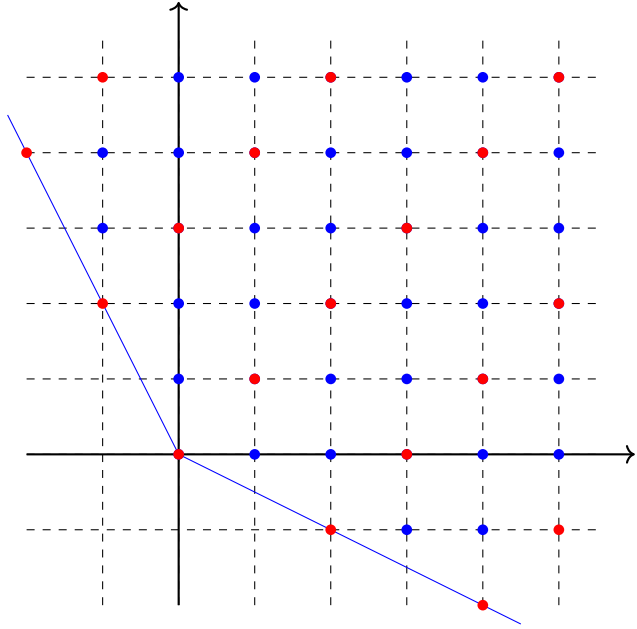
  
\end{ex}

\begin{df}
  The group homomorphism \[\ZZ^A/\langle e_0\rangle\to \ZZ \] defined on generators by sending $e_\lambda$ to 1 is written $q\mapsto |q|$ and we call $|q|$ the \emph{value} of $q$.
  \end{df}
  
  \begin{lem}\label{lem:pos-gen}
  Let $M$ be a finitely generated monoid and $v\colon M\to \ZZ$ a homomorphism. If $M$ is generated by elements $s$ such that $v(s)\geq 1$ then $M$ is sharp.
  \end{lem}
  
  \begin{proof}
  Suppose that $s+s'=0$. Then $v(s)+v(s')=0$ and since $v(s)\geq 1$ unless $v=0$ we get that $s=s'=0$.
  \end{proof}
  
  \begin{lem}
  The monoids $P_A^{int}$ and $Q_A^{int}$ are fine and sharp.
  \end{lem}
  
  \begin{proof}
  Both $P_A^{int}$ and $Q_A^{int}$ are finitely generated and integral (see Remark \ref{rk:2}) and hence fine. It remains to show that the monoids are sharp and since $P_A^{int}\subseteq Q_A^{int}$ it is enough to show that $Q_A^{int}$ is sharp.
  But $Q_A^{int}$ and the homomorphism $q\mapsto |q|$ satisfies the condition of Lemma \ref{lem:pos-gen}, and hence $Q_A^{int}$ is sharp.
  \end{proof}
  
  \begin{lem}
  The monoids $P_A$ and $Q_A$ are quasi-fine and sharp.
  \end{lem}
  
  \begin{proof}
  The canonical morphism $Q_A\to Q_A^{int}$ sends $(p,\lambda)$ to $\varphi_A(p)+e_\lambda$. Hence $Q_A$ and the homomorphism $(p,q)\mapsto |\varphi_A(p)+e_\lambda|$ satisfies the condition of Lemma \ref{lem:pos-gen}, and hence $Q_A$ is sharp. But then $P_A$ is also sharp since it is a submonoid.
  Both $P_A$ and $Q_A$ are finitely generated by definition so it remains so prove that they are quasi-integral. This follows from the fact that they are generated by elements of strictly positive value. Indeed, $|\varphi_A(p)+e_\lambda|\geq 1$ unless $(p,\lambda)=0$. If $(p,\lambda)=(p,\lambda)+(p',\lambda')$ then $|\varphi_A(p')+e_\lambda'|=0$ and hence $(p',\lambda')=0$. This completes the proof.
  \end{proof}

We will now define $Q_A$ as a quotient of $P_A\oplus Q_A^+$. 

\begin{df}
For $q\in Q_A^+$, write $q=\sum_{i=1}^ne_{\lambda_i}$ where we may have $\lambda_i=\lambda_j$ for $i\neq j$. Define $h(q)\in P_A$ by $h(q)=0$ if $n\leq 1$, and
\[
h(q)=e_{\lambda_1,\lambda_2}+e_{\lambda_1+\lambda_2,\lambda_3}+\dots+e_{\lambda_1+\dots+\lambda_{n-1},\lambda_n}
\]
otherwise.
\end{df}

\begin{rk}\label{rk:Q-rel}
  Note that by definition of the equivalence relation $R$ in Definition \ref{df:congr-rel}, the element $h(q)$ is independent of the order of the $\lambda_i$'s in the representation $q=\sum_{i=1}^ne_{\lambda_i}$. This implies that we have a set-theoretic function
  \[\begin{split}h\colon Q_A^+ & \to P_A \\
    q & \mapsto h(q)\,,
    \end{split}
  \]
  which satisfies $q=\varphi_A(h(q))+e_{m(q)}$.
\end{rk}

\begin{rk}\label{rk:ass}
Note that $\sum_{i=1}^n(p_i,\lambda_i)=(\sum_{i}p_i+h(\sum_{i}e_{\lambda_i}),\sum_i\lambda_i)$ in $Q_A$.
\end{rk}

\begin{rk}\label{rk:f-happy-rel}
The function $h\colon Q_A^+\to P_A$ satisfies the relation $h(q+q')=h(q)+h(q')+e_{m(q),m(q')}$.
\end{rk}

\begin{df}
We define
\[\begin{split}
j\colon Q_A^+& \to Q_A \\
q & \mapsto (h(q),m(q))\,.
\end{split}
\]
\end{df}

\begin{rk}
The map $j$ is a homomorphism by Remark \ref{rk:f-happy-rel}. Furthermore, there is a morphism $Q_A\to \ZZ^A/\langle e_0\rangle$ sending $(p,\lambda)$ to $\varphi_A(p)+e_\lambda$ and by Remark \ref{rk:Q-rel}, the composition $Q_A^+\to Q_A\to \ZZ^A/\langle e_0\rangle$ is the canonical inclusion. Hence we conclude that $j$ is injective and from now on we view $Q_A^+$ as a submonoid of $Q_A$.
\end{rk}

\begin{df}\label{df:mon-Q}
Let $R_P$ be the congruence relation on $P_A\oplus Q_A^+$ generated by the relation $(e_{\lambda,\lambda'},e_{\lambda+\lambda'})\sim (0,e_{\lambda}+e_{\lambda'})$.
\end{df}

\begin{rk}\label{rk:good-gen-for-R_P}
By transitivity we have $(e_{\lambda,\lambda'}+e_{\lambda+\lambda',\lambda''},e_{\lambda+\lambda'+\lambda''})\sim (0,e_{\lambda}+e_{\lambda'}+e_{\lambda''})$
since
\[\begin{split}(e_{\lambda,\lambda'}+e_{\lambda+\lambda',\lambda''},e_{\lambda+\lambda'+\lambda''}) & \sim  (e_{\lambda,\lambda'},e_{\lambda+\lambda'}+e_{\lambda''})  \\
& \sim (0,e_{\lambda}+e_{\lambda'}+e_{\lambda''})\,.
\end{split}\]
Iterating this process we conclude that $R_P$ contains the relation $(0,q)\sim (h(q),e_{m(q)})$ for any $q\in Q_A^+$. Since $R_P$ is a congruence relation, it is symmetric and closed under addition. Hence we conclude that $R_P$ contains the relation $R'$ defined by $(p,q)\sim (p',q')$ if $m(q)=m(q')$ and $p+h(q)=p'+h(q')$. It is clear that $R'$ is an equivalence relation and by Remark \ref{rk:f-happy-rel} it follows that $R'$ is a congruence relation. Note that $R'$ contains the relation $(e_{\lambda,\lambda'},e_{\lambda+\lambda'})\sim (0,e_{\lambda}+e_{\lambda'})$ and hence $R_P=R'$.
\end{rk}

\begin{rk}\label{rk:iso-mon}
  We have a function
  \[\begin{split}\tau\colon P_A\oplus Q_A^+ & \to Q_A \\
    (p,q) & \mapsto (p+h(q),m(q))\,
    \end{split}
  \]
 and by Remark \ref{rk:f-happy-rel} we have
 \[\begin{split}
\tau(p,q)+\tau(p',q') & = (p+p'+h(q)+h(q')+e_{m(q),m(q')},m(q)+m(q')) \\
& = (p+p'+h(q+q'),m(q+q')) \\
& = \tau(p+p',q+q')
\end{split}
\]
and hence $\tau$ is a homomorphism of monoids. Furthermore, we have $\tau(p,q)=\tau(p',q')$ if and only if $m(q)=m(q')$ and $p+h(q)=p'+h  (q')$. This means that $\tau$ induces a morphism $(P_A\oplus Q_A^+)/R_P\to Q_A$.

We also have a function
 \[\begin{split}\eta\colon Q_A & \to P_A\oplus Q_A^+ \\
   (p,\lambda) & \mapsto (p,e_\lambda)
   \end{split}
 \]
 which becomes a homomorphism when we quotient by $R_P$. Furthermore, $\tau\circ \eta=\id_{Q_A}$ and we conclude that $(P_A\oplus Q_A^+)/R_P\to Q_A$ is an isomorphism.
\end{rk}

\subsection*{Flat Kummer homomorphisms}

Now consider the following situation. Suppose that we have a flat Kummer homomorphism (Definition \ref{df:Kummer} and \ref{df:int-mor}) of quasi-fine (Definition \ref{mon-defs}) and sharp monoids $\gamma\colon P\to Q$. Note that $A=Q/P$ is a finite abelian group.
For $q,q'\in Q$ we write $q\leq q'$ if there exists a $p\in P$ such that $q'=q+p$. The relation $\leq$ is a partial order since $Q$ and $P$ are quasi-integral and sharp.

\begin{prop}\label{prop:mon-minimal}
Let $P\hookrightarrow Q$ be a flat Kummer morphism of quasi-fine and sharp monoids. Put $A=Q/P$ and let $m\colon Q\to A$ be the quotient. The set $Q_\lambda=m^{-1}(\lambda)$ has a unique minimal element $\iota(\lambda)$. Moreover, $Q_\lambda=\iota(\lambda)+P$.
\end{prop}

\begin{proof}
If $m(q_1)=m(q_2)$, then there exists $p_1$ and $p_2$ such that $q_1+p_1=q_2+p_2$ and since $P\to Q$ is flat there exists $q'\in Q$ and $p_1',p_2'\in P$ such that $q_1=q'+p_1'$, $q_2=q'+p_2'$, and $p_1+p_1'=p_2+p_2'$. Hence $q'\leq q_1$ and $q'\leq q_2$.
Since $Q$ is finitely generated and $m^{-1}(\lambda)$ is partially ordered and positive, we get that $m^{-1}(\lambda)$ has a unique minimal element \cite[Corollary 1.2]{Evans-Tot-ord-mon}.
\end{proof}

%
\begin{cor}\label{cor:un-sec}
Let $P\hookrightarrow Q$ be a flat Kummer morphism of quasi-fine and sharp monoids and let $\lambda\in A=Q/P$ with $m\colon Q\to A$ the quotient. There is a set-theoretic section $\iota\colon A\to Q$ sending $\lambda$ to the unique minimal element $\iota(\lambda)$ in $m^{-1}(\lambda)$. Furthermore, $0\to P\to Q\to A\to 0$ is a free extension of $A$ by $P$ with basis $A$.
\end{cor}

\begin{proof}
By Proposition \ref{prop:mon-minimal} we need only show that the map $P\times A\to Q$ sending $(p,\lambda)$ to $p+\iota(\lambda)$ is injective. Suppose that $p_1+\iota(\lambda)=p_2+\iota(\lambda)$. Then the flatness hypothesis says that there exists a $q\in Q$ and $p'\in P$ such that $q+p'=\iota(\lambda)$ and $p'+p_1=p'+p_2$.
But $q\in m^{-1}(\lambda)$ and $q\leq \iota(\lambda)$ implies that $q=\iota(\lambda)$ by minimality and unicity of $\iota(\lambda)$.
Hence $\iota(\lambda)+p'=\iota(\lambda)$ so $p'=0$ since $Q$ is quasi-integral. Hence $p_1=p_2$. This proves the claim.
\end{proof}

\begin{prop}\label{prop:flat-can-diag}
Let $\gamma\colon P\to Q$ be a flat Kummer homomorphism of quasi-fine and sharp monoids $P$ and $Q$, and let $A=Q/P$. Then there exists canonical morphisms $Q_A\to Q$ and $P_A\to P$ such that the following diagram commutes and the upper square is a pushout:
\[\begin{tikzcd}P_A\ar{r}\ar{d} & P\ar{d} \\ Q_A\ar{r}\ar{d} & Q\ar{d} \\ A\ar[equal]{r} & A\,.\end{tikzcd}\]
\end{prop}

\begin{proof}
This follows from Proposition \ref{prop:ext-coc}, Remark \ref{rk:uni-coc}, and Corollary \ref{cor:un-sec}. We leave the proof to the reader. 
\end{proof}

\begin{rk}
Proposition \ref{prop:flat-can-diag} in particular implies that, whenever we have a Deligne--Faltings datum $(P,Q,\Li)$ with $P$ and $Q$ (constant) quasi-fine and sharp, and with $P\to Q$ a flat Kummer homomorphism, we have a canonical diagram
\[
\begin{tikzcd}
  P_A\ar{r}\ar{d} & P\ar{d} \\
  Q_A \ar{r} & Q
\end{tikzcd}
\]
inducing an isomorphism of root stacks $S_{Q_A/P_A}\simeq S_{Q/P}$. This also works when we have sheaves of monoids instead of constant monoids (see Definition \ref{df:monoids-gen}).
\end{rk}

%% file: stacky-2-cocycles.tex
\section{2-cocycles in symmetric monoidal categories}\label{sec:2-coc-in-sym-mon}
The theory of 2-cocycles and extensions can be generalized to the setting of categories fibered in symmetric monoidal groupoids. Throughout the section we ignore the base category, which will in all applications be the category of schemes over our base scheme $S$. A category fibered in symmetric monoidal groupoids which is a stack will be referred to as a \emph{symmetric monoidal stack}.

\begin{df}
  A \emph{1-cochain} of $\A$ with values in $\M$ is a morphism of stacks $\zeta\colon \A\to \M$ together with an isomorphism $\zeta(0)\simeq \mathbbm{1}_\M$. 

\end{df}

\begin{df}\label{df:weak-2-cocycle}
  Let $\A$ be a sheaf of abelian groups and let $\M$ be a symmetric monoidal stack. A \emph{(commutative) 2-cocycle of $\A$ with values in $\M$} consists of
  \begin{enumerate}
    \item a morphism $f\colon \A\times \A\to \M$ of (ordinary) fibered categories\,,
    \item natural isomorphisms:
    \begin{enumerate}
      \item[(i)] $\uuu_{\lambda}\colon f(0,\lambda)\cong \mathbbm{1}\,,$
      \item[(ii)] $\sss_{\lambda,\lambda'}\colon f(\lambda,\lambda')\cong f(\lambda',\lambda)\,,$
      \item[(iii)] $\aaa_{\lambda,\lambda',\lambda''}\colon f(\lambda',\lambda'')\otimes f(\lambda,\lambda'+\lambda'')\cong f(\lambda,\lambda')\otimes f(\lambda+\lambda',\lambda'')$
    \end{enumerate}
    for all $\lambda,\lambda',\lambda''\in \A$.
  \end{enumerate}
  Furthermore, this data is required to satisfy the following commutativities (analogous to a symmetric monoidal category):
  \begin{enumerate}
    \item $\uuu_0\circ \sss_{0,0}=\uuu_0$\,,
    \item $\sss_{\lambda,\lambda'}\circ\sss_{\lambda',\lambda}=\id$\,,
    \item
      \[\begin{tikzcd}
        f(0,\lambda'')\otimes f(\lambda,0+\lambda'')\ar{d}[swap]{\uuu\otimes \id}\ar{r}{\aaa} & f(\lambda,0)\otimes f(\lambda+0,\lambda'')\ar{d}{(\uuu\circ\sss)\otimes \id} \\
        f(\lambda,\lambda'')\ar[equal]{r} & f(\lambda,\lambda'')\,,
      \end{tikzcd}\]
    \item (Pentagon axiom)
      \[\begin{tikzcd}
        (f(\lambda'',\lambda''')\otimes f(\lambda',\lambda''+\lambda'''))\otimes f(\lambda,\lambda'+\lambda''+\lambda''')\ar{d}[swap]{(\ggamma\otimes\id)\circ \aalpha\circ(\id\otimes\aaa)\circ\aalpha^{-1}}\ar{r}{\aaa\otimes \id} &
        (f(\lambda',\lambda'')\otimes f(\lambda'+\lambda'',\lambda'''))\otimes f(\lambda,\lambda'+\lambda''+\lambda''')\ar{dd}{\aalpha\circ(\id\otimes \aaa)\circ \aalpha^{-1}} \\
        (f(\lambda,\lambda')\otimes f(\lambda'',\lambda'''))\otimes f(\lambda+\lambda',\lambda''+\lambda''')\ar{d}[swap]{\aalpha\circ(\id\otimes \aaa)\circ\aalpha^{-1}} & \\
        (f(\lambda,\lambda')\otimes f(\lambda+\lambda',\lambda''))\otimes f(\lambda+\lambda'+\lambda'',\lambda''')\ar{r}{\aaa^{-1}\otimes \id} & (f(\lambda',\lambda'')\otimes f(\lambda,\lambda'+\lambda''))\otimes f(\lambda+\lambda'+\lambda'',\lambda''')\,,
      \end{tikzcd}\]
    \item (Hexagon axiom)
      \[\begin{tikzcd}
        f(\lambda',\lambda'')\otimes f(\lambda,\lambda'+\lambda'')\ar{r}{\aaa}\ar{d}[swap]{\id\otimes \sss} &
        f(\lambda,\lambda')\otimes f(\lambda+\lambda',\lambda'')\ar{r}{\sss\otimes \id} &
        f(\lambda',\lambda)\otimes f(\lambda+\lambda',\lambda'')\ar{d}{\aaa^{-1}} \\
        f(\lambda',\lambda'')\otimes f(\lambda'+\lambda'',\lambda)\ar{r}[swap]{\aaa^{-1}} &
        f(\lambda'',\lambda)\otimes f(\lambda',\lambda+\lambda'')\ar{r}[swap]{\sss\otimes \id} &
        f(\lambda,\lambda'')\otimes f(\lambda',\lambda+\lambda'')\,.
      \end{tikzcd}\]
  \end{enumerate}
 
  A 2-\emph{coboundary of 2-cocycles} $f'\to f$ is a pair $(\zeta, \tau)$ consisting of a 1-cochain $\zeta\colon \A\to \M$ and a natural isomorphism 
    \[
      \tau\colon f'\otimes (\otimes\circ(\zeta\times\zeta)) \simeq f\otimes(\zeta\circ \Sigma)\,,
    \] 
  where $\Sigma\colon \A\times \A\to \A$ is the addition. This means that we have an isomorphism, natural in $\lambda,\lambda'$, 
    \[
      \tau_{\lambda,\lambda'}\colon f'(\lambda,\lambda')\otimes \zeta(\lambda)\otimes\zeta(\lambda')\to f(\lambda,\lambda')\otimes \zeta(\lambda+\lambda')\  
    \]  
  for all local sections $\lambda,\lambda'$ of $\A$\,.
  Furthermore, $\tau$ is required to commute with the morphisms $\uuu,\sss$, and $\aaa$.
  A 2-\emph{morphism of 2-coboundaries} $(\zeta', \tau')\to (\zeta, \tau)$ is a natural isomorphism $\theta\colon \zeta'\to \zeta$ compatible with $\tau'$ and $\tau$ in the obvious sense. 

  We denote by $\CAM$ the 1-category with 2-cocycles as objects and natural isomorphisms $f'\to f$, which commutes with the morphisms $\uuu,\sss$, and $\aaa$, as morphisms. 
  
  We denote by $\HAM$ the 2-category with 
    \begin{itemize}
      \item objects -- 2-cocycles $f\colon \A\times \A\to \M$\,,
      \item 1-morphisms --  2-coboundaries of 2-cocycles $(\zeta, \tau)\colon f'\to f$\,,
      \item 2-morphisms -- 2-morphisms of 2-coboundaries $\theta\colon (\zeta', \tau')\to (\zeta, \tau)$\,.
    \end{itemize}
\end{df}

Given a 2-cocycle $f\colon \A\times \A\to \M$, we define a symmetric monoidal stack $\M\times_f\A$ consisting of the data
  \[(\M\times \A,\otimes, \mathbbm{1}_{\M\times \A},\aalpha,\llambda,\rrho,\ggamma)\,.\]
The underlying category is $\M\times \A$, the tensor product
  \[\begin{split}
    \otimes\colon (\M\times\A) & \times (\M\times\A) \to \M\times\A \\
    ((x,\lambda) & ,(y,\lambda')) \mapsto ((x\otimes y)\otimes f(\lambda,\lambda'), \lambda+ \lambda')\,,
  \end{split}\]
and the unit is
  \[
    \mathbbm{1}_{\M\times \A}=(\mathbbm{1}_\M,0)\,.
  \]
The natural isomorphisms are defined by
  \[\begin{split}
    \aalpha\colon (x,\lambda)\otimes ((y,\lambda')\otimes (z,\lambda'')) & =((x\otimes ((y\otimes z)\otimes f(\lambda',\lambda'')))\otimes f(\lambda,\lambda'+ \lambda''),\lambda+ (\lambda'+\lambda'')) \\
    & \cong ((x\otimes(y\otimes z))\otimes (f(\lambda',\lambda'')\otimes f(\lambda,\lambda'+ \lambda'')),\lambda+ (\lambda'+\lambda''))\\
    & \cong (((x\otimes y)\otimes z))\otimes (f(\lambda,\lambda')\otimes f(\lambda+ \lambda',\lambda'')),(\lambda+ \lambda')+\lambda'') \\        & \cong ((((x\otimes y)\otimes f(\lambda,\lambda'))\otimes z)\otimes f(\lambda+\lambda',\lambda''),(\lambda+\lambda')+\lambda'')\\
    & = ((x,\lambda)\otimes (y,\lambda'))\otimes (z,\lambda'')\,,
  \end{split}\]
  \[\begin{split}
    \llambda\colon (\mathbbm{1}_\M,0)\otimes (x,\lambda) & = ((\mathbbm{1}_\M\otimes x)\otimes f(0,\lambda),0+ \lambda) \\
    & \cong (x,\lambda)\,,
  \end{split}\]
  \[\begin{split}
    \rrho\colon (x,\lambda)\otimes (\mathbbm{1}_\M,0) & = ((x \otimes\mathbbm{1}_\M)\otimes f(\lambda ,0), \lambda + 0) \\
    & \cong (x,\lambda)\,,
  \end{split}\]
and
  \[\begin{split}
    \ggamma\colon (x,\lambda)\otimes (y,\lambda') & = ((x\otimes y)\otimes f(\lambda,\lambda'),\lambda+ \lambda')\\
    & \cong ((y\otimes x)\otimes f(\lambda',\lambda), \lambda'+ \lambda) \\
    & = (y,\lambda')\otimes (x,\lambda) \\
  \end{split}\]
for all $x,y,z\in \M$ and all $\lambda,\lambda',\lambda''\in \A$.

\begin{prop}
  The data
  \[(\M\times \A,\otimes, \mathbbm{1}_{\M\times \A},\aalpha,\llambda,\rrho,\ggamma)\]
  defines a symmetric monoidal stack $\M\times_f\A$.
\end{prop}

\begin{proof}
This is tedious but straight forward and left to the reader.
\end{proof}

\begin{df}\label{df:ext-cat}
  Let $\A$ be a sheaf of abelian groups and let $\M$ be a symmetric monoidal stack. A \emph{free extension of $\A$ by $\M$} is a short exact sequence $E$ of symmetric monoidal stacks
    \[0\to \M\xrightarrow{\gamma} \C\xrightarrow{m} \A\to 0\]
  together with a section $\iota\colon \A\to \C$ of $m$ such that the induced morphism $\gamma\otimes \iota\colon \M\times \A\to \C$ is an equivalence of fibered categories.

  A \emph{1-morphism} of free extensions $E\to E'$ of $\A$ by $\M$ is a triple $(\varphi,\alpha,\beta)$ consisting of a morphism of symmetric monoidal stacks $\varphi\colon \C\to \C'$ such that $m'\circ\varphi=m$, and natural isomorphisms $\alpha\colon \varphi\circ \gamma\simeq \gamma'$, and $\beta\colon \varphi\circ \iota \simeq \iota'$. 
  A \emph{2-morphism} $(\varphi,\alpha,\beta)\to (\varphi',\alpha',\beta')$ is a natural isomorphism $\varphi\simeq \varphi'$ which commutes with $\alpha,\beta,\alpha', \beta'$ in the obvious sense. 
  We denote the 2-category of free extensions by $\ExtAM$.
\end{df}

We now state the equivalent of Proposition \ref{prop:ext-coc}, which shows that $\ExtAM$ is naturally equivalent to a 1-category. 

\begin{prop}\label{prop:equiv-coc-ext}
  We have an equivalence of groupoids
    \[\ExtAM\simeq \CAM\,.\]
\end{prop}

\begin{proof}
  This is of course very similar to the proof of Proposition \ref{prop:ext-coc}. 
  
  Given a 2-cocycle $f\colon \A\times \A\to \M$ we get an associated symmetric monoidal stack $\M\times_f\A$ as defined above. Given a natural isomorphism of 2-cocycles $\tau\colon f'\to f$, we define $\varphi\colon \M\times_{f'}\A\to \M\times_{f}\A$ to be the identity on objects and with comparison isomorphism
    \[
      \varphi((x,\lambda)\otimes(y,\lambda'))=(x\otimes y\otimes f'(\lambda,\lambda'),\lambda+\lambda')\xrightarrow{(\id\otimes \tau,\id)} (x\otimes y\otimes f(\lambda,\lambda'),\lambda+\lambda')\,.
    \]
  This defines a functor $\psi\colon \CAM \to \ExtAM$. To see that $\psi$ is essentially surjective, consider an extension 
    \[0\to \M\xrightarrow{\gamma} \C\xrightarrow{m} \A\to 0\,.\]
  Choose a quasi-inverse $\C\to \M\times \A$ to obtain a retraction $\rho\colon \C\to \M$. Then the composition 
    \[
      f\colon \A\times \A\xrightarrow{\iota\times\iota}\C\times \C\xrightarrow{\otimes}\C\xrightarrow{\rho}\M\  
    \]
  may be given the structure of a 2-cocycle using the symmetric monoidal structure of $\C$ and we get an isomorphism of extensions $\C\to \M\times_f\A$. 
  
  We leave to the reader to check that $\psi$ is fully faithful. 
\end{proof}

\subsection*{The universal extension and the universal 2-cocycle}
We need a sheafified version of the monoid $P_A$.
\begin{df}\label{df:monoids-gen}
Let $\A$ be an \'etale sheaf of abelian groups of finite type on a scheme $S$. Define $\N^{\A^2}$ and $P_\A$ to be the sheafification of
\[\begin{split}
  U & \mapsto \N^{\A(U)} \,, \mbox{ and }\\
  U &\mapsto P_{\A(U)} \,,
\end{split}\]
respectively (see Definition \ref{df:P}).
Furthermore, we define
  \[
    Q_\A = P_\A\times_{e_{(-,-)}}\A\,.
  \]
\end{df}

We also define groupoid versions of the monoids $P_\A$ and $Q_\A$ (Definition \ref{df:congr-rel}, Definition \ref{df:P}, Definition \ref{df:Q}).

\begin{df}\label{df:groupoid-Q}
  We define $\PA=[R_\A\rightrightarrows \N^{\A^2}]$
  and $\QA = \PA\times_{e_{(-,-)}}\A$.
\end{df}

As in the case of monoids, there is a universal 2-cocycle
  \[
    e_{(-,-)}\colon \A\times \A \to \PA
  \]
and a universal extension
  \[
    \PA\to \QA\to \A\,.
  \]

\begin{lem}\label{lem:quot-equiv}
  The canonical maps $\PA\to P_\A$ and $\QA\to Q_\A$ are equivalences of stacks.
\end{lem}

\begin{proof}
We have that $\PA=[R_\A\rightrightarrows \N^{\A^2}]$ is the stackification of the fibered category $\{R_\A\rightrightarrows \N^{\A^2}\}$ whose fiber over a scheme $U$ is the groupoid $\{R_\A(U)\rightrightarrows \N^{\A^2}(U)\}$. But this groupoid has at most one arrow between any two objects and hence $\{R_\A\rightrightarrows \N^{\A^2}\}$ is equivalent to a category fibered in sets, i.e., an ordinary presheaf. Applying the stackification functor to a presheaf gives the associated sheaf and hence we conclude that
$\PA$ is equivalent to a sheaf. The canonical functor $\PA\to P_\A$ just identifies isomorphic objects and is hence an equivalence. This also shows that $\QA\to Q_\A$ is an equivalence.
\end{proof}

  We will now compare 2-cocycles $f\colon \A\times \A\to \M$ and symmetric monoidal functors $\Li\colon P_\A\to \M$. We assume that our base category is $(\Sch/S)$ and that $\A$ is a sheaf of abelian groups on the small \'etale site of $S$. Passing to the espace \'etal\'e we think of $\A$ as an \'etale algebraic space over $S$.

  Let $\C$ and $\M$ be symmetric monoidal stacks. Let $\HOM(\C,\M)$ denote the fibered category with objects $(T,\Li)$ consisting of a scheme $T\to S$ and a symmetric monoidal functor $\Li\colon \C|_T\to \M|_T$. A morphism $(T',\Li')\to (T,\Li)$ is a pair $(h,\varphi)$ consisting of an $S$-morphism $h\colon T'\to T$ and a natural isomorphism $\varphi\colon h^*\Li\to \Li'$.

  Let $\CAMS$ be the fibered category over $S$ with objects $(T,f)$ consisting of a scheme $T\to S$ and a 2-cocycle $f\colon (\A|_T)^2\to \M|_T$. A morphism $(T',f')\to (T,f)$ is a pair $(h,\psi)$ consisting of an $S$-morphism $h\colon T'\to T$ and a natural isomorphism $\psi\colon h^*f\to f'$ of 2-cocycles.

  The universal 2-cocycle $e_{(-,-)}\colon \A\times \A\to \PA$ gives a functor
    \[
      \Theta\colon \HOM(\PA,\M)\to \CAMS
    \]
  which sends a symmetric monoidal functor $(\Li,\epsilon, \mu)$ (notation as in \cite[Definition 2.1]{Borne-Vistoli}) to the 2-cocycle $(f,\uuu,\sss,\aaa)$ defined by $f(\lambda,\lambda')=\Li(e_\lala)$, $\uuu_\lambda=\Li(e_{0,\lambda},0)$, $\sss_{\lambda,\lambda'}=\Li(e_{\lambda,\lambda'},e_{\lambda',\lambda})$, and
    \[
      \aaa_{\lambda,\lambda',\lambda''}=\Li(e_{\lambda,\lambda'+\lambda''}+e_{\lambda',\lambda''},e_{\lambda,\lambda'}+e_{\lambda+\lambda',\lambda''})\,,
    \]
  where $(e_{0,\lambda},0)$, $(e_{\lambda,\lambda'},e_{\lambda',\lambda})$, and $(e_{\lambda,\lambda'+\lambda''}+e_{\lambda',\lambda''},e_{\lambda,\lambda'}+e_{\lambda+\lambda',\lambda''})$ are local sections of $R_\A$. 
  On morphisms $\Theta$ takes a morphism of symmetric monoidal functors $\Phi\colon \Li'\to \Li$ to the morphism $\Theta(\Phi)\colon \Theta(\Li')\to \Theta(\Li)$ which is just $\Theta(\Phi)_{(\lambda,\lambda')}=\Phi_{e_{\lambda,\lambda'}}\colon \Li'(e_{\lala})\to \Li(e_\lala)$.
  We leave to the reader to check that $\Theta(\Phi)$ is compatible with the 2-cocycle structure.

  \begin{rk}\label{rk:morph-det-on-gen}
    Let $\Phi\colon \Li'\to \Li$ be a morphism of symmetric monoidal functors $\Li',\Li\colon \PA\to \M$. Then for every $p,p'\in \PA$, we have that the diagram
    \[\begin{tikzcd}
      \Li'(p)\otimes\Li'(p')\ar{r}{\mu}\ar{d}{\Phi_p\otimes \Phi_{p'}} & \Li'(p+p')\ar{d}{\Phi_{p+p'}} \\
      \Li(p)\otimes\Li(p')\ar{r}{\mu} & \Li(p+p')
    \end{tikzcd}\]
    commutes. This means that the morphism $\Phi$ is determined on generators of $\PA$.
  \end{rk}

  \begin{prop}\label{prop:comparison}
  The functor $\Theta\colon \HOM(\PA,\M)\to \CAMS$ is an equivalence.
  \end{prop}

  \begin{proof}
  We first show that $\Theta$ is fully faithful. This may be checked on fibers.

  To see that $\Theta$ is full, let $\varphi\colon \Theta(\Li')\to \Theta(\Li)$ be a morphism of 2-cocycles. Then $\varphi$ gives morphisms $\varphi_{\lambda,\lambda'}\colon \Li'(e_{\lala})\to \Li(e_{\lala})$ which via Remark \ref{rk:morph-det-on-gen} gives a morphism of symmetric monoidal functors $\Phi\colon \Li'\to \Li$ such that $\Theta(\Phi)=\varphi$. Hence $\Theta$ is full.

  To see that $\Theta$ is faithful, suppose that $\Theta(\Phi')=\Theta(\Phi)$. Then $\Phi'$ and $\Phi$ agree on generators and again, by Remark \ref{rk:morph-det-on-gen} they must be equal. This shows that $\Theta$ is faithful and completes the proof.

  To show that $\Theta$ is essentially surjective, note that given a 2-cocycle $(f,\uuu,\sss,\aaa)$ we may define a symmetric monoidal functor $\Li\colon \PA\to \M$ on generators by $\Li(e_{\lambda,\lambda'})=f(\lambda,\lambda')$, $\Li(e_{0,\lambda},0)=\uuu_\lambda$, $\Li(e_{\lambda,\lambda'},e_{\lambda',\lambda})=\sss_{\lambda,\lambda'}$,
    \[
      \Li(e_{\lambda,\lambda'+\lambda''}+e_{\lambda',\lambda''},e_{\lambda,\lambda'}+e_{\lambda+\lambda',\lambda''})=\aaa_{\lambda,\lambda',\lambda''}\,,
    \]
  and extend it linearly. Then $\Theta(\Li)=(f,\uuu,\sss,\aaa)$ and hence $\Theta$ is in fact surjective.
  \end{proof}

  \begin{cor}\label{cor:comparison}
    We have an equivalence $\HOM(P_\A,\M)\to \CAMS$.
  \end{cor}

\subsection*{The stack associated to a 2-cocycle}
We will now associate to a 2-cocycle $f\colon \A\times \A\to \DIV_{S_\et}$ a stack $S_{\A,f}$ which will be equivalent to the root stack associated of the universal extension $\PA\to \QA\to \A$ and the symmetric monoidal functor $\Li_f\colon \PA\to \DIV_{S_\et}$ corresponding to $f$. 

\begin{df}\label{def:SAf}
  Let $f\colon \A\times \A\to \M$ be a 2-cocycle. We define $S_{\A,\M,f}$ to be the fibered category over $S$ with objects $(T,\zeta,\kappa)$, consisting of an $S$-scheme $h\colon T\to S$, a 1-cochain $\zeta\colon \A|_T\to \M|_T$, and a natural isomorphism $\kappa\colon \zeta \otimes \zeta\simeq h^*f\otimes (\zeta\circ \Sigma)$.

  A morphism $(T',\zeta',\kappa')\to(T,\zeta,\kappa)$ is a pair $(\varphi, \eta)$ consisting of a morphism of $S$-schemes $\varphi\colon T'\to T$ and a natural isomorphism $\eta\colon \varphi^*\zeta\to \zeta'$ such that the diagram
    \[
      \begin{tikzcd}[column sep = 4.5em]
        \varphi^*\zeta\otimes \varphi^*\zeta \ar{r}{\eta\otimes \eta}\ar{d}{\varphi^*(\kappa)} & \zeta'\otimes\zeta' \ar{d}{\kappa'} \\
        (h')^*f\otimes (\varphi^*\zeta\circ \Sigma) \ar{r}{\id\otimes (\eta\circ \id_\Sigma)} & (h')^*f\otimes (\zeta'\otimes \Sigma)
      \end{tikzcd}
    \]
  commutes.

  In case $\M=\DIV_S$ we write $S_{\A,f}=S_{\A,\M,f}$.
\end{df}


Any 2-cocycle $f\colon \A\times \A\to \DIV_{S_\et}$ factors through a symmetric monoidal functor $\Li_f\colon \PA\to \DIV_{S_\et}$ from which we get a root stack $S_{\PA,\QA,\Li_f}$.

\begin{prop}
  We have a canonical equivalence of fibered categories
    \[
      S_{\A, f}\simeq S_{\PA,\QA,\Li_f}.
    \]
\end{prop}

\begin{proof}
Let us define a functor $\vartheta\colon S_{\A, f}\simeq S_{\PA,\QA,\Li_f}$ on objects as follows: given an object $(T,\zeta,\kappa)$ of $S_{\A,f}$, we define a symmetric monoidal functor 
  \[
    \begin{split}
      \Ei\colon h^*\QA & \to \DIV_{T_\et} \\
      (p,\lambda) & \mapsto h^*\Li_f(p)\otimes \zeta(\lambda)
    \end{split}  
  \]
with comparison isomorphisms 
  \[
    \Ei(p,\lambda)\otimes\Ei(p',\lambda')=h^*\Li_f(p)\otimes h^*\Li_f(p')\otimes \zeta(\lambda)\otimes\zeta(\lambda')\xrightarrow{\mu\otimes \kappa} h^*\Li_f(p+p')\otimes h^*f(\lambda,\lambda')\otimes \zeta(\lambda+\lambda')\,,
  \]
where $\mu$ is the comparison isomorphism of $\Li_f$. Let $\gamma\colon \PA\to \QA$ be the usual inclusion. Then we get an \emph{equality} $\Ei\circ h^*\gamma=h^*\Li_f$. 

Given a morphism $(\varphi, \eta)\colon (T',\zeta',\kappa')\to(T,\zeta,\kappa)$ in $S_{\A,f}$ we define a natural isomorphism $\varphi^*\Ei\to \Ei'$ by $(h')^*\Li_f(p)\otimes \varphi^*\zeta(\lambda)\xrightarrow{\id\otimes \eta} (h')^*\Li_f(p)\otimes \zeta'(\lambda)$. We leave to the reader to check compatibility with $\kappa'$ and $\kappa$. 

The functor $\vartheta$ is essentially surjective since every object in $S_{\PA,\QA,\Li_f}$ is isomorphic to an object whose diagram is strictly commuting. The functor $\vartheta$ is full since a natural isomorphism $\varphi^*\Ei\to \Ei'$ where $\Ei'\circ(h')^*\gamma = (h')^*\Li = \varphi^*\Ei\circ(h')^*\gamma$ comes from a morphism in $S_{\A, f}$. We leave to the reader to check that $\vartheta$ is faithful. 
\end{proof}

%% file: DF-data-from-covers.tex
\section{Deligne--Faltings data from ramified $D(A)$-covers}\label{sec:DF-from-cov}
Let $A$ be a finite abelian group.
Recall that every $D(A)$-cover $f\colon X\to S$ comes with a canonical splitting
\[f_*\oj_X\cong \bigoplus_{\lambda\in A}\Li_\lambda\] and multiplication morphisms
$\Li_\lambda\otimes \Li_{\lambda'}\to \Li_{\lambda+\lambda'}$
for every $\lambda,\lambda'\in A$,
which we think of as global sections \[s_{\lambda,\lambda'}\in \Gamma(S,\Li^{-1}_\lambda\otimes\Li^{-1}_{\lambda'}\otimes\Li_{\lambda+\lambda'})\,.\]
\begin{rk}
  We have that $f\colon X\to S$ is a $D(A)$-torsor if and only if every $s_{\lambda,\lambda'}$ is invertible.
\end{rk}

The quotient stack $\X=[X/D(A)]$ has a canonical $D(A)$-torsor $p\colon X\to \X$ and we have a canonical splitting of $p_*\oj_X$ indexed by the elements of $A$.

\begin{df}
With the notation above we write $\oj_\X[\lambda]$ for the line bundle which is the direct summand of $p_*\oj_X$ of weight $\lambda$, so that
\[p_*\oj_X\cong \bigoplus_{\lambda\in A}\oj_\X[\lambda]\,.\]
\end{df}

We explain the notation in the following remark:

\begin{rk}
Recall that $\oj_\X$-modules corresponds to equivariant $\oj_X$-modules and since $f\colon X\to S$ is affine, these in turn corresponds to equivariant $f_*\oj_X$-modules on $S$. The module $\oj_\X$ may hence be thought of as the $f_*\oj_X$-module $f_*\oj_X$ on $S$ together with the $A$-grading given by the action.
The canonical $D(A)$-torsor $p\colon X\to \X$ corresponds to a morphism of stacks $\X\to BD(A)$ and the character $\lambda\colon D(A)\to \Gm_{m}$ gives a morphism $\lambda\colon BD(A)\to B\Gm_{m}$. The stack $B\Gm_{m}$ has a canonical $\Gm_{m}$-torsor $\spec \ZZ\to B\Gm_{m}$ which is the relative spectrum of a $\ZZ$-graded
$B\Gm_{m}$-algebra. The graded piece of weight 1 is just $\oj_{B\Gm_{m}}[1]$, that is, $\oj_{B\Gm_{m}}$ shifted by 1. This means that $\oj_{B\Gm_{m}}[1]$ is the $\ZZ$-graded $\ZZ$-module which is 0 in every degree except in degree $-1$ where it is $\ZZ$.
Pulling back $\oj_{B\Gm_{m}}[1]$ along $\lambda\colon BD(A)\to B\Gm_{m}$ we get the line bundle $\oj_{BD(A)}[\lambda]$, which is $\oj_{BD(A)}$ shifted by $\lambda$.
If we now pull back $\oj_{BD(A)}[\lambda]$ along $\X\to BD(A)$ we get the line bundle $\oj_\X[\lambda]$ which is nothing but $\oj_\X$ shifted by $\lambda$.
\end{rk}

\begin{rk}\label{rk:counit}
Let $\pi\colon \X=[X/D(A)]\to S$ be the structure morphism. First note that $\pi_*\oj_\X[\lambda]=\Li_\lambda$.
We have a counit \[\varepsilon_\lambda\colon \pi^*\pi_*\oj_\X[\lambda]=\Li_\lambda\otimes_{\oj_S}\oj_\X\to \oj_\X[\lambda]\]
which is $s_{\lambda,\lambda'}\colon \Li_\lambda\otimes\Li_{\lambda'}\to \Li_{\lambda+\lambda'}$ in degree $\lambda'\in A$.
We call the corresponding generalized effective Cartier divisor
  \[
    (\Ei_{\lambda},\varepsilon_{\lambda})=((\pi^*\pi_*\oj_\X[\lambda])^\vee\otimes \oj_\X[\lambda], \varepsilon_\lambda)=(\pi^*\Li_\lambda^\vee[\lambda], \varepsilon_\lambda)\,,
  \] 
the \emph{universal divisor} associated to the character $\lambda\in A$ and we call $\Ei_{\lambda}$ the \emph{universal line bundle} associated to the character $\lambda\in A$.

Similarly, for every pair of universal line bundles $\Ei_\lambda$ and $\Ei_{\lambda'}$ with characters $\lambda,\lambda'\in A$, we have a morphism
\[\oj_S\cong \pi_*\Ei_{\lambda}\otimes\pi_*\Ei_{\lambda'}\to \pi_*(\Ei_{\lambda}\otimes\Ei_{\lambda'})\cong \Li_{\lambda}^{\vee}\otimes\Li_{\lambda'}^{\vee}\otimes\Li_{\lambda+\lambda'}\,,\]
which is nothing but $s_{\lambda,\lambda'}$. Let $(\Li_{\lala},s_\lala):=(\Li_{\lambda}^{\vee}\otimes\Li_{\lambda}^{\vee}\otimes\Li_{\lambda+\lambda'}, s_\lala)$ be the corresponding generalized effective Cartier divisor.
\end{rk}

\begin{rk}\label{rk:iso-2-coc}
  Note that we have canonical isomorphisms
    \[
      \Ei_\lambda\otimes\Ei_{\lambda'}\cong \pi^*(\Li_\lambda^{-1}\otimes\Li_{\lambda'}^{-1}\otimes \Li_{\lambda+\lambda'})\otimes \Ei_{\lambda+\lambda'}\,.
    \]
\end{rk}

\begin{df}\label{df:sym-mon-func}
The \emph{2-cocycle associated to} $X$ is
  \[\begin{split}
    f_X\colon A\times A & \to \DIV S \\
    (\lambda,\lambda') & \mapsto (\Li_{\lala},s_\lala)\,,
  \end{split}\]
with $\uuu, \sss$, and $\aaa$ chosen to be the canonical natural isomorphisms. 
We define the \emph{1-cochain associated to} $\X$ as
  \[\begin{split}
    \zeta_\X\colon A & \to \DIV \X \\
    \lambda & \mapsto (\Ei_{\lambda},\varepsilon_{\lambda})\,,
  \end{split}\]
where $(\Ei_{\lambda},\varepsilon_{\lambda})$ is the universal line bundle defined in Remark \ref{rk:counit}.
Finally, we define 
  \[
    \kappa_X\colon \zeta\otimes\zeta\to \pi^*f_X\otimes (\zeta\circ\Sigma)  
  \]
to be the canonical isomorphism of Remark \ref{rk:iso-2-coc}.
  %
  %
\end{df}

\begin{prop}\label{lem:local}
We have an isomorphism of stacks \[[X/D(A)]\simeq  S_{A, f_X}\] where $S_{A, f_X}$ is the stack defined in Definition \ref{def:SAf}.
\end{prop}

\begin{proof}
The stack $\X=[X/D(A)]\to S$ has a universal object $(X,\varphi)$ where $X\to \X$ is the canonical $D(A)$-torsor and $\varphi\colon X\to X\times_S\X$ the canonical $D(A)$-equivariant morphism over $\X$. In a diagram:
  \[
    \begin{tikzcd}
      X \ar{r}{\varphi}\ar{dr} & X\times_S\X \ar{d} \\
      & \X\,.
    \end{tikzcd}
  \]
Using Definition \ref{df:sym-mon-func} we obtain an object $(\zeta_X,\kappa_X)$ defining a morphism $\X\to S_{A,f_X}$. 

Conversely, from the universal object $(f_X, \zeta,\kappa)$ on the stack $\pi\colon \Y\to S$, we construct an $\oj_\Y$-algebra whose underlying module is
\[\bigoplus_{\lambda\in A}\pi^*\Li_{\lambda}\otimes\zeta(\lambda)\]
and with multiplication $m$, defines via $\kappa$ and the commutative diagram
  \[
    \begin{tikzcd}
      (\pi^*\Li_{\lambda}\otimes\zeta(\lambda))\otimes (\pi^*\Li_{\lambda'}\otimes\zeta(\lambda'))\ar{r}{m}\ar{d}{\sim} &  \pi^*\Li_{\lambda+\lambda'}\otimes\zeta(\lambda+\lambda')\ar{d}{\sim} \\
      \pi^*\Li_{\lambda}\otimes\pi^*\Li_{\lambda'}\otimes \zeta(\lambda)\otimes \zeta(\lambda') \ar{r}{\id\otimes \kappa_{\lambda,\lambda'}} & \pi^*\Li_{\lambda}\otimes\pi^*\Li_{\lambda'}\otimes f_X(\lambda,\lambda')\otimes \zeta(\lambda+\lambda')\,.
    \end{tikzcd}
  \]
This defines a $D(A)$-torsor $Y_{(\zeta,\kappa)}$ on $\Y$ and the global sections defining $\zeta$ yields a $D(A)$-equivariant morphism $Y_{(\zeta,\kappa)}\to \X\times_SX$. This defines a morphism $\Y\to \X$ and we leave to the reader to check that the two functors are quasi-inverse to each other.   
\end{proof}

%% file: stacky-covers-as-root-stacks.tex
\section{Stacky covers as root stacks}\label{sec:birat}

\begin{df}
An algebraic stack $\X$ admitting a coarse moduli space $\pi\colon \X\to S$ is \emph{tame} if the functor $\pi_*\colon \QCoh(\X)\to \QCoh(S)$ is exact.
\end{df}

\begin{df}\label{df:st-cov}
Let $\X$ be a tame stack with finite diagonalizable stabilizers at geometric points and let $S$ be a scheme. We say that $\pi\colon\X\to S$ is a \emph{stacky cover} if it is
\begin{enumerate}
  \item flat, proper, of finite presentation,
  \item a coarse moduli space, and
  \item for any morphism of schemes $T\to S$, the base change $\pi|_T\colon \X_T\to T$ has the property that $(\pi|_T)_*$ takes invertible sheaves to invertible sheaves.
\end{enumerate}
We denote by $\STCOV$ the $(2,1)$-category of stacky covers.
\end{df}

In this section we describe how to reconstruct a stacky cover $\X\to S$ from logarithmic data on $S$.

\begin{rk}
A flat good moduli space $\pi\colon\X\to S$ takes vector bundles to vector bundles. This follows from \cite[Theorem 4.16]{Alper}. Indeed, $\pi_*$ preserves coherence and flatness relative to $S$. If $\Ei$ is a vector bundle on $\X$, then it is flat over $S$ since $\pi$ is flat. Hence $\pi_*\Ei$ is coherent and flat, or equivalently, locally free of finite rank.
\end{rk}

\begin{rk}\label{rk:local-structure}
A stacky cover $\pi\colon \X\to S$ will \'etale locally on $S$ look like a quotient of a ramified cover. That is, for every geometric point $\bar{s}\in S$, there is an \'etale neighborhood $U\to S$ of $\bar{s}$, an abelian group $A$, and a ramified $D(A)$-cover $X\to U$ such that $\X\times_SU\simeq [X/D(A)]$. Indeed, by the local structure theorem for tame stacks \cite[Theorem 3.2]{Abramovich--Olsson--Vistoli} there is a finite $f\colon X\to U$ with action of $\stab(\bar{s})$
such that $\X\times_SU\simeq [X/\stab(\bar{s})]$. But $\stab(\bar{s})$ is of the form $\stab(\bar{s})=D(A)$ since $\X$ has diagonalizable stabilizers at geometric points. Since $X\to \X\times_SU$ is finite flat of finite presentation and $\X\to S$ is finite flat of finite presentation, we get that $f\colon X\to U$ is finite, flat, and of finite presentation. It remains to show that $f_*\oj_X$ looks fppf locally like the regular representation. Since $p\colon X\to \X\times_SU$
is a $D(A)$-torsor, we get a splitting
  \[p_*\oj_X\cong \bigoplus_{\lambda\in A}\oj_\X[\lambda]\]
and since $(\pi|_U)_*$ is exact and takes line bundles to line bundles, we get that \[f_*\oj_X= (\pi|_U)_*p_*\oj_X\cong \bigoplus_{\lambda\in A}\Li_{\lambda}\]
where $\Li_\lambda=(\pi|_U)_*\oj_\X[\lambda]$. Hence we conclude that $X\to U$ is a $D(A)$-cover.
\end{rk}

\subsection*{The relative Picard functor}
Let $\Pic_{\X/S}\colon (\Sch/S)\to \Ab$ denote the relative Picard functor, i.e., the fppf sheafification of the functor
    \[
        U\mapsto \Pic(U\times_S\X)/\Pic(U)\,.
    \]
By \cite[Section 2]{Brochard-Functeur-de-Picard} $\Pic_{\X/S}$ sits in a exact sequence of Picard stacks
    \[
        0\to B\Gm_{m,S}\to \pi_*B\Gm_{m,\X}\to \Pic_{\X/S}\to 0\,.
    \]

\begin{lem}\label{lem:piBGm}
    The stack $\pi_*B\Gm_{m,\X}$ is algebraic and locally of finite presentation. 
\end{lem}

\begin{proof}
    This follows from \cite[Theorem 1.3]{Hall--Rydh:Tannaka}.
\end{proof}

\begin{prop}
    Let $\X\to S$ be a stacky cover. Then $\Pic_{\X/S}$ is representable by an \'{e}tale algebraic space. 
\end{prop}

\begin{proof}
    The representability is \cite[Proposition 2.3.3]{Brochard-Functeur-de-Picard} so we need to show that it is \'{e}tale. We will show that $\Pic_{\X/S}$ is formally \'{e}tale. Let $i\colon T_0\hookrightarrow T$ be a closed immersion of affine schemes over $S$, defined by a quasi-coherent ideal $J$ such that $J^2=0$. Let $\J$ denote the sheaf of ideals corresponding to $\X_{T_0}\hookrightarrow \X_T$. We have an exact sequence 
        \[
            0\to \J\to \Gm_{m,\X_T} \to i_*\Gm_{m,\X_{T_0}}\to 0   
        \]  
    where the morphism $\J\to \Gm_{m,T}$ sends a local section $r$ to $1+r$, which is a homomorphism since $(1+r)(1+r')=1+(r+r')$. Hence we get a long exact sequence in cohomology 
        \[
            \dots \to H^1(\X_T,\J) \to \Pic(\X_T)\to \Pic(\X_{T_0})\to H^2(\X_T, \J)\to \dots    
        \]
    since the Leray spectral sequence and $R^1i_*\Gm_m=0$ yields $H^1(\X_{T}, i_*\Gm_{m,\X_{T_0}})=\Pic(\X_{T_0})$. Since $\X\to S$ is cohomologically of dimension zero and $T$ is affine, it follows that $H^1(\X_T, \J)=H^2(\X_T, \J)=0$ and hence $\Pic(\X_T)\to \Pic(\X_{T_0})$ is an isomorphism. Hence $\Pic_{\X/S}$ is formally \'etale. By Lemma \ref{lem:piBGm} we then conclude that $\Pic_{\X/S}$ is \'etale.    
\end{proof}


\subsection*{Deligne--Faltings data from stacky covers}
Given a stacky cover $\pi\colon \X\to S$ we want to realize $\X$ as a root stack over $S$. 
We will realize $\X$ as a root stack by defining a universal diagram
\[
    \begin{tikzcd}
        \pi^*\PA\arrow{r}[name=U]{}\ar{d}  & \DIV_\X  \\
        \pi^*\QA\arrow[Rightarrow, from=U, shorten <>=10pt]{}{\alpha}\ar{ur}[below]{{\Ei}} & \,,
    \end{tikzcd}
\]
where $\A=\Pic_{\X/S}$. 

\begin{df}
Let $\DIV_{\X/S}$ be the stack of \emph{relative divisors}. That is, the fibered category $\DIV_{\X/S}=\pi_*\DIV_\X$ with objects $(T,D)$ with $T\to S$ a scheme and $D$ an object in $\DIV(T\times_S\X)$.
\end{df}

Note that we have forgetful morphisms
    \[\begin{split}
    \DIV_{\X/S} & \to \pi_*B\Gm_{m,\X}\,, \\
    \DIV_S & \to B\Gm_{m,S}\,.
    \end{split}\]
and hence a diagram
    \[
    \begin{tikzcd}
    \DIV_S\ar{r}\ar{d} & B\Gm_{m,S}\ar{d} \\
    \DIV_{\X/S}\ar{r}\ar{d} & \PIC_{\X/S}\ar{d} \\
    \Pic_{\X/S}\ar[equal]{r} & \Pic_{\X/S}
    \end{tikzcd}
    \]
where the right column is exact by definition and the left column is exact because $\Gamma(S,\oj_S)\cong \Gamma(\X, \oj_\X)$ since $\X\to S$ is a coarse moduli space.

\begin{df}
Let $\X$ be a stacky cover. We define $\PIC_{\X/S}^{triv}$ to be the fibered category (which is equivalent to a sheaf in sets) with objects that are triples $(T,\Ei,t)$ where
\begin{itemize}
\item $T\to S$ is a scheme,
\item $\Ei$ is a line bundle on $T\times_S\X$, and
\item $t\colon \oj_T\to (\pi_T)_*\Ei$ is an isomorphism.
\end{itemize}
A morphism $(T',\Ei',t')\to (T,\Ei,t)$ consists of
\begin{itemize}
    \item a morphism $f\colon T'\to T$,
    \item an isomorphism $\varphi\colon \Ei'\to f^*\Ei$ such that $((\pi_{T'})_*\varphi)\circ t'=f^*t$.
\end{itemize}
\end{df}

\textbf{WARNING:} We may have that $\pi_*(\Ei\otimes\Ei')\not\cong \oj_S$ even if $\pi_*\Ei\cong\oj_S\cong \pi_*\Ei'$:

\begin{ex}
    Let $S=\spec \ZZ[\sqrt{-5}]$ and let $\Li \subseteq \oj_S$ be the line bundle given by the non-principal ideal $I=(2, 1+\sqrt{-5})$. Let $\oj_X$ be the locally free $\ZZ/3\ZZ$-graded $\oj_S$-module $\oj_S\oplus\oj_S\oplus \Li$ and define $s_{1,1}=1+\sqrt{-5}, s_{1,2}=(1+\sqrt{-5})/2$, and $s_{2,2}=1/2$. That is, $X$ is the spectrum of the $\ZZ[\sqrt{-5}]$-algebra $\ZZ[\sqrt{-5}][x,y]/(x^3-(-2+\sqrt{-5}), xy-(1+\sqrt{-5}), y^2-2x)$, where $x$ has weight 1 and $y$ has weight 2. If we let $\X=[X/\mmu_3]$ and $\Ei=\Ei'=\oj_\X[1]$, we get $\pi_*(\Ei)\cong\pi_*(\Ei')\cong \oj_S$ since the degree-zero part of $\oj_\X[1]$ is the degree-one part of $\oj_\X$, which is $\oj_S$. On the other hand, $\Ei\otimes\Ei'=\oj_\X[1]^{\otimes 2}\cong \oj_\X[2]$ which has $\Li$ is degree zero. That is $\pi_*(\Ei\otimes\Ei')\cong\Li\not\cong \oj_S$.     
\end{ex}

%

\begin{df}\label{df:LS}
Let
\[\Lambda\colon \PIC_{\X/S}^{triv}\to \DIV_{\X/S}\]
be the functor given on objects by
\[(T,\Ei,t)\mapsto (T,\Ei,\varepsilon\circ \pi^*t)\,,\] where $\varepsilon\colon \pi^*\pi_*\Ei\to \Ei$ is the counit.
\end{df}

\subsection*{Uniqueness}

Let $\Ei$ be a line bundle on $\X$ and let $f\colon T\to \X$ be an object in the lisse-\'etale site. A lift $T\to \I_\X$ is equivalent to an automorphism $(T,f)\to (T,f)$ and hence we get a linear automorphism $\Ei(T)\to \Ei(T)$. This means that we have a linear action of $\I_\X$ on $\Ei$ and this action is via some character $\lambda\colon \I_\X\to \Gm_{m,\X}$. 

\begin{lem}
    Let $S$ be the spectrum of a strictly henselian ring and let $\X=[X/D_S(A)]$ be the quotient stack of a ramified cover $X\to S$ whose stabilizer over the closed point $\spec k\to S$ is $D_k(A)$. If $\Ei$ is a line bundle on $\X$, then $\I_\X$ acts on $\Ei$ via a character $\I_\X\to \Gm_{m,\X}$ which factors through $\I_X\to D_\X(A)$, i.e., the character of $\Ei$ corresponds to an element of $A$.  
\end{lem}

\begin{proof}
    Over the closed point we know that $\Ei$ has character corresponding to an element $\lambda\in A$. We have a tautological line bundle $\oj_\X[\lambda]$ and $\Ei\otimes \oj_\X[-\lambda]$ has trivial character over the closed point. Hence by \cite[Theorem 10.3]{Alper}, we get that $\Ei\otimes \oj_\X[-\lambda]$ is the pullback of a line bundle on $S$, which says that it has trivial character globally. Thus $\Ei$ has character $\lambda$ globally.     
\end{proof}

\begin{lem}
Let \[\lambda\colon \I_\X\to \Gm_{m,\X}\] be a morphism of $\X$-groups and let $\Ei$ and $\Ei'$ be line bundles on $\X$ on which $\I_\X$ acts via the character $\lambda$. Then the counit \[\pi^*\pi_*(\Ei^\vee\otimes \Ei')\to \Ei^\vee\otimes \Ei'\] is an isomorphism.
\end{lem}

\begin{proof}
This follows from \cite[Theorem 10.3]{Alper}.
\end{proof}

\begin{lem}\label{lem:unique}
Let \[\lambda\colon \I_\X\to \Gm_{m,\X}\] be a morphism of $\X$-groups and let $(\Ei,s)$ and $(\Ei',s')$ be objects in $\PIC_{\X/S}^{triv}(S)$ on which $\I_\X$ acts via the character $\lambda$. Then there exists a unique isomorphism $(\Ei,s)\to (\Ei',s')$ in $\PIC_{\X/S}^{triv}(S)$.
\end{lem}

\begin{proof}
We have a morphism
\[\begin{split}\psi\colon\pi_*\HOM_{\oj_\X}(\Ei,\Ei') & \to \HOM_{\oj_S}(\pi_*\Ei,\pi_*\Ei') \\ \varphi & \mapsto \pi_*\varphi\end{split}\]
and we want to show that this is an isomorphism. Indeed, then there will be a unique morphism $\Ei\to \Ei'$ mapping to the morphism \[\pi_*\Ei\xrightarrow{s^{-1}}\oj_S\xrightarrow{s'}\pi_*\Ei'\,,\] which is what we want to prove.
It is enough to show that $\psi$ is an isomorphism after passing to an \'{e}tale cover and hence we may assume that $\X\simeq[X/D(A)]$ for a ramified cover $X\to S$ and that $\lambda$ comes from a character of $D(A)$, which we think of as an element of $A$. We have a tautological sheaf $\oj_\X[\lambda]$ and \[\Ei\simeq \pi^*\pi_*(\oj_\X[-\lambda]\otimes \Ei)\otimes\oj_\X[\lambda]\,.\]
Hence $\Ei$ is generated as an $\oj_\X$-module in degree $-\lambda$ and any morphism $\Ei\to \Ei'$ is completely determined by what it does in degree $-\lambda$. By the projection formula we get $\pi_*\Ei\simeq \pi_*(\oj_\X[-\lambda]\otimes \Ei)\otimes\Li_\lambda\simeq \Ei_{-\lambda}\otimes\Li_\lambda$ and hence
\[\begin{split}\HOM(\pi_*\Ei,\pi_*\Ei') & \simeq \HOM(\Ei_{-\lambda}\otimes\Li_{\lambda},\Ei'_{-\lambda}\otimes\Li_{\lambda}) \\ & \simeq \HOM(\Ei_{-\lambda},\Ei'_{-\lambda})\,.\end{split}\]
Hence we conclude that $\Ei\to \Ei'$ is completely determined by $\pi_*\Ei\to \pi_*\Ei'$. On the other hand, since $\Ei$ and $\Ei'$ are both generated in degree $-\lambda$, any morphism $\pi_*\Ei\to \pi_*\Ei'$ induces a morphism $\Ei\to \Ei'$. Hence we conclude that \[\Hom(\Ei,\Ei')\simeq \Hom(\pi_*\Ei,\pi_*\Ei')\,.\qedhere\]
\end{proof}

\begin{cor}\label{cor:D-pic}
    Let $\X\to S$ be a stacky cover. Then the morphism
    \[\PIC_{\X/S}^{triv}\to \Pic_{\X/S}\,,\] sending a pair $(\Ei,t)$ to the class $[\Ei]$, is an equivalence of sheaves of sets.
\end{cor}

\begin{proof}
    Let $U\to S$ be \'etale. Injectivity follows from Lemma \ref{lem:unique} since two line bundles on $\X_U$ which differ by a line bundle coming from $U$ must have the same character. Surjectivity follows from the fact that if $\Li$ is any line bundle on $\X_U$, then $((\pi^*\pi_*\Li)^\vee\otimes \Li, \varepsilon)$ represents an element in $\PIC_{\X/S}^{triv}(U)$ and $[(\pi^*\pi_*\Li)^\vee\otimes \Li]=[\Li]$\,.
\end{proof}

\subsection*{The canonical DF-datum}
Now we are in a position to define a Deligne--Faltings datum associated to $\X$. The goal is to obtain a pointed section $\Pic_{\X/S}\to \DIV_{\X/S}$, which by Proposition \ref{prop:equiv-coc-ext} would give us a 2-cocycle
\[\Pic_{\X/S}\times \Pic_{\X/S}\to \DIV_S\]
as desired.
The functor
\[
  \Lambda\colon \PIC_{\X/S}^{triv}\to \DIV_{\X/S}
\]
of Definition \ref{df:LS} is fully faithful.

Choose a quasi-inverse $\psi\colon \Pic_{\X/S}\xrightarrow{\sim} \PIC_{\X/S}^{triv}$ of the morphism of Corollary \ref{cor:D-pic}.
We denote by $\zeta_\X$ the morphism
\[
  \zeta_\X\colon \Pic_{\X/S}\xrightarrow{\psi} \PIC_{\X/S}^{triv}\xrightarrow{\Lambda} \DIV_{\X/S}\,.
\]

\begin{df}
The exact sequence
  \[
    0\to \DIV_S\xrightarrow{\pi^*} \DIV_{\X/S}\to \Pic_{\X/S}\to 0
  \]
together with the section $\zeta_\X$ defines a free extension of $\Pic_{\X/S}$ by $\DIV_S$ (Definition \ref{df:ext-cat}) which we call the \emph{free extension associated to }$\X$. The corresponding 2-cocycle 
\[
    f_\X=\pi_*(\zeta_\X(-)\otimes\zeta_\X(-))\colon \Pic_{\X/S}\times \Pic_{\X/S}\to \DIV_S
\]
is referred to as the \emph{2-cocycle associated to} $\X$. 
\end{df}

\begin{lem}\label{lem:kappa}
    There exists a canonical natural isomorphism 
        \[
            \kappa_\X\colon \zeta_\X\otimes \zeta_\X\simeq f_\X\otimes (\zeta\circ\Sigma)\,.    
        \]
\end{lem}

\begin{proof}
    This follows from the fact that $\zeta_\X(\lambda)\otimes \zeta_\X(\lambda')$ and $f_\X(\lambda,\lambda')\otimes \zeta(\lambda+\lambda')$ have the same character and since there is a canonical isomorphism between them, via the using the projection formula, after pushing down to $S$. 
\end{proof}




The pair $(\zeta_\X, \kappa_\X)$ defines a morphism $\X\to S_{\A,f_\X}$ which is in fact an equivalence. 

\begin{thm}\label{thm:main}
The morphism \[\X\to S_{\A,f_\X}\] given by the pair $(\zeta_\X, \kappa_\X)$ is an isomorphism of stacks, where $\A=\Pic_{\X/S}$. Hence there exists a canonical (up to canonical isomorphism) symmetric monoidal functor $\Li\colon P_\A\to \DIV_{S_\et}$, and a canonical isomorphism of stacks $\X\to S_{(\A,\Li)}$ where $S_{(\A,\Li)}$ is the root stack associated to the building datum $(\A,\Li)$.
\end{thm}

\begin{proof}
We may work locally on $S$ and hence we may assume that $\X$ is a quotient by a ramified cover under an action of a finite diagonalizable group (Remark \ref{rk:local-structure}). By Remark \ref{rk:counit} and Proposition \ref{lem:local} we conclude that $\eta$ is an isomorphism.
\end{proof}

\subsection*{The Cartier dual of the inertia stack}
We saw in Lemma \ref{lem:unique} that any two line bundles with the same character must differ by a line bundle from $S$. 
Now we will see that when $\X\to S$ is a stacky cover which is \emph{Deligne--Mumford}, then $\Pic_{\X/S}\cong \pi_*D(\I_\X)$. This does \emph{not} hold in general if $\X$ is not Deligne--Mumford and a counter-example is obtained by taking $\X$ to be the quotient stack of the ramified cover given in Remark \ref{rk:important-example}. 


To show that there exists a line bundle with character $\lambda$ for every character $\lambda$ of $\I_\X$ when $\X$ is Deligne--Mumford, we show that it exists \'{e}tale locally (on the base) and that these local bundles may be glued to a global one.

\begin{lem}\label{lem:exist}
Let $\lambda$ be a character of $\I_\X$. Then there exists a line bundle on $\X$ on which $\I_\X$ acts with character $\lambda$.
\end{lem}

\begin{proof}
Let $\{U_i\to S\}$ be an \'{e}tale cover such that every $\X_i:=\X\times_S{U_i}$ may be written as a quotient $\X_i\simeq [X_i/D(A_i)]$ of a ramified cover (Remark \ref{rk:local-structure}) with the property that the canonical morphism $\oj_{X_i}[A_i]\to \oj_{X_i}[\stab(X_i)]$ is an isomorphism on group-like elements (see Proposition \ref{prop:gp-like-iso}) since $\X$ is Deligne--Mumford. The character $\lambda$ pulls back to a character for each $\stab(X_i)$ which we may think of as an element $\lambda_i\in A_i$. Write $\X_i=U_i\times_S\X$. The line bundle $\oj_{\X_i}[\lambda_i]$
has character $\lambda$.

Now consider the two objects $(\Ei_i,t_i)$ and $(\Ei_j,t_j)$, where
\[\begin{split}\Ei_i & =(\pi^*\pi_*\oj_{\X_i}[\lambda_i])^\vee\otimes \oj_{\X_i}[\lambda_i]\, \mbox{ and } \\ \Ei_j & =(\pi^*\pi_*\oj_{\X_j}[\lambda_j])^\vee\otimes \oj_{\X_j}[\lambda_j]\end{split}\]
(here we write $\pi$ for both maps $\X_i\to U_i$ and $\X_j\to U_j$) and $t_i$ is the trivialization \[\begin{split}\oj_S & \cong (\pi_*\oj_{\X_i}[\lambda_i])^\vee\otimes \pi_*\oj_{\X_i}[\lambda_i] \\ & \cong \pi_*((\pi^*\pi_*\oj_{\X_i}[\lambda_i])^\vee\otimes \oj_{\X_i}[\lambda_i])\,,\end{split}\]
where the second isomorphism is the projection morphism.
Both $\Ei_i$ and $\Ei_j$ pulls back to a line bundle on $\X\times_S(U_i\times_SU_j)$ with character $\lambda$ and the trivializations $t_i$ and $t_j$ pulls back to trivializations. Hence Lemma \ref{lem:unique} gives a unique isomorphism \[\Ei_i|_{\X\times_S(U_i\times_SU_j)}\cong \Ei_j|_{\X\times_S(U_i\times_SU_j)}\] whose pushforward sends (the pullback of) $t_i$ to $t_j$. This implies that the line bundles $\Ei_i$ will glue to a global line bundle on $\X$ with character $\lambda$ and trivial pushforward.
\end{proof}

\begin{cor}
    If $\X\to S$ is a stacky cover which is Deligne--Mumford, then $\Pic_{\X/S}\cong \pi_*D(\I_\X)$.
\end{cor}

\input{nodal.tex}

%% file: nodal.tex
\begin{ex}[Stack ramified over a nodal cubic] The following is an example of a stacky cover which is not globally a quotient of a ramified cover (but of course \'etale locally).
Consider the nodal cubic $V(y^2-x^2(x+1))\subset \spec \mathbb{C}[x,y]=\mathbb{A}^2_{\mathbb{C}}=:S$. Let $\pi\colon \X\to S$ be the 2nd root stack associated to the log structure given by $f=y^2-x^2(x+1)\in \oj_S$ and let $\A=\pi_*D(\I_\X)$. More precisely, let
$j\colon U\hookrightarrow S$ be the complement of $V(f)\hookrightarrow S$ and consider the canonical log structure $\oj_S\cap j_*\oj_U^\times\to \oj_S$. Let $\Li\colon\PP\to \DIV_{S_\et}$ be the associated Deligne--Faltings structure (see Remark \ref{rk:log}), put $\Q=\PP$, and let $\gamma\colon \PP\to \Q$ be given by multiplication by 2.
We put $\X=S_{(\PP,\Q,\Li)}$.

We have an \'{e}tale cover $U_1,U_2\to S$ where $U_1$ and $U_2$ are given as the spectrum of \[\begin{split} R_1 & =\CC[x,x^{-1},y]\, \mbox{ and} \\ R_2 & = \CC[x,(x+1)^{-1},y,w]/(w^2-(x+1))\end{split}\] respectively.
We have that $\X_{U_1}\simeq [X_1/D(\ZZ/2\ZZ)]$ and $\X_{U_2}\simeq [X_2/D(\ZZ/2\ZZ\times \ZZ/2\ZZ)]$ respectively, where
\[\begin{split}X_1 & = R_1[z]/(z^2-f)\\ X_2 & = R_2[u,v]/(u^2-f_1,v^2-f_2)\,,\end{split}\]
and $f_1=y-xw$ and $f_2=y+xw$. Over $U_1\times_SU_2$ the diagonal $\Delta\colon \ZZ/2\ZZ\to \ZZ/2\ZZ\times\ZZ/2\ZZ$ gives a $\Delta$-equivariant morphism \[\begin{split}X_1\times_{U_1}(U_1\times_SU_2) & \to X_2\times_{U_2}(U_1\times_SU_2)\end{split}\] inducing an isomorphism
\[[X_1\times_{U_1}(U_1\times_SU_2)/D(\ZZ/2\ZZ)]\simeq [X_2\times_{U_2}(U_1\times_SU_2)/D(\ZZ/2\ZZ\times\ZZ/2\ZZ)]\,.\]
We have
  \[\begin{split}P_{\ZZ/2\ZZ}\cong \N\xrightarrow{(1,0,1)}P_{\ZZ/2\ZZ\times\ZZ/2\ZZ} & =\N^{\ZZ/2\ZZ\times\ZZ/2\ZZ\times \ZZ/2\ZZ\times\ZZ/2\ZZ}/R \\ & \cong \N_{(1,1),(1,0)}\times \N_{(1,0),(0,1)}\times \N_{(0,1),(1,1)}\end{split}\]
and a commutative diagram
  \[\begin{tikzcd}\N\ar{rr}{(1,0,1)}\ar{dr}{\alpha} && \N^3\ar{dl}{\beta}\\ & (P_\A)|_{U_1\times_SU_2} & \end{tikzcd}\]
where $\alpha_{D(f_i)}(1)$ is the non-trivial element for $i=1,2$ whereas
  \[\begin{split}\beta_{D(f_1)}(1,0,0) & = 0 \\ \beta_{D(f_1)}(0,1,0) & = 0 \\ \beta_{D(f_1)}(0,0,1) & = 1 \\ \beta_{D(f_2)}(1,0,0) & = 1 \\ \beta_{D(f_2)}(0,1,0) & = 0 \\ \beta_{D(f_2)}(0,0,1) & = 0\,. \end{split}\]
(This corresponds to projection onto the second and first $\ZZ/2\ZZ$ factor respectively).
Note that $\Gamma(U_1\times_SU_2,\GI)\cong \ZZ/2\ZZ\times\ZZ/2\ZZ$ but $\Gamma(U_1\times_SU_2,P_\A)\cong \N$ (!). The canonical cover $\ZZ/2\ZZ\to \GI\times_S(U_1\times_SU_2)$ is an epimorphism of \'{e}tale sheaves which is not surjective on global sections.

From here we see that both $P_{\ZZ/2\ZZ}\to \DIV U_1\times_SU_2$ and $P_{\ZZ/2\ZZ\times\ZZ/2\ZZ}\to \DIV U_1\times_SU_2$ factor through $\N^{\A^2}/R\times_S(U_1\times_SU_2)$ and hence we can glue to a global symmetric monoidal functor \[\Li_\X\colon\N^{\A^2}/R\to \DIV_{S_\et}\,.\] Note that $\Gamma(S,P_\A)\cong \N$
and hence there cannot be a global chart for this log/DF structure.
\end{ex}

%% file: building-data.tex
\section{Building data for stacky covers}\label{sec:building-data}
The goal of this subsection is to describe the 2-category of stacky covers via \emph{stacky building data}.

\begin{df}
Let $\A$ be an \'etale sheaf of abelian groups of finite type on a scheme $S$ and $\Li\colon P_\A\to \DIV_{S_\et}$ a symmetric monoidal functor. Then we define $\A^\perp\subseteq \A$ to be the subsheaf (of sets) defined by \[\A^\perp(U)=\{\lambda\in \A(U):\Li(e_{\lala})\simeq (\oj_S,1)\,,\forall \lambda'\in \A(U)\}\,,\] for every \'etale $U\to S$.
\end{df}

\begin{lem}
The subsheaf $\A^\perp$ is a subgroup.
\end{lem}

\begin{proof}
Let us write $A=\A(U)$ for an \'etale $U\to S$ and similarly for $\A^\perp$. Clearly $0\in A^\perp$. Let $\lambda\in A^\perp$ and write $s_\lala$ for the global section of $\Li_U(e_{\lala})$. Then for all $\lambda'\in A$ we have
\[\begin{split}
s_{-\lambda,\lambda'} & \simeq s_{-\lambda,\lambda'}s_{\lambda'-\lambda,\lambda} \\ & \simeq s_{\lambda,-\lambda}s_{0,\lambda'} \\ & \simeq s_{\lambda,-\lambda} \\ & \simeq 1\,.
\end{split}\]
Hence $-\lambda\in A^\perp$.
If $\lambda_1,\lambda_2\in A^\perp$, then $s_{\lambda_1,\lambda_1'}$ and $s_{\lambda_2,\lambda_2'}$ are invertible for every $\lambda_1',\lambda_2'\in A$. In particular, $s_{\lambda_1,\lambda_2}$, $s_{\lambda_2,\lambda}$, and $s_{\lambda_2+\lambda,\lambda_1}$ are all invertible for every $\lambda\in A$.
The identity \[s_{\lambda_1,\lambda_2}s_{\lambda_1+\lambda_2,\lambda}=s_{\lambda_2,\lambda}s_{\lambda_2+\lambda,\lambda_1}\] then implies that $s_{\lambda_1+\lambda_2,\lambda}$ is invertible for every $\lambda\in A$.
Hence $A^\perp$ is a subgroup of $A$.
\end{proof}

\begin{df}\label{df:stacky-building-data}
Let $\STDATA_S$ be the (2,1)-category of \emph{stacky building data on} $S$ which consists of the following data:
\begin{enumerate}
  \item its objects consists of pairs $(\A,\Li)$ where
  \begin{itemize}
    \item $\A$ is an \'etale sheaf of abelian groups of finite type on $S$, and
    \item $\Li\colon P_\A\to \DIV_{S_\et}$ is a symmetric monoidal functor,
  \end{itemize}
  and these satisfy the condition that the subsheaf \[\A^\perp=\{\lambda\in \A: \Li_{\lambda,\lambda'}\simeq (\oj_S,1)\,, \forall \lambda'\in \A\}=0\,.\]
  We refer to the pair $(\A,\Li)$ as a \emph{stacky building datum} or just a \emph{building datum}.
  \item for each pair of objects $(\A,\Li)\to (\A',\Li')$ a 1-morphism $(\A,\Li)\to(\A',\Li')$ is a triple $(\varphi,\Hj,\tau)$ consisting of
  \begin{itemize}
    \item a morphism $\varphi\colon \A'\to \A$,
    \item a symmetric monoidal functor $\Hj\colon Q_{\A'}\to \DIV_{S_\et}$, and
    \item an isomorphism $\tau\colon (\Hj\circ \gamma_{\A'})\otimes(\Li\circ\varphi_P)\simeq \Li'$ of symmetric monoidal functors, where $\varphi_P\colon P_{\A'}\to P_\A$ is the morphism induced by $\varphi$.
  \end{itemize}
  \item a 2-morphism
  \[\begin{tikzcd}(\A,\Li)
  \arrow[bend left=50]{r}[name=U]{(\varphi,\Hj,\tau)}
  \arrow[bend right=50]{r}[name=D,below]{(\varphi',\Hj',\tau')}
  & (\A',\Li')
  \arrow[Rightarrow, to path=(U) -- (D), shorten <>=10pt]{}
  \end{tikzcd}\]
  is given by
  \begin{itemize}
    \item an equality $\varphi'=\varphi$, and
    \item a natural isomorphism \[\theta\colon \Hj\simeq \Hj'\] such that
    the induced diagram
    \[\begin{tikzcd} (\Hj\circ\gamma)\otimes (\Li\circ\varphi_P)\ar{rr}{(\theta\circ\gamma)\otimes\id}[swap]{\simeq}\ar{dr}{\tau} & & (\Hj'\circ\gamma)\otimes (\Li\circ\varphi_P)\ar{dl}[swap]{\tau'} \\ & \Li' & \end{tikzcd}\]
    commutes.
  \end{itemize}
  \item The composition bifunctor \[\Map((\A,\Li),(\A',\Li'))\times \Map((\A',\Li'),(\A'',\Li''))\to \Map((\A,\Li),(\A'',\Li''))\] sends a pair of objects $(\varphi_1, \Hj_1, \tau_1), (\varphi_2, \Hj_2, \tau_2)$ to $(\varphi_1\circ \varphi_2, \Hj_2\otimes(\Hj_1\circ (\varphi_2)_Q), \tau_3)$, where $\tau_3$ is the isomorphism 
    \[
        \begin{split}
                  & ((\Hj_2\otimes(\Hj_1\circ (\varphi_2)_Q))\circ \gamma_{\A'})\otimes (\Li\circ(\varphi_1)_P\circ (\varphi_2)_{P})  \\
          \simeq\  & (\Hj_2\circ \gamma_{\A''})\otimes (\Hj_1\circ(\varphi_2)_Q\circ\gamma_{\A''})\otimes (\Li\circ(\varphi_1)_P\circ (\varphi_2)_{P}) \\
          \simeq\  & (\Hj_2\circ \gamma_{\A''})\otimes (\Hj_1\circ\gamma_{\A'}\circ(\varphi_2)_P)\otimes (\Li\circ(\varphi_1)_P\circ (\varphi_2)_{P}) \\
          \simeq\  & (\Hj_2\circ \gamma_{\A''})\otimes (((\Hj_1\circ\gamma_{\A'})\otimes (\Li\circ (\varphi_1)_P))\circ(\varphi_2)_P) \\ 
          \simeq\  & (\Hj_2\circ \gamma_{\A''})\otimes (\Li'\circ(\varphi_2)_P) \\ \simeq\ & \Li'' \,,
        \end{split}
    \]
  where the the first three isomorphisms are canonical, the second to last isomorphism is induced by $\tau$, and the last isomorphism is given by $\tau'$. The composition bifunctor is defined on morphisms as follows: A pair of 2-morphisms  
    \[
      \begin{tikzcd}(\A,\Li)
      \arrow[bend left=50]{r}[name=U]{(\varphi_1,\Hj_1,\tau_1)}
      \arrow[bend right=50]{r}[name=D,below]{(\varphi'_1,\Hj'_1,\tau'_1)}
      & (\A',\Li')
      \arrow[Rightarrow, to path=(U) -- (D), shorten <>=10pt]{}
      \arrow[bend left=50]{r}[name=V]{(\varphi_2,\Hj_2,\tau_2)}
      \arrow[bend right=50]{r}[name=E,below]{(\varphi'_2,\Hj'_2,\tau'_2)}
      & (\A'',\Li'')
      \arrow[Rightarrow, to path=(V) -- (E), shorten <>=10pt]{}
      \end{tikzcd}
    \]
  is sent to the 2-morphism 
  \[
    \begin{tikzcd}(\A,\Li)
    \arrow[bend left=50]{r}[name=U]{(\varphi_1\circ \varphi_2, \Hj_2\otimes(\Hj_1\circ (\varphi_2)_Q), \tau_3)}
    \arrow[bend right=50]{r}[name=D,below]{(\varphi'_1\circ \varphi'_2, \Hj'_2\otimes(\Hj'_1\circ (\varphi'_2)_Q), \tau'_3)}
    & (\A'',\Li'')
    \arrow[Rightarrow, to path=(U) -- (D), shorten <>=10pt]{}
    \end{tikzcd}
  \]
  given by the equality $\varphi_1\circ \varphi_2=\varphi'_1\circ \varphi'_2$ and the natural isomorphism 
    \[
      (\varphi_1\circ \varphi_2, \Hj_2\otimes(\Hj_1\circ (\varphi_2)_Q), \tau_3)\simeq (\varphi'_1\circ \varphi'_2, \Hj'_2\otimes(\Hj'_1\circ (\varphi'_2)_Q), \tau'_3)  
    \]
  induced by the natural isomorphisms $\Hj_1\simeq \Hj'_1$ and $\Hj_2\simeq \Hj'_2$. 

  The definition of the identity $I\to \Map((\A,\Li),(\A,\Li))$ is left to the reader.
\end{enumerate}
\end{df}

\begin{rk}
  The structure of $\STDATA_S$ is motivated by the fact that we want an equivalence of (2,1)-categories $\STCOV_S\simeq \STDATA_S$ (see Theorem \ref{thm:build}).
  Note that there is an alternative definition on $\STDATA_S$ using the language of 2-cocycles in symmetric monoidal stacks developed in Section \ref{sec:2-coc-in-sym-mon}. The objects would then be pairs $(\A, f)$ where $\A$ is an \'etale sheaf of abelian groups of finite type on $S$ and $f\colon \A\times \A\to \DIV_{S_\et}$ is a 2-cocycle (Definition \ref{df:weak-2-cocycle}). For two objects $(\A, \Li)$ and $(\A', \Li')$ corresponding to $(\A, f)$ and $(\A', f')$, the symmetric monoidal functor $\Hj$ present in Definition \ref{df:stacky-building-data} then corresponds to a 2-coboundary $f'\to f$ (Definition \ref{df:weak-2-cocycle}). Similarly, 2-morphisms in $\STDATA_S$ can be thought of as 2-morphisms of 2-coboundaries (Definition \ref{df:weak-2-cocycle}).    
  Defining $\STDATA_S$ using the language of 2-cocycles is somehow conceptually more clear but we still chose to use Borne--Vistoli's language of DF-structures since this was already present and perhaps easier to digest for people interested in this subject. 
\end{rk}

\begin{rk}
  There is also a stack-version of $\STDATA_S$ which is not hard to write down using Definition \ref{df:stacky-building-data}. 
\end{rk}

We write $S_{(\A,\Li)}=S_{(P_\A,Q_\A,\Li)}$ and we denote by $\Pic_{\X/S}^{triv}$ the sheaf obtained from $\PIC_{\X/S}^{triv}$ by identifying isomorphic objects.

\begin{prop}\label{prop:build-st-cov}
Let $(\A,\Li)$ be a stacky building datum and $\pi\colon \X=S_{\A,\Li}\to S$ the associated root stack. Then $\X\to S$ is a stacky cover and we have a canonical isomorphism $\A\simeq \Pic_{\X/S}$.
\end{prop}

\begin{proof}
There is an \'etale neighborhood $\bar{s}\in U$, a finite abelian group $A$, and an epimorphism $A\to \A_U$ inducing an isomorphism $A\cong \A_{\bar{s}}$.
This implies that we have a chart
\[
\begin{tikzcd}
P_A\ar{r}\ar{d}{{\gamma_A}} & P_\A|_U\ar{d} \\
Q_A\ar{r} & Q_\A|_U
\end{tikzcd}
\]
and hence $\X_U\simeq S_{Q_A/P_A}$. The stack $S_{Q_A/P_A}$ sits in a Cartesian diagram
\[
\begin{tikzcd}
  S_{Q_A/P_A}\ar{r}\ar{d} & {[\spec\ZZ[Q_A]/D(Q_A^{gp})]}\ar{d} \\
  S\ar{r}{{\Li}} & {[\spec\ZZ[P_A]/D(P_A^{gp})]}
\end{tikzcd}
\]
and if we pull back along the canonical $D(P_A^{gp})$-torsor $\spec\ZZ[P_A]\to [\spec\ZZ[P_A]/D(P_A^{gp})]$, we get a Cartesian diagram
\[
\begin{tikzcd}
  S_{Q_A/P_A}\times_{[\spec\ZZ[P_A]/D(P_A^{gp})]}{\spec\ZZ[P_A]}\ar{r}\ar{d} & {[\spec\ZZ[Q_A]/D(A)]}\ar{d} \\
  S\times_{[\spec\ZZ[P_A]/D(P_A^{gp})]}{\spec\ZZ[P_A]}\ar{r} & {\spec\ZZ[P_A]}\,.
\end{tikzcd}
\]
By Remark \ref{rk:mon-ram-cov}, $\spec\ZZ[Q_A]\to \spec\ZZ[P_A]$ is a $D(A)$-cover. This means that $\X$ is fppf locally the quotient of a ramified cover and hence a stacky cover.

Let $\Ei\colon Q_{\pi^*\A}\to \DIV_{\X_\et}$ be the universal DF-object. We have a canonical set-theoretic section $\iota\colon \A\to Q_{\A}$ sending a local section $\lambda$ to $(0,\lambda)$, and we define a morphisms of sheaves of sets $\beta'\colon \A\to \PIC_{\X/S}^{triv}$ by sending $\lambda$ to the class represented by $(\Ei_\lambda,\pi_*\varepsilon_\lambda)$ and let $\beta\colon \A\to \Pic_{\X/S}$ be the composition of $\beta'$ with the canonical morphism $(\Ei_\lambda,\pi_*\varepsilon_\lambda)\mapsto [\Ei_\lambda]$ in Corollary \ref{cor:D-pic}. 
To see that $\beta$ is an equivalence we may work \'etale locally on the base and hence assume that $\X=[X/D(A)]\simeq S_{(A,\Li)}$.
Then $\beta$ is an epimorphism since every line bundle on $\X$ is of the form $\Li_0\otimes \Ei_\lambda$, where $\Li_0$ is the pullback of a line bundle on $S$ and $\Ei_\lambda$ is the universal line bundle associated to $\lambda\in \A$. We also see that $\beta$ is a homomorphism since $\Ei_\lambda$ and $\oj_\X[\lambda]$ differ by a line bundle coming from $S$ and hence they define the same class in $\Pic_{\X/S}$. 
Furthermore, since $\varepsilon_\lambda\colon \oj_\X\to \Ei_\lambda$ is given by $s_{\lambda',\lambda}$ in degree $\lambda'$, we get that $\varepsilon_\lambda$ is an isomorphism if and only if $\lambda\in \A^\perp=0$. This means that zero is the unique element mapped to the trivial element $(\oj_\X, \pi_* 1)$ by $\beta'$. Hence zero is the unique element mapping to zero by $\beta$ and we conclude that $\beta$ is an isomorphism. 
\end{proof}

\begin{rk}
There is a monoidal structure on $\STDATA_S$ which on objects is defined by
\[(\A,\Li)\otimes (\A',\Li'):=(\A\oplus\A',\Li\otimes \Li')\,.\]
\end{rk}

\begin{thm}\label{thm:build}
There exists an equivalence of (2,1)-categories \[\STCOV_S\simeq \STDATA_S\] between the 2-category of stacky covers over $S$ and the 2-category of stacky building data on $S$.
\end{thm}

\begin{proof}
We define \[\Phi\colon \STDATA_S\to \STCOV_S\] on objects by sending $(\A,\Li)$ to the corresponding root stack $\X=S_{(\A,\Li)}$. Hence it is clear from Theorem \ref{thm:main} that $\Phi$ will be essentially surjective. For building data $D=(\A,\Li)$ and $D'=(\A',\Li')$ the functor
\[\Phi_{D,D'}\colon \Map((\A,\Li),(\A',\Li'))\to \Map(S_{(P_\A,Q_\A,\Li)},S_{(P_{\A'},Q_{\A'},\Li')})\]
is defined as follows. Put $\pi\colon \X=S_{(\A,\Li)}\to S$ and $\pi'\colon\Y=S_{(\A',\Li')}\to S$ and let $(\Ei,\alpha)$ be the universal object on $\X$ and $(\Ei',\alpha')$ the universal object on $\Y$.

Let $(\varphi, \Hj, \tau)\colon D=(\A,\Li)\to D'=(\A',\Li')$ be a 1-morphism. A morphism $f\colon\X\to \Y$ is completely determined by the pullback of the universal diagram on $\Y$ to $\X$ along $f$. That is, we need to define $f^*\Ei'$ and $f^*\alpha'$ in the diagram
  \[
    \xymatrix{\pi^*P_{\A'}\ar[r]^{\pi^*\Li'}\ar[d]_{\gamma_{\A'}} \ar@{}[dr] |(.3){f^*\alpha'}  & \DIV_{\X_{\et}} \\ 
    \pi^*Q_{\A'} \ar[ur]_{f^*\Ei'} &  \,.}
  \]
Put \[f^*\Ei':=\pi^*\Hj\otimes (\Ei\circ \varphi_Q)\] and define $f^*\alpha'$ as the composition \[f^*\Ei'\circ\gamma_{\A'}=(\pi^*\Hj\circ\gamma_{\A'})\otimes(\Ei\circ\varphi_Q\circ\gamma_{\A'})\xrightarrow[\sim]{\alpha} (\pi^*\Hj\circ\gamma_{\A'})\otimes(\pi^*\Li\circ\varphi_P)\xrightarrow{\pi^*\tau}\pi^*\Li'\,.\]
This defines a morphism of stacks $f_{(\varphi, \Hj, \tau)}\colon\X\to \Y$ and we define $\Phi$ on 1-morphisms by
\[\Phi_{D,D'}(\varphi, \Hj, \tau)=f_{(\varphi, \Hj, \tau)}\,.\]

Now assume that we have a 2-morphism $\theta$:
\[\begin{tikzcd}(\A,\Li)
\arrow[bend left=50]{r}[name=U]{(\varphi,\Hj,\tau)}
\arrow[bend right=50]{r}[name=D,below]{(\varphi,\Hj',\tau')}
& (\A',\Li')\,.
\arrow[Rightarrow, to path=(U) -- (D), shorten <>=10pt]{}
\end{tikzcd}\]
Write $f=\Phi_{D,D'}(\varphi,\Hj,\tau)$ and $g=\Phi_{D,D'}(\varphi,\Hj',\tau')$. Then we get an isomorphism
\[f^*\Ei'\simeq \pi^*\Hj\otimes (\Ei\circ\varphi_Q)\xrightarrow{\pi^*\theta\otimes\id}\pi^*\Hj'\otimes(\Ei\circ\varphi_Q)\simeq g^*\Ei'\] such that the induced diagram
\[\xymatrix{ (f^*\Ei'\circ\gamma_{\A'})\ar[rr]^{\simeq}\ar[dr]_{f^*\alpha'} & & (g^*\Ei'\circ\gamma_{\A'})\ar[dl]^{g^*\alpha'} \\ & \pi^*\Li' &}\] commutes.
This means that we get a well-defined natural transformation
\[\begin{tikzcd}\X
\arrow[bend left=50]{r}[name=U]{f}
\arrow[bend right=50]{r}[name=D,below]{g}
& \Y\,.
\arrow[Rightarrow, to path=(U) -- (D), shorten <>=10pt]{}
\end{tikzcd}\]

To show that $\Phi_{D,D'}$ is an equivalence of categories we define a quasi-inverse $\Psi$. Let $f\colon \X\to \Y$ be a morphism of stacks, where $\X$ and $\Y$ are root stacks constructed from building data $(\A,\Li)$ and $(\A',\Li')$ with universal Deligne--Faltings objects
\[\Ei\colon Q_\A\to \DIV_{\X_\et}\,, \quad \Ei'\colon Q_{\A'}\to \DIV_{\Y_\et}\,.\]
By pullback we get a morphism $\Pic_{\Y/S}\to \Pic_{\X/S}$ and by Proposition \ref{prop:build-st-cov} a morphism $\varphi\colon\A'\to \A$.
We may pullback the universal diagram on $\Y$ to get a diagram \[\xymatrix{\pi^*P_{\A'}\ar[r]^{\pi^*\Li'}\ar[d]_{\gamma_{\A'}} \ar@{}[dr] |(.3){f^*\alpha'} & \DIV_{\X_{\et}} \\ \pi^*Q_{\A'} \ar[ur]_{f^*\Ei'} & \,.}\]
There is a unit morphism $\Ei'_\lambda\to f_*f^*\Ei'_\lambda$ and applying $\pi'_*$ we get for $\lambda\in\A'$
\[\oj_S\simeq\pi'_*\Ei'_\lambda\to\pi'_*f_*f^*\Ei'_\lambda\simeq \pi_*f^*\Ei'_\lambda\,.\] Now apply $\pi^*$ to get a global section \[\oj_\X\to \pi^*\pi_*f^*\Ei'_\lambda\,.\]
This implies that we have a canonical morphism
\[(\Ei\circ\varphi_Q)_\lambda\simeq(\pi^*\pi_*f^*\Ei'_\lambda)^\vee\otimes f^*\Ei'_\lambda\to \oj_\X\otimes f^*\Ei'_\lambda\simeq f^*\Ei'_\lambda\,,\] which we may view as a global section $h\in \Gamma(S,\Hj_\lambda)$ where
\begin{equation}\label{eq:H-def}\pi^*\Hj_\lambda:=f^*\Ei'_\lambda\otimes (\Ei\circ\varphi_Q)_\lambda^\vee\,.\end{equation}
Since the inertia acts trivially on each $\pi^*\Hj_\lambda$, the symmetric monoidal functor $\pi^*\Hj=f^*\Ei'\otimes (\Ei\circ\varphi_Q)^\vee$ descends to a symmetric monoidal functor
$\Hj\colon Q_{\A'}\to \DIV_{S_\et}\,.$
We also have an isomorphism \[\tau\colon(\Hj\circ \gamma_{\A'})\otimes(\Li\circ\varphi_P)\simeq \Li'\] by descending the isomorphism
\[(\pi^*\Hj\circ\gamma_{\A'})\otimes (\Ei\circ\varphi_Q\circ\gamma_{\A'})\simeq f^*\Ei'\circ \gamma_{\A'}\xrightarrow{f^*\alpha'}\pi^*\Li'\,.\]
This defines the quasi-inverse on objects by putting $\Psi(f)=(\varphi,\Hj,\tau)$.

To define $\Psi$ on morphisms, suppose that we have a natural transformation $\eta\colon f\to g$ sitting in a diagram
\[\begin{tikzcd}\X
\arrow[bend left=50]{r}[name=U]{f}
\arrow[bend right=50]{r}[name=D,below]{g}
& \Y\,,
\arrow[Rightarrow, to path=(U) -- (D), shorten <>=10pt]{}
\end{tikzcd}\]
where $\X$ and $\Y$ are root stacks constructed from building data $(\A,\Li)$ and $(\A',\Li')$ with universal Deligne--Faltings objects
\[\Ei\colon Q_{\pi^*\A}\to \DIV_{\X_\et}\,, \quad \Ei'\colon Q_{\pi^*\A'}\to \DIV_{\Y_\et}\,,\]
$f=\Phi_{D,D'}(\varphi,\Hj,\tau)$, and $g=\Phi_{D,D'}(\varphi',\Hj',\tau')$. Then $\eta$ corresponds to a natural isomorphism
\[\eta\colon f^*\Ei'\to g^*\Ei'\] and from the construction of $\pi^*\Hj$ and $\pi^*\Hj'$ as in Equation (\ref{eq:H-def}), we see that we get a natural isomorphism $\theta\colon \Hj\to\Hj'$ such that
\[\begin{tikzcd} (\Hj\circ\gamma)\otimes (\Li\circ\varphi_P)\ar{rr}{(\theta\circ\gamma)\otimes\id}[swap]{\simeq}\ar{dr}{\tau} & & (\Hj'\circ\gamma)\otimes (\Li\circ\varphi_P)\ar{dl}[swap]{\tau'} \\ & \Li' & \end{tikzcd}\]
commutes. We define $\Psi(\eta)=\theta$.

Now it is straight forward to show that
\[\Psi\circ\Phi_{D,D'}\simeq \id_{\Map((\A,\Li),(\A,\Li'))}\] and
\[\Phi_{D,D'}\circ \Psi\simeq \id_{\Map(S_{(\A,\Li)},S_{(\A',\Li')})}\,.\]
We leave this to the reader.
\end{proof}

\input{cyclotomic-example}

%% file: cyclotomic-example.tex
\begin{ex}\label{ex:cy}
The following is an example where the building datum from a stacky cover is not determined on the global sections of $\A$.
Let $\xi$ be a primitive $5$th root of unity and consider the ring of integers $\ZZ[\xi]$ of the cyclotomic extension $\QQ[\xi]$. The Galois group $G=(\ZZ/5\ZZ)^\times$ acts by
\[\begin{split}G\times \ZZ[\xi] & \to \ZZ[\xi]
\\ (g,\xi) & \mapsto \xi^{g^{-1}}
\end{split}\]
and the corresponding coaction is
\[\begin{split}\ZZ[\xi] & \to \ZZ[\xi]\otimes \ZZ[G]
\\ \xi & \mapsto \sum_{n=1}^4\xi^{n^{-1}}e_n
\end{split}\]
where $e_n$ is the standard basis of idempotents.

If we invert 2 and add a square root $i$ of $-1$ we get an \'etale morphism $\spec \ZZ[2^{-1},i]\to\spec\ZZ$ and we have an isomorphism of Hopf algebras
\[\begin{split}
  \ZZ[2^{-1},i][G] & \to \ZZ[2^{-1},i][T]/(T^4-1) \\
  e_n & \mapsto \frac{1}{4}\sum_{j=0}^3i^{-j\varphi(n)}T^j
\end{split}
\]
where $\varphi$ is the inverse of $\ZZ/4\ZZ\to (\ZZ/5\ZZ)^\times ; \ a\mapsto 2^a$.
Via this isomorphism we get a coaction given on the generator $\xi$ by
\[\begin{split}\ZZ[2^{-1},i][\xi] & \to \ZZ[2^{-1},i][\xi]\otimes \ZZ[2^{-1},i][T]/(T^4-1)
\\ \xi & \mapsto \frac{1}{4}\sum_{n=1}^4\sum_{j=0}^3\xi^{n^{-1}}i^{-j\varphi(n)}T^j\,.
\end{split}\]
After some calculations, one concludes that the induced splitting is
\[
\ZZ[2^{-1},i,\xi]=\ZZ[2^{-1},i]\oplus\langle d_1\rangle\oplus\langle d_2\rangle\oplus\langle d_3\rangle\,,
\]
where
\[
\begin{split}
  d_1 & = (\xi-\xi^4)+(\xi^2-\xi^3)i\,, \\
  d_2 & = (\xi+\xi^4)-(\xi^2+\xi^3)=\sqrt{5}\,, \\
  d_3 & = (\xi-\xi^4)-(\xi^2-\xi^3)i\,, \\
\end{split}
\]
and the multiplication is given by the global sections
\[s_{1,1}=-1-2i\,, \quad s_{1,2}=1+2i\,, \quad s_{1,3}=5\,, \quad s_{2,2}=5\,, \quad s_{2,3}=1-2i\,, \quad s_{3,3}=-1+2i\,.\]
Hence we have a building datum over the \'etale chart $\spec \ZZ[2^{-1},i]\to\spec\ZZ$.

Over the open chart $\spec \ZZ[5^{-1}]\to\spec\ZZ$ we have a trivial building datum. Note that all the sections $s_{i,j}$ above become invertible when inverting $5$. One may check that the two building data glue to a global building datum $(\A,\Li)$ such that $[\spec\ZZ[\xi]/(\ZZ/5\ZZ)^\times]\simeq S_{(\A,\Li)}$.
\end{ex}

%% file: application.tex
\section{Parabolic sheaves and an application}\label{sec:application}
In Section \ref{sec:DF-from-cov} we saw that a $D(A)$-ramified cover gives rise to a Deligne--Faltings datum which we may think of as \emph{ramification data}. In \cite{Biswas-Borne} Biswas--Borne consider ramification data $(\overline{D},\overline{r})$ where $\overline{D}=({D_i})_{i\in I}$ is a
simple normal crossings divisor
on a scheme $S$ over a field $k$, and $\overline{r}=({r_i})_{i\in I}$ is a family of positive integers. Then they give a criterion \cite[Theorem 2.8]{Biswas-Borne} for when this birational building datum comes from a
\emph{tamely ramified $G$-torsor} \cite[Definition 2.2]{Biswas-Borne} where $G$ is a finite abelian group scheme over $k$.
Following \cite[Remark 2.9.(2)]{Biswas-Borne}, we extend this result to the more general setting where
\begin{enumerate}
  \item the ramification datum $(\overline{D},\overline{r})$ is replaced with a \emph{birational building datum} $(\A,\Li)$, i.e., all global sections of $\Li$ are regular, and
  \item \emph{tamely ramified torsor} is replaced with a \emph{tamely ramified cover} (Definition \ref{df:ram-cov}).
\end{enumerate}

\subsection*{Parabolic sheaves} Before stating the theorem, we recall the notion of \emph{parabolic sheaf} in \cite{Borne-Vistoli}.
First we need to define the \emph{category of weights} of a monoid.

\begin{df}
Let $P$ be an integral monoid. We denote by $P^{wt}$ the partially ordered set which is the strict symmetric monoidal category with objects that are elements of $P^{gp}$ and whose arrows $p\colon p'\to p''$ are elements $p\in P$ such that $p'+p=p''$. Similarly, when $\PP$ is a fine sheaf of monoids, we denote by $\PP^{wt}$ the corresponding symmetric monoidal stack.
\end{df}

\begin{rk}
Every symmetric monoidal functor $\Li\colon \PP\to \DIV_{S_\et}$ gives rise to a symmetric monoidal functor $\Li^{wt}\colon \PP^{wt}\to \PIC_{S_{\et}}$ and vice versa (see \cite[Proposition 5.4 and Remark 5.5]{Borne-Vistoli}).
\end{rk}

\begin{df}\label{df:par-sheaf}
Let $(\PP,\Q,\Li)$ be a Deligne--Faltings datum on a scheme $S$. Consider the actions \[\begin{split}\Sigma\colon \PP^{wt}\times \Q^{wt} & \to \Q^{wt}\,, \\ T\colon \PIC_{S_\et}\times\QCOH_{S_\et} & \to \QCOH_{S_\et}   \end{split}\]
where the first one is given by addition and the second by taking tensor products (this means that $\Q^{wt}$ and $\QCOH_{S_\et}$ are module categories over $\PP^{wt}$ and $\PIC_{S_\et}$ respectively).
A \emph{parabolic sheaf} $(E,\rho)$ on $(S,\PP,\Q,\Li)$ consists of
\begin{enumerate}
\item a cartesian functor \[E\colon \Q^{wt}\to \QCOH_{S_\et}\,,\] which we write on objects as $q\mapsto E_q$ and $q'\mapsto E(q')$ on arrows, and
\item an isomorphism 
  \[
    \rho\colon E\circ\Sigma\simeq T\circ (\Li^{wt}\times E)
  \]
realizing $E$ as an $\Li^{wt}$-equivariant functor. 
\end{enumerate}

\begin{rk}
For a more explicit definition of a parabolic sheaf, see \cite[Definition 5.6]{Borne-Vistoli}.
\end{rk}

We denote by $\PAR(S,\PP,\Q,\Li)$ the category of parabolic sheaves on $S$ with respect to the Deligne--Faltings datum $(\PP,\Q,\Li)$.
\end{df}

Let $\pi\colon \X\to S$ be the root stack of a Deligne--Faltings datum $(P,Q,\Li)$, with $P$ and $Q$ constant, and let \[\Ei\colon Q^{wt}\to \PIC \X\] be the universal Deligne--Faltings structure on $\X$. Let $F$ be a quasi-coherent sheaf on $\X$. The corresponding parabolic sheaf $(E,\rho)$ is given as follows:
\begin{enumerate}
\item For $q\in Q^{wt}$, put
\begin{equation}\label{eq:para}E_q=\pi_*(F\otimes\Ei_q)\end{equation}
and for $q'\in Q$, we let $E(q')\colon E_q\to E_{q+q'}$ be the pushforward of the morphism \[F\otimes \Ei_q\simeq F\otimes \Ei_q\otimes \oj_X\xrightarrow{\id\otimes \Ei(q')}F\otimes \Ei_q\otimes\Ei_{q'}\simeq F\otimes \Ei_{q+q'}\,.\]
\item For $p\in P^{wt}$, $q\in Q^{wt}$, we let \[\rho_{p,q}\colon E_{q+p}\simeq \Li_p\otimes E_q\] be the isomorphism obtained via the projection formula \[\Li_p\otimes \pi_*(F\otimes \Ei_q)\simeq \pi_*(F\otimes \Ei_q\otimes \pi^*\Li_p)\] and the isomorphism $\Ei_q\otimes \Ei_p\simeq \Ei_{q+p}$.
\end{enumerate}
This construction may be globalized and one has the following theorem:

\begin{thm}[{{\cite[Theorem 6.1]{Borne-Vistoli}}}]\label{thm:parabolic}
Let $(S,\PP,\Q,\Li)$ be a Deligne--Faltings datum where $\Li\colon \PP\to \DIV_{S_\et}$. There is an equivalence of symmetric monoidal categories \[\PAR(S,\PP,\Q,\Li)\simeq \QCOH(S_{\PP,\Q,\Li})\,.\]
\end{thm}



\subsection*{Ramified $G$-covers}

\begin{prop}
Let $k$ be a field and $S$ a $k$-scheme. Let $\pi\colon \X\to S$ be a stacky cover and $X\to \X$ a $G_\X$-torsor where $X$ is a scheme and $G$ is a finite abelian group scheme over $k$.
Then for every point $s\in S$ there is
\begin{enumerate}
  \item a field extension $k'\supseteq k$,
  \item an abelian group $A$, a monomorphism $D_{k'}(A)\to G_{k'}$,
  \item an fppf neighborhood $V\to S$ of $s$,
  \item a $D_V(A)$-cover $Y\to V$, and
  \item a $G$-equivariant isomorphism $Y\times^{D(A)}G\cong V\times_SX$.
\end{enumerate}
\end{prop}

The proof will be very much along the same lines as the proof of \cite[Proposition 3.5]{Biswas-Borne}.

\begin{proof}
If $x\in \X$ is a closed point then there is an \'etale neighborhood $U$ of $\pi(x)$ such that $\X_U:=\X\times_SU\simeq [Y/H]$ where $H=D(A)$ (where $A$ is an abstract abelian group) and $D(A)_x$ is the stabilizer at $x$. We replace $\X$ by $\X_U$. Then we have $\mathcal{G}_x\simeq BH_{\kappa(\pi(x))}$ and $BH\to \X$ is a section of the morphism $\X\to BH$
corresponding to the $H$-torsor $Y\to \X$. The torsor $X\to \X$
corresponds to a representable morphism $\X\to BG$. After base change to a field extension of $\kappa(\pi(x))$ we may assume that the composition
$BH\to \X\to BG$ is induced by a group morphism
$H\to G$ which must be a monomorphism since
$BH\to BG$
is representable.
The composition \[\X\xrightarrow{Y}BH\to\X\xrightarrow{X} BG\] is the torsor $Y\times^{H}G$ which we may compare with $X$.
The rest of the proof follows from \cite[Proposition 3.13]{Biswas-Borne}.
\end{proof}

If $G$ is a finite abelian group scheme over $k$, there is a sequence
\[0\to G^0\to G\to G_\et\to 0\] where $G_\et$ is the spectrum of the largest separable subalgebra of $k[G]$ (hence \'etale). Considering the corresponding sequence for the dual $D(G)$ we get
\[0\to D(G)^0\to D(G)\to D(G)_\et\to 0\] and if we dualize we get
\[0\to D_G:=D(D(G)_\et)\to D(D(G))\cong G\to D(D(G)^0)\to 0\,.\]
The subgroup $D_G$ is the maximal diagonalizable subgroup and has the universal property that every morphism $D(A)\to G$ ($A/k$ \'etale) factors uniquely through $D_G\to G$.

Passing to the perfect closure $k\subseteq k'$ (\cite[\href{https://stacks.math.columbia.edu/tag/046W}{Tag 046W}]{stacks-project}) the sequences above splits canonically (see e.g. \cite[6.8]{Waterhouse}). Hence we have the following results:

\begin{prop}\label{prop:insep-cover}
Let $k$ be a field and $S$ a $k$-scheme. Let $\pi\colon \X\to S$ be a stacky cover and $X\to \X$ a $G$-torsor where $X$ is a scheme and $G$ is a finite abelian group scheme over $k$. Then
\begin{enumerate}
  \item there exists a purely inseparable field extension $k\subseteq k'$, a $D_{G}$-cover $X'\to S_{k'}$, and a $G$-equivariant $S_{k'}$-isomorphism
  $X'\times^{D_{G}}G\cong X_{k'}\,,$ and
  \item $\stab(X)\subseteq D_G\times X$ is a closed subgroup scheme.
\end{enumerate}
\end{prop}

\begin{proof}
(1) By the argument above, after a purely inseparable field extension we may assume that $k$ is perfect and hence $G$ splits into two parts where one is the maximal diagonalizable subgroup $D_G$. Since $\X$ has diagonalizable stabilizers, we conclude that the non-diagonalizable part $G^\perp$ must act freely and hence we may quotient out by $G^\perp$.
 Thus $X\to S$ is induced from a $D_G$-cover.

 (2) From (1) we have that $\stab(X)\subseteq D_G\times X$ after a purely inseparable field extension. By faithfully flat descent we see that $\stab(X)\subseteq D_G\times X$ already over $k$.
\end{proof}


Proposition \ref{prop:insep-cover} motivates the following definition:

\begin{df}\label{df:ram-cov}
Let $G$ be a finite abelian group scheme over a field $k$ and let $S$ be a scheme over $k$. A \emph{tamely ramified $G$-cover} is a finite flat morphism $p\colon X\to S$ of finite presentation together with an action of $G$ on $X$ over $S$ such that
\begin{enumerate}
  \item $\stab(X)\subseteq D_G\times X$, where $D_G\subseteq G$ is the maximal diagonalizable subgroup, and
  \item there exists an fppf cover $\{U_i\to S\}$ and an isomorphism of $\oj_S[G]$-comodules \[(f_*\oj_X)|_{U_i}\cong \oj_{U_i}[G]\,,\] where the comodule structure on the right hand side is the regular representation.
\end{enumerate}
\end{df}



\begin{df}\label{df:ramification-datum}
A stacky building datum $(\A,\Li)$ on $S$ is called a \emph{birational building datum} if the image of $\Li$ is contained in the full subcategory of effective Cartier divisors, i.e., for every pair $(\Li_p,s_p)$ where $p\in P_\A(U)$ for some \'etale $U\to S$, we have that $s_p$ is a regular section.
\end{df}

\begin{rk}
Let $S$ be a reduced scheme. If $\pi\colon \X\to S$ is the root stack associated to a birational building datum, then the complement $j\colon U\subset \X$ of the ramification locus is \emph{scheme theoretically dense} (and topologically dense), that is, $\oj_\X\to j_*\oj_U$ is injective. In this cases $\pi|_U\colon \X_U\to U$ is an isomorphism and we say that $\X\to S$ is \emph{birational}.
\end{rk}

Let $p\colon X\to S$ be a tamely ramified $G$-cover which is generically a torsor. Then $[X/G]\to S$ is a birational stacky cover and hence a flat root stack corresponding to some birational building datum $(\A,\Li)$ where $\A\cong \Pic_{\X/S}$. This means that we may talk about the \emph{birational building datum associated to }$X$.

\begin{df}
We say that a tamely ramified $G$-cover $X\to S$ \emph{has a birational building datum} $(\A,\Li)$ if $[X/G]\cong S_{(\A,\Li)}$.
\end{df}

\begin{rk}
If $p\colon X\to S$ is a ramified $G$-cover with birational building datum $(\A,\Li)$ then $\A\cong \Pic_{[X/G]/S}$.
\end{rk}

\subsection*{Existence of ramified $G$-covers}

\begin{lem}\label{lem:inflex}
Let $S$ be a scheme proper over a field $k$. If $\pi\colon \X\to S$ is a birational stacky cover, then $\X$ is geometrically reduced and geometrically connected if and only if $S$ is geometrically reduced and geometrically connected.
\end{lem}

\begin{proof}
Since $\pi$ is a universal homeomorphism, $\X$ is geometrically connected if and only if $S$ is geometrically connected.
Furthermore, since $\pi$ is flat, we have that $\pi^{-1}(U)$ is scheme-theoretically dense in $\X$ and since $\pi|_U$ is an isomorphism we get that $\pi^{-1}(U)$ is reduced. The same argument work after any field extension. This completes the proof.
\end{proof}

\begin{df}[Borne--Vistoli]An object of a $k$-linear rigid tensor category $V$ is called
  \begin{itemize}
  \item \emph{finite} if there are two distinct polynomials $f,g\in \N[x]$ such that $f(V)\cong g(V)$ where we interpret sum as direct sum and power as tensor power, and
  \item \emph{essentially finite} if it is the kernel of a morphism of finite vector bundles.
  \end{itemize}
  \end{df}

The following is a direct consequence of \cite[Proposition 3.19]{Biswas-Borne} and Lemma \ref{lem:inflex}:

\begin{lem}
Let $S$ be a geometrically connected and geometrically reduced scheme proper over a field $k$. If $\X\to S$ is a birational stacky cover then there exists a finite abelian group scheme $G$ and a $G$-torsor $X\to \X$, where $X$ is a scheme, if and only if for any closed point $x$ in $\X$, every representation of the residual gerbe $\mathcal{G}_x$ is a quotient of a subbundle of the restriction of an essentially finite, basic vector bundle on $\X$ along $\mathcal{G}_x\to \X$.
\end{lem}





When $S$ is geometrically connected and geometrically reduced, \cite[Theorem 7.9]{Borne--Vistoli-Fund} and Lemma \ref{lem:inflex} implies that the category $\EFVect(\X)$ of essentially finite vector bundles on $\X$ is tannakian. When $V$ is an object of a tannakian category we write $\langle V\rangle$ for the tannakian subcategory generated by $V$, i.e., every object of $\langle V\rangle$ is a subquotient of $p(V,V^\vee)$ for some $p(x,y)\in \N[x,y]$.

\begin{df}[{\cite[D\'efinition 6.1]{Deligne-Le-Groupe}}]
Let $\C$ be a tannakian category. The \emph{fundamental group} $\pi(\C)$ of $\C$ is an ind Hopf algebra in $\C$ (considered as an object in Ind$(\C)^{\op}$) such that for any fiber functor $\omega\colon \C\to \Vect k'$ over an extension $k'\supseteq k$, we have \[\spec \omega(\pi(\C))\cong \Aut(\omega)\] in a functorial way. If $V\in \C$ then $\pi(V):=\pi(\langle V\rangle)$.
\end{df}

\begin{df}[{\cite[Definition A.12, A.14]{Biswas-Borne}}]
A tannakian category is called \emph{basic} if there is an affine group scheme $G$ over $k$ such that $G\otimes_k\mathbbm{1}\cong \pi(\C)$. We say that $V\in \C$ is basic if $\langle V\rangle$ is basic (see also \cite[A.15]{Biswas-Borne}).
\end{df}

Now we are ready to state the main theorem of this section.

\begin{thm}\label{thm:appl-main}
Let $S$ be a scheme proper over a field $k$ and assume that $S$ is geometrically connected and geometrically reduced. Let $(\A,\Li)$ be a birational building datum and $(P_\A,Q_\A,\Li)$ the associated Deligne--Faltings datum.
Then the following are equivalent:
\begin{enumerate}
  \item There exists a finite abelian group scheme $G$ over $k$ and a ramified $G$-cover $X\to S$ with birational building datum $(\A,\Li)$;
  \item For every geometric point $\bar{s}$ in the branch locus, we have that
  \begin{enumerate}
    \item[(i)] the map $\Gamma(S,\A)\to \A_{\bar{s}}$ is surjective, and
    \item[(ii)] for every $\lambda\in \A_{\bar{s}}$, there exists an essentially finite, basic, parabolic vector bundle $(E,\rho)$ on $(S,P_\A,Q_\A,\Li)$ such that the morphism
    \[\bigoplus_{\lambda'}E_{e_\lambda-e_{\lambda'}}|_{\bar{s}}\xrightarrow{(E(e_{\lambda'})|_{\bar{s}})_{\lambda'}} E_{e_\lambda}|_{\bar{s}}\] is not surjective, where the direct sum is over all $\lambda'\in \Gamma(S,\A)$ such that $\lambda'_{\bar{s}}\neq 0$.
  \end{enumerate}
\end{enumerate}
\end{thm}

This proof is analogous to that in \cite{Biswas-Borne}.

\begin{proof}
Let $\pi\colon \X=S_{(\A,\Li)}\to S$ be the associated root stack and $\Ei\colon \pi^*Q_{\A}\to \DIV_{\X_\et}$ be the universal Deligne--Faltings object. Any vector bundle on $\mathcal{G}_{\bar{s}}$ is a direct sum of $\Ei_\lambda|_{\mathcal{G}_{\bar{s}}}$ for $\lambda\in \A_{\bar{s}}$. By Proposition \cite[Proposition 3.18]{Biswas-Borne}, (1) holds if and only if every vector bundle on $\mathcal{G}_{\bar{s}}\cong B\A_{\bar{s}}$ is a quotient of a subbundle of $\F|_{\mathcal{G}_{\bar{s}}}$,
for some essentially finite vector bundle $\F$ on $\X$. This implies that (1) holds if and only if, for every point $\bar{s}$ and every $\lambda\in \A_{\bar{s}}$, there exists an essentially finite vector bundle $\F$ on $\X$, such that $\Hom_{\mathcal{G}_{\bar{s}}}(\F|_{\mathcal{G}_{\bar{s}}},\Ei_\lambda|_{\mathcal{G}_{\bar{s}}})\neq 0$.
Let $\F$ be an essentially finite vector bundle on $\X$. We have a commutative diagram
\[\begin{tikzcd}\mathcal{G}_{\bar{s}}\ar{r}{g}\ar{d}{\hat{\pi}} & \X\ar{d}{\pi} \\ \spec k\ar{r}{\bar{s}} & S\end{tikzcd}\]
and isomorphisms
\[\begin{split}\Hom_{\oj_{\mathcal{G}_{\bar{s}}}}(\F|_{\mathcal{G}_{\bar{s}}},\Ei_\lambda|_{\mathcal{G}_{\bar{s}}}) & \simeq \hat{\pi}_*\HOM_{\oj_{\mathcal{G}_{\bar{s}}}}(\F|_{\mathcal{G}_{\bar{s}}},\Ei_\lambda|_{\mathcal{G}_{\bar{s}}}) \\ & \simeq \hat{\pi}_*g^*(\F^\vee\otimes \Ei_\lambda)\,.\end{split}\]
Let $E_{\F^\vee}$ be the parabolic vector bundle corresponding to $\F^\vee$ under the equivalence of Theorem \ref{thm:parabolic}.

Consider the ideals $J_{\bar{s}}\subseteq \oj_S$ and $\J_{\bar{s}}\subseteq \oj_\X$ generated by the image of
\[\bigoplus_{\lambda,\lambda'}\Li^\vee_{\lambda',\lambda}\xrightarrow{\sum_{\lambda,\lambda'}s^\vee_{\lambda,\lambda'}}\oj_S\,,\]
and
\[\bigoplus_{\lambda}\Ei^\vee_{\lambda}\xrightarrow{\sum_{\lambda}\varepsilon^\vee_{\lambda}}\oj_\X\]
respectively, where the first direct sum is over all $\lambda,\lambda'\in \Gamma(S,\A)$ such that $\lambda'_{\bar{s}}\,,\lambda_{\bar{s}}\neq 0$ and the second direct sum is over all $\lambda\in \Gamma(S,\A)$ such that $\lambda_{\bar{s}}\neq 0$.
We have an exact sequence
\[\bigoplus_{\lambda}\Ei^\vee_{\lambda}\to \oj_\X\to h_*\oj_{V(\J_{\bar{s}})}\to 0\]
and a commutative diagram \[\begin{tikzcd}\mathcal{G}_{\bar{s}}\ar{r}{j}\ar{d}{\hat{\pi}} & V(\J_{\bar{s}})\ar{r}{h}\ar{d}{p} & \X\ar{d}{\pi} \\ \spec k\ar{r}{i} & V(J_{\bar{s}})\ar{r} & S\,,\end{tikzcd}\]
where $p$ is a good moduli space. Note that the left diagram is Cartesian. This implies that we have a natural isomorphism
\[i^*p_*\simeq \hat{\pi}_*j^*\,,\] since $p$ is a good moduli space.

Since $\pi_*$ is exact, $\bar{s}^*$ is right exact, and $\F^\vee\otimes \Ei_\lambda\cong \HOM_{\oj_S}(\F, \Ei_\lambda)$ is a vector bundle, we get an exact sequence \[\bigoplus_{\lambda'}\pi_*(\Ei^\vee_{\lambda'}\otimes\F^\vee\otimes \Ei_\lambda)|_{\bar{s}}\to E_{\F^\vee,\lambda}|_{\bar{s}}\to \Hom_{\oj_{\mathcal{G}_{\bar{s}}}}(\F|_{\mathcal{G}_{\bar{s}}},\Ei_\lambda|_{\mathcal{G}_{\bar{s}}})\to 0\,.\]
It remains to find a criterion for when the morphism $\bigoplus_{\lambda'}\pi_*(\Ei^\vee_{\lambda'}\otimes\F^\vee\otimes \Ei_\lambda)|_{\bar{s}}\to E_{\F^\vee,\lambda}|_{\bar{s}}$ is not surjective.
We have that $\pi_*(\Ei_{\lambda'}^\vee\otimes\F^\vee\otimes\Ei_\lambda)=E_{\F^\vee,{e_\lambda}-e_{\lambda'}}$
and hence $\bigoplus_{\lambda'}\pi_*(\Ei^\vee_{\lambda'}\otimes\F^\vee\otimes \Ei_\lambda)|_{\bar{s}}\to E_{\F^\vee,\lambda}|_{\bar{s}}$ is surjective if and only if \[\bigoplus_{\lambda'}E_{e_\lambda-e_{\lambda'}}|_{\bar{s}}\xrightarrow{\sum_{\lambda'}E(e_{\lambda'})|_{\bar{s}}} E_{e_\lambda}|_{\bar{s}}\] is surjective. This completes the proof.
\end{proof}

%% file: groups.tex
\section{Closed subgroups of groups of multiplicative type}\label{sec:groups}
In the following section we investigate the Cartier dual of a non-necessarily flat closed subgroup $H$ of a diagonalizable group $G=D(A)$. 

\begin{df}
If $K$ is a group scheme over $S$ the \emph{Cartier dual of K} is the functor \[D_S(K)=\HOM_{grp}(K,\GmS)\,.\]
Hence $D_S(K)$ is isomorphic to the functor \[T\mapsto \Gamma(K_T,\oj_{K_T})^{gr}=\{g\in \Gamma(K_T,\oj_{K_T}): g\mbox{ is group-like}\}\]
(see Remark \ref{rk:gp-like}).
We often ignore the subscript $S$ and write $D(K)=D_S(K)$.
We call a group homomorphism $\lambda\colon K\to \Gm_{m}$ a \emph{character}.
\end{df}

\begin{rk}
Note that $H\hookrightarrow G$ is a closed subgroup scheme if and only if
\begin{enumerate}
\item it is a closed subscheme and
\item $I=\ker(R[G]\to R[H])$ is a Hopf ideal, i.e., if $\varepsilon, \sigma, \Delta$ denotes the counit, coinverse and comultiplication respectively, we require that \[\begin{split}\varepsilon(I)&=0\,, \\ \sigma(I)&\subseteq I\,, \mbox{ and}\\ \Delta(I)&\subseteq I\otimes R[G]+R[G]\otimes I\,.\end{split}\]
\end{enumerate}
\end{rk}

\begin{rk}\label{rk:gp-like}
Recall that an element $g$ of a Hopf algebra is \emph{group-like} if $\Delta(g)=g\otimes g$ and $\varepsilon(g)=1$. We denote the subset of group-like elements of a Hopf algebra $E$ by $E^{gr}$. The set $E^{gr}$ forms a group under multiplication. Note that when $A$ is an abstract abelian group then $R[A]^{gr}=A$.
\end{rk}

\begin{rk}\label{rk:act-char}
Let $G$ be a group scheme over $S$ acting on a line bundle $\Li$ on $S$ (this also works when $S$ is a stack). This corresponds to a character $\lambda\colon G\to \AUT_{\oj_S}(\Li)\cong\Gm_{m,S}$ and we say that $G$ \emph{acts on $\Li$ via the character $\lambda$}.

To understand what happens, choose an \'etale covering $\{U_i\to S\}$ such that $\Li|_{U_i}$ is trivial for every $i$. The coaction
\[c_i\colon \oj_{U_i}\cong\Li|_{U_i}\to \Li_{U_i}\otimes \oj[G_{U_i}]\cong \oj_{U_i}\otimes \oj[G_{U_i}]\cong \oj[G_{U_i}]\] is completely determined by $c_i(1)$.
Using the axioms of a coaction we get that $\varepsilon(c_i(1))=1$ and $\Delta(c_i(1))=c_i(1)\otimes c_i(1)$. This means that $c_i(1)$ is group-like for every $i$. Since these $c_i$'s are restrictions of a global coaction
\[c\colon \Li\to \Li\otimes_{\oj_S}\oj[G]\]
we get that $c_i$ and $c_j$ agree over the intersection $U_i\times_SU_j$ and hence the $c_i(1)$'s glues to a group-like element in $\Gamma(G,\oj_G)$ corresponding to a character $\lambda\colon G\to \Gm_m$. It follows that the global coaction is given by $c(x)=x\otimes\lambda$ for every local section $x$ of $\Li$.
\end{rk}

\begin{ex}\label{ex:ram}
Let $(R,\m)$ be a local ring and $A$ a finite abelian group. We also let $\oj_X$ be an $R$-algebra with a coaction of $R[A]$ such that $X\to S$ is a ramified $D(A)$-cover, where $S=\spec R$ and $X=\spec \oj_X$.
Hence we may write $\oj_X\cong R[\{x_\lambda\}_{\lambda\in A}]/(\{x_\lambda x_{\lambda'}-s_{\lambda,\lambda'}x_{\lambda+\lambda'}\})$ where each $x_\lambda$ is a generator for the line bundle with character $\lambda$ (see Remark \ref{rk:lin-bun}).
Let $I$ be the Hopf ideal cutting out the stabilizer group $H=X\times_{X\times_SX}D_X(A)\hookrightarrow D_X(A)$ with respect to the action on $X$. Then $I$ is generated by $\{x_\lambda(\lambda-1)\}_{\lambda\in A}$.
\end{ex}


Now let $(R,\m)$ be a local ring, let $S=\spec R$, let $A$ be a finite abelian group, and let $H\hookrightarrow G:=D_S(A)$ be a closed subgroup cut out by a Hopf ideal $I$. 

\begin{lem}\label{lem:imp-lem}
If the order $d$ of $A$ is invertible in $S$, then
    \[
        R[H]^{gr}\cap (1+\m R[H])=\{1\}\,.
    \]
\end{lem}

\begin{proof}
Let $g$ be an element of $R[A]$ representing an element in $R[H]^{gr}\cap (1+\m R[H])$. The goal is to show that $g$ represents the element $1$ in $R[H]$, i.e., $g=1$ modulo $I$. Since $I\subseteq \ker \varepsilon$ and the ideal $\ker \varepsilon$ is generated by elements of the form $\lambda-1$, we get that $g$ must be of the form $g=1+\sum_{\lambda\in A\setminus\{0\}} a_\lambda(\lambda-1)\in R[A]$ where $a_\lambda\in \m$ for all $\lambda\in A\setminus\{0\}$.
Using that $\Delta(g)=g\otimes g$ modulo $I$ iteratively, we get that $g^d=1$ modulo $I$ and hence $(g-1)(1+g+g^2+\dots+g^{d-1})\in I$. But $1+g+g^2+\dots+g^{d-1}$ is a unit if $d$ is a unit and hence $g=1$ modulo $I$.
\end{proof}

\begin{prop}\label{prop:gp-like-iso}
If the order of $A$ is invertible in $S$,
then the quotient $\pi\colon R[H]\to R[H]\otimes_{R}k$ is an isomorphism on group-like elements.
\end{prop}

\begin{proof}
Injectivity follows immediately from Lemma \ref{lem:imp-lem}. To prove surjectivity, note that $R[H]\otimes_{R}k\cong k[A']$ where $A'$ is a quotient of $A$. We have a commutative diagram
  \[\begin{tikzcd}
    R[A]\ar{r}\ar{d} & R[H]\ar{d} \\
    k[A]\ar{r} & k[A']
  \end{tikzcd}\]
where all arrows are surjections. The map $R[A]\to k[A]$ is an isomorphism on group-like elements and $k[A]\to k[A']$ is a surjection on group-like elements. Thus $R[H]\to k[A']$ is a surjection on group-like elements and hence an isomorphism on group-like elements.
\end{proof}

\begin{rk}\label{rk:important-example}
If $|A|$ is not invertible, then Lemma \ref{lem:imp-lem} does not hold in general. Not even for ramified covers as in Example \ref{ex:ram}. For instance, consider the following example:
Let $S$ be the spectrum of $\oj_S=\FF_2[a,b,c]/J$ with $J=(a,b,c)^3+(ab+ac-bc)$. Then put $s = ab, t = ac$, and define $X$ to be the spectrum of 
    \[
        \oj_X=\oj_S[x_{10},x_{01},x_{11}]/(x_{10}^2-s, x_{01}^2-t, x_{11}^2-(s+t), x_{10}x_{01}-ax_{11}, x_{10}x_{11}-bx_{01}, x_{01}x_{11}-cx_{10})\,.
    \] 
This is a ramified $D(A)$-cover with $A=(\ZZ/2\ZZ)^2$, with weights $x_{10}\ :\ (1,0)$, $x_{01}\ :\ (0,1)$, $x_{11}\ :\ (1,1)$. The stabilizer group is the spectrum of $(\oj_X[T_{10},T_{01}]/(T_{10}^2-1, T_{01}^2-1))/I$ where 
    \[
        I=(x_{10}(T_{10}-1), x_{01}(T_{01}-1), x_{11}(T_{10}T_{01}-1))\,.
    \]
The element $g=1+t(T_{10}-1)$ is not equal to 1 (we used Macaulay 2 to check this) but group-like modulo $I$. Indeed, since 
    \[
        T_{\lambda}-1=(T_{\lambda-\lambda'}-1)(T_{\lambda'}-1)+(T_{\lambda-\lambda'}-1)+(T_{\lambda'}-1)\,,
    \] 
we have $x_{\lambda'}(T_\lambda-1)=x_{\lambda'}(T_{\lambda-\lambda'}-1)$ modulo $I$.
This implies that, modulo $I\otimes \oj_X[A]+\oj_X[A]\otimes I$, we have
    \[
        \begin{split}
            \Delta(g)-g\otimes g 
            & = \Delta(1+t(T_{10}-1))-(1+t(T_{10}-1))\otimes (1+t(T_{10}-1)) \\
            & = t(T_{10}-1)\otimes (T_{10}-1) \\
            & = (s+t)(T_{10}-1)\otimes (T_{10}-1) \\
            & = (s+t)(T_{01}-1)\otimes (T_{10}-1) \\
            & = t(T_{01}-1)\otimes (T_{10}-1)+(T_{01}-1)\otimes s(T_{10}-1) \\ 
            & = 0\,.
        \end{split}    
    \]
Hence $g$ is group-like. 
This says that the morphism $\oj_X[A]\to \oj_X[A]/I=\oj_X[H]$ is \emph{not surjective} on group-like elements. 
\end{rk}